\documentclass[3p,computermodern,11pt]{elsarticle}
\usepackage{algorithm,algorithmic}
\usepackage[titletoc,title]{appendix}
\usepackage{amsmath}
\usepackage{amsfonts}
\usepackage{mathtools}
\newcommand{\error}{\mathbf{e}}
\newcommand{\RHSbasisvec}{\mathbf{U}}
\newcommand{\sampvec}{\mathbf{P}}
\newcommand\norm[1]{\left\lVert#1\right\rVert_2}
\newcommand{\RR}[1]{\mathbb{R}^{#1}}
\newcommand{\RRC}[1]{\mathcal{R}^{#1}}
\newcommand{\PDEspace}{\mathbb{R}^N \times [0,T]}

\newcommand{\icspace}{\RRC{N}}
\newcommand{\icspacecoarse}{\tilde{\MC{H}}}
\newcommand{\icspacefine}{\MC{H}'}
\newcommand{\RHS}{\mathbf{R}}
\newcommand{\errorcoarse}{\mathbf{\tilde{e}}}
\newcommand{\errorfine}{\mathbf{e'}}
\newcommand{\trieq}{\vcentcolon=}
\newcommand{\vdummy}{\mathbf{q}}
\newcommand{\errorcoarsegen}{\mathbf{\tilde{a}^e}}

\newcommand{\statedummy}{\mathbf{y}}
\newcommand{\defeq}{\trieq}
\newcommand{\timeSpace}{\MC{T}}

\usepackage{mdframed}
\usepackage{float}
\usepackage{wrapfig,blindtext}
\usepackage{graphicx} 
\usepackage{epstopdf} 
\usepackage{pslatex} 
\usepackage{amsmath,amssymb}
\usepackage{caption}
\usepackage{subcaption}
\usepackage{amsfonts,amsthm} 
\usepackage[english]{babel} 
\usepackage[dvipsnames]{xcolor}
\usepackage{wasysym}
\usepackage{pdfpages}
\usepackage{float}
\usepackage{comment}
\usepackage{soul}
\newtheorem{theorem}{Theorem}
\newtheorem{corollary}{Corollary}[theorem]
\usepackage{nomencl}

\usepackage[font=small,labelfont=bf,
   justification=justified,format=plain]{caption}

\usepackage{comment}

\newcommand{\MC}{\mathcal}

\newcommand{\afinevec}{{\mathbf{a}^{\prime}}}

\newcommand{\afullvec}{\mathbf{a}}
\newcommand{\acoarsevec}{\tilde{\mathbf{a}}}

\newcommand{\xfullvec}{\mathbf{a_0}}
\newcommand{\xcoarsevec}{\tilde{\mathbf{a}}_0}
\newcommand{\xfinevec}{{\mathbf{a}_0^{\prime}}}

\newcommand{\Wfullvec}{{\mathbf{W}}}
\newcommand{\Wfinevec}{{\mathbf{W}^{\prime}}}
\newcommand{\Wcoarsevec}{\tilde{\mathbf{W}}}
\newcommand{\Pifine}{{\Pi^{\prime}}}
\newcommand{\Picoarse}{\tilde{\Pi}}
\newcommand{\Vfine}{{\MC{V}^{\prime}}}
\newcommand{\Vcoarse}{\tilde{\MC{V}}}
\newcommand{\Wfine}{{\MC{W}^{\prime}}}
\newcommand{\Wcoarse}{\tilde{\MC{W}}}
\newcommand{\Vfull}{{\MC{V}}}

\newcommand{\ufullvec}{{\mathbf{u}}}
\newcommand{\ufinevec}{{\mathbf{u}^{\prime}}}
\newcommand{\ucoarsevec}{\tilde{\mathbf{u}}}
\newcommand{\vfinevec}{{\mathbf{V}^{\prime}}}
\newcommand{\vcoarsevec}{\tilde{\mathbf{V}}}

\newcommand{\vfullvec}{{\mathbf{V}}}

\usepackage{nomencl}
  \makenomenclature

\journal{Computer Methods in Applied Mechanics and Engineering}

\begin{document}
\topmargin -1.5cm
\textheight 23cm
\begin{frontmatter}

\title{The Adjoint Petrov--Galerkin Method for Non-Linear Model Reduction}

\author[b]{Eric J. Parish}
\ead{ejparis@sandia.gov}
\author[a]{Christopher Wentland}
\ead{chriswen@umich.edu}
\author[a]{Karthik Duraisamy}
\ead{kdur@umich.edu}

\address[a]{Department of Aerospace Engineering, University of Michigan, Ann Arbor, MI}
\address[b]{Sandia National Laboratories,  Livermore, CA}


\begin{abstract}
We formulate a new projection-based reduced-ordered modeling technique for non-linear dynamical systems. The proposed technique, which we refer to as the Adjoint Petrov--Galerkin (APG) method, is derived by decomposing the generalized coordinates of a dynamical system into a resolved coarse-scale set and an unresolved fine-scale set. A Markovian finite memory assumption within the Mori-Zwanzig formalism is then used to develop a reduced-order representation of the coarse-scales. This procedure leads to a closed reduced-order model that displays commonalities with the adjoint stabilization method used in  finite elements. The formulation is shown to be equivalent to a Petrov--Galerkin method with a non-linear, time-varying test basis, thus sharing some similarities with the Least-Squares Petrov--Galerkin method. 
Theoretical analysis examining \textit{a priori} error bounds and computational cost is presented. Numerical experiments on the compressible Navier-Stokes equations demonstrate that the proposed method can lead to improvements in numerical accuracy, robustness, and computational efficiency over the Galerkin method on problems of practical interest. Improvements in numerical accuracy and computational efficiency over the Least-Squares Petrov--Galerkin method are observed in most cases.
\end{abstract}

\end{frontmatter}

\section{Introduction}
High-fidelity numerical simulations play a critical role in modern-day engineering and scientific investigations. The computational cost of high-fidelity or full-order models (FOMs) is, however, often prohibitively expensive. This limitation has led to the emergence of reduced-order modeling techniques. Reduced-order models (ROMs) are formulated to \textit{approximate} solutions to a FOM on a low-dimensional manifold. Common reduced-order modeling techniques include balanced truncation~\cite{balanced_truncation_moore,balanced_truncation_roberts}, Krylov subspace techniques~\cite{krylov_rom}, reduced-basis methods~\cite{Hesthaven2016}, and the proper orthogonal decomposition approach~\cite{chatterjee_pod_intro}.
Reduced-order models based on such techniques have been implemented in a wide variety of disciplines and have been effective in reducing the computational cost associated with high-fidelity numerical simulations~\cite{kerschen_mech_pod,padhi_neural_net_pod,cao_meteorology_pod}.

Projection-based reduced-order models constructed from proper orthogonal decomposition (POD) have proved to be an effective tool for model order reduction of complex systems. In the POD-ROM approach, snapshots from a high-fidelity simulation (or experiment) are used to construct an orthonormal basis spanning the solution space. A small, truncated set of these basis vectors forms the \emph{trial} basis.
The POD-ROM then seeks a solution within the range of the trial basis via projection. Galerkin projection, in which the FOM equations are projected onto the same trial subspace, is the simplest type of projection. The Galerkin ROM (G ROM) has been used successfully in a variety of problems. When applied to general non-self-adjoint and non-linear problems, however, theoretical analysis and numerical experiments have shown that Galerkin ROM lacks \textit{a priori} guarantees of stability, accuracy, and convergence~\cite{rowley_pod_energyproj}. This last issue is particularly challenging as it demonstrates that enriching a ROM basis does not necessarily improve the solution~\cite{huang_combustion_roms}. The development of stable and accurate reduced-order modeling techniques for complex non-linear systems is the motivation for the current work.
\begin{comment}
Research examining the stability and accuracy of ROMs is typically approached from either a stabilization viewpoint or from a closure modeling viewpoint. 
\end{comment}

A significant body of research aimed at producing accurate and stable ROMs for complex non-linear problems exists in the literature. These efforts include, but are not limited to, ``energy-based" inner products~\cite{rowley_pod_energyproj,Kalashnikova_sand2014}, symmetry transformations~\cite{sirovich_symmetry_trans}, basis adaptation~\cite{carlberg_hadaptation,adeim_peherstorfer}, $L^1$-norm minimization~\cite{l1}, projection subspace rotations~\cite{basis_rotation}, and least-squares residual minimization approaches~\cite{bui_resmin_steady,bui_unsteady,rovas_thesis,carlberg_thesis,bui_thesis,carlberg_lspg,carlberg_lspg_v_galerkin,carlberg_gnat}. The Least-Squares Petrov--Galerkin (LSPG)~\cite{carlberg_lspg} method comprises a particularly popular least-squares residual minimization approach and has been proven to be an effective tool for non-linear model reduction. Defined at the fully-discrete level (i.e., after spatial and temporal discretization), LSPG relies on least-squares minimization of the FOM residual at each time-step. While the method lacks \textit{a priori} stability guarantees for general non-linear systems, it has been shown to be effective for complex problems of interest~\cite{carlberg_gnat, carlberg_lspg_v_galerkin, huang_scitech19}. Additionally, as it is formulated as a minimization problem, physical constraints such as conservation can be naturally incorporated into the ROM formulation~\cite{carlberg_conservative_rom}. At the fully-discrete level, LSPG is sensitive to both the time integration scheme as well as the time-step. For example, in Ref.~\cite{carlberg_lspg_v_galerkin} it was shown that LSPG produces optimal results at an intermediate time-step. Another example of this sensitivity is that, when applied to explicit time integration schemes, the LSPG approach reverts to a Galerkin approach. This limits the scope of LSPG to implicit time integration schemes, which can in turn increase the cost of the ROM~\cite{carlberg_lspg_v_galerkin}\footnote{It is possible to use LSPG with an explicit time integrator by formulating the ROM for an implicit time integration scheme, and then time integrating the resulting system with an explicit integrator.}. This is particularly relevant in the case where the optimal time-step of LSPG is small, thus requiring many time-steps of an implicit solver. Despite these challenges, the LSPG approach is arguably the most robust  technique that is used for ROMs of non-linear dynamical systems.

A second school of thought addresses stability and accuracy of ROMs from a closure modeling viewpoint. This follows from the idea that instabilities and inaccuracies in ROMs can, for the most part, be attributed to the truncated  modes. While these truncated  modes may not contain a significant portion of the system energy, they can play a significant role in the dynamics of the ROM~\cite{Wang_ROM_thesis}.  This is analogous to the closure problem encountered in large eddy simulation. Research has examined the construction of mixing length~\cite{aubry_mixlength_pod}, Smagorinsky-type~\cite{Wang_ROM_thesis,Ullmann_smag,wang_smag,smag_ROM}, and variational multiscale (VMS) closures~\cite{Wang_ROM_thesis,san_iliescu_geostrophic,Bergmann_pod_vms,Stabile2019} for POD-ROMs. The VMS approach is of particular relevance to this work. Originally developed in the context of finite element methods, VMS is a formalism to derive stabilization/closure schemes for numerical simulations of multiscale problems. The VMS procedure is centered around a sum decomposition of the solution $u$ in terms of resolved/coarse-scales $\tilde{u}$ and unresolved/fine-scales ${u}^{\prime}$. The impact of the fine-scales on the evolution of the coarse-scales is then accounted for by devising an approximation to the fine-scales. This approximation is often referred to as a ``subgrid-scale'' or ``closure'' model. 

Research has examined the application of both phenomenological and residual-based subgrid-scale models to POD-ROMs. In Refs.~\cite{san_iliescu_geostrophic,iliescu_pod_eddyviscosity,iliescu_vms_pod_ns}, Iliescu and co-workers examine the construction of eddy-viscosity-based ROM closures via the VMS method. These eddy-viscosity methods are directly analogous to the eddy-viscocity philosophy used in turbulence modeling. While they do not guarantee stability \textit{a priori}, these ROMs have been shown to enhance accuracy on a variety of problems in fluid dynamics. However, as eddy-viscosity methods are based on phenomenological assumptions specific to three-dimensional turbulent flows, their scope may be limited to specific types of problems. Residual-based methods, which can also be derived from VMS, constitute a more general modeling strategy. The subgrid-scale model emerging from a residual-based method typically appears as a term that is proportional to the residual of the full-order model; if the governing equations are exactly satisfied by the ROM, then the model is inactive. While residual-based methods in ROMs are not as well-developed as they are in finite element methods, they have been explored in several contexts. In Ref.~\cite{Bergmann_pod_vms}, ROMs of the Navier-Stokes equations are stabilized using residual-based methods. This stabilization is performed by solving a ROM stabilized with a method such as streamline upwind Petrov--Galerkin (SUPG) and augmenting the POD basis with additional modes computed from the residual of the Navier-Stokes equations. In Ref.~\cite{iliescu_ciazzo_residual_rom}, residual-based stabilization is developed for velocity-pressure ROMs of the incompressible Navier-Stokes equations. Both eddy-viscosity and residual-based methods have been shown to improve ROM stability and performance. The majority of existing work on residual-based stabilization (and eddy-viscosity methods) is focused on ROMs formulated from continuous projection (i.e., projecting a continuous PDE using a continuous basis). In this instance, the ROM residual is defined at the continuous level and is directly linked to the governing partial differential equation. In many applications (arguably the majority~\cite{Kalashnikova_sand2014}), however, the ROM is constructed through discrete projection (i.e., projecting the spatially discretized PDE using a discrete basis). In this instance, the ROM residual is defined at the semi-discrete level and is tied to the \textit{spatially discretized} governing equations. Residual-based methods for ROMs developed through discrete projections have, to the best of the authors' knowledge, not been investigated.
 
Another approach that displays similarities to the variational multiscale method is the Mori-Zwanzig (MZ) formalism. Originally developed by Mori~\cite{MoriTransport} and Zwanzig~\cite{ZwanzigLangevin} and reformulated by Chorin and co-workers~\cite{ChorinOptimalPrediction,ChorinOptimalPredictionMemory,Chorin_book,ProblemReduction}, the MZ formalism is a type of model order reduction framework. The framework consists of decomposing the state variables in a dynamical system into a resolved (coarse-scale) set and an unresolved (fine-scale) set. An exact reduced-order model for the resolved scales is then derived in which the impact of the unresolved scales on the resolved scales appears as a memory term. This memory term depends on the temporal history of the resolved variables. In practice, the evaluation of this memory term is not tractable. It does, however, serve as a starting point to develop closure models. As MZ is formulated systematically in a dynamical system setting, it promises to be an effective technique for developing stable and accurate ROMs of non-linear dynamical systems. A range of research examining the MZ formalism as a multiscale modeling tool exists in the community. Most notably, Stinis and co-workers~\cite{stinisEuler,stinisHighOrderEuler,Stinis-rMZ,stinis_finitememory,PriceMZ,PriceMZ2} have developed several models for approximating the memory, including finite memory and renormalized models, and examined their application to the semi-discrete systems emerging from Fourier-Galerkin and Polynomial Chaos Expansions of Burgers' equation and the Euler equations. Application of MZ-based techniques to the classic POD-ROM approach has not been undertaken.

This manuscript leverages work that the authors have performed on the use of the MZ formalism to develop closure models of partial differential equations~\cite{parishAIAA2016,parishMZ1,parish_dtau,GouasmiMZ1,parishVMS}. In addition to focusing on the development and analysis of MZ models, the authors have examined the formulation of the MZ formalism within the context of the VMS method~\cite{parishVMS}. By expressing MZ models within a VMS framework, similarities were discovered between MZ and VMS models. In particular, it was discovered that several existing MZ models are residual-based methods. 

 The contributions of this work include:
\begin{enumerate}
\item The development of a novel projection-based reduced-order modeling technique, termed the Adjoint Petrov--Galerkin (APG) method. The method leads to a ROM equation that is driven by the residual of the discretized governing equations. The approach is equivalent to a Petrov--Galerkin ROM and displays similarities to the LSPG approach.  The method can be evolved in time with explicit integrators (in contrast to LSPG). This potentially lowers the cost of the ROM.

\item Theoretical error analysis examining conditions under which the \textit{a priori} error bounds in APG may be smaller than in the Galerkin method.

\item Computational cost analysis (in FLOPS) of the proposed APG method as compared to the Galerkin and LSPG methods. This analysis shows that the APG ROM is twice as expensive as the G ROM for a given time step, for both explicit and implicit time integrators. In the implicit case, the ability of the APG ROM to make use of Jacobian-Free Newton-Krylov methods suggests that it may be more efficient than the LSPG ROM.

\item Numerical evidence on ROMs of compressible flow problems demonstrating that the proposed method is more accurate and stable than the G ROM on problems of interest. Improvements over the LSPG ROM are observed in most cases. An analysis of the computational cost shows that the APG method can lead to lower errors than the LSPG and G ROMs for the same computational cost.

\item Theoretical results and numerical evidence that provides a relationship between the time-scale in the APG ROM and the spectral radius of the right-hand side Jacobian. Numerical evidence suggests that this relationship also applies to the selection of the optimal time-step in LSPG.

\end{enumerate}


The structure of this paper is as follows: Section~\ref{sec:FOM}  outlines the full-order model of interest and its formulation in generalized coordinates. Section~\ref{sec:ROM}  outlines the reduced-order modeling approach applied at the semi-discrete level. Galerkin, Petrov--Galerkin, and VMS ROMs will be discussed. Section~\ref{sec:MZ} details the Mori-Zwanzig formalism and the construction of the Adjoint Petrov--Galerkin ROM. Section~\ref{sec:analysis} provides theoretical error analysis. Section~\ref{sec:cost} discusses the implementation and computational cost of the Adjoint Petrov--Galerkin method. Numerical results and comparisons with Galerkin and LSPG ROMs are presented in Section~\ref{sec:numerical}. Conclusions are provided in Section~\ref{sec:conclude}.

Mathematical notation in this manuscript is as follows: matrices are written as bold uppercase letters (e.g. $\vfullvec$), vectors as lowercase bold letters (e.g. $\ufullvec$), and scalars as italicized lowercase letters (e.g. $a_i$). Calligraphic script may denote vector spaces or special operators (e.g. $\MC{V}$, $\MC{L}$). Bold letters followed by parentheses indicate a matrix or vector function (e.g. $\mathbf{R}(\cdot)$, $\ufullvec (\cdot)$), and those followed by brackets indicate a linearization about the bracketed argument (e.g. $\mathbf{J}[\cdot]$).

\section{Full-Order Model and Generalized Coordinates}\label{sec:FOM}
Consider a full-order model that is described by the dynamical system,
\begin{equation}\label{eq:FOM}
\frac{d }{dt}\ufullvec(t) = \mathbf{R}(\ufullvec(t)), \qquad \ufullvec(0) = \ufullvec_0, \qquad t \in [0,T], 
\end{equation}
where $T \in \mathbb{R}^+$ denotes the final time, $\ufullvec : [0,T] \rightarrow \RR{N}$ denotes the state, and $\ufullvec_0 \in \mathbb{R}^N$ the initial conditions. The function $\mathbf{R}: \mathbb{R}^N  \rightarrow \mathbb{R}^N$ with $\statedummy \mapsto \RHS(\statedummy)$ is a (possibly non-linear) function and will be referred to as the ``right-hand side" operator. Equation~\ref{eq:FOM} arises in many disciplines, including the numerical discretization of partial differential equations. In this context, $\mathbf{R}(\cdot)$ may represent a spatial discretization scheme with source terms and applicable boundary conditions.  

In many practical applications, the computational cost associated with solving Eq.~\ref{eq:FOM} is prohibitively expensive due to the high dimension of the state. The goal of a ROM is to transform the $N$-dimensional dynamical system presented in Eq.~\ref{eq:FOM} into a $K$ dimensional dynamical system, with $K \ll N$. To achieve this goal, we pursue the following agenda:
\begin{enumerate}
\item Develop a weak form of the FOM in generalized coordinates.
\item Decompose the generalized coordinates into a $K$-dimensional resolved coarse-scale set and an $N-K$ dimensional unresolved fine-scale set.
\item Develop a $K$-dimensional ROM for the coarse-scales by making approximations to the fine-scale coordinates.
\end{enumerate}
The remainder of this section will address task 1 in the above agenda.


To develop the weak form of Eq.~\ref{eq:FOM}, we start by defining a trial basis matrix comprising $N$ orthonormal basis vectors,
\begin{equation*}
\mathbf{V} \equiv \begin{bmatrix}
\mathbf{v}_1 & \mathbf{v}_2 & \cdots & \mathbf{v}_N
\end{bmatrix},
\end{equation*}
where $\mathbf{v}_i \in \mathbb{R}^N$, $\mathbf{v}_i^T \mathbf{v}_j = \delta_{ij}$.
The basis vectors may be generated, for example, by the POD approach. 
We next define the \textit{trial space} as the range of the trial basis matrix,
$$\Vfull \defeq \text{Range}(\vfullvec).$$
As $\mathbf{V}$ is a full-rank $N \times N$ matrix, $\MC{V} \equiv \RR{N}$ and the state variable can be exactly described by a linear combination of these basis vectors,
\begin{equation}\label{eq:genCord}
\ufullvec(t) = \sum_{i=1}^N \mathbf{v}_i a_i(t).  
\end{equation}
Following~\cite{carlberg_lspg}, we collect the basis coefficients $a_i(t)$ into $\afullvec : [0,T] \rightarrow \RR{N}$ and refer to $\afullvec$ as the \emph{generalized coordinates}. We similarly define the test basis matrix, $\Wfullvec$, whose columns comprise linearly independent basis vectors that span the \textit{test space}, $\MC{W}$,
\begin{equation*}
\Wfullvec \equiv \begin{bmatrix}
\mathbf{w}_1 & \mathbf{w}_2 & \cdots& \mathbf{w}_N
\end{bmatrix}, \qquad 
\MC{W} \defeq \text{Range}(\Wfullvec) ,
\end{equation*}
with $\mathbf{w}_i \in \mathbb{R}^N$. 

Equation~\ref{eq:FOM} can be expressed in terms of the generalized coordinates by inserting Eq.~\ref{eq:genCord} into Eq.~\ref{eq:FOM}, 
\begin{equation}\label{eq:FOM2}
\vfullvec \frac{d }{dt}\afullvec(t)= \mathbf{R}(\vfullvec \afullvec(t)).
\end{equation}
The weak form of Eq.~\ref{eq:FOM2} is obtained by taking the $L^2$ inner product with $\Wfullvec$,\footnote{The authors recognize that many types of inner products are possible in formulating a ROM. To avoid unnecessary abstraction, we focus here on the simplest case.}
\begin{equation}\label{eq:FOM3}
\Wfullvec^T  \vfullvec \frac{d }{dt}\afullvec(t)  = \Wfullvec^T  \mathbf{R}(\vfullvec \afullvec(t)).
\end{equation}
Manipulation of Eq.~\ref{eq:FOM3} yields the following dynamical system,
\begin{equation}\label{eq:FOM_generalized}
\frac{d }{dt}\afullvec(t) =  [\Wfullvec^T \vfullvec ]^{-1} \Wfullvec^T \mathbf{R}(\vfullvec \afullvec(t)), \qquad \afullvec(t=0) = \afullvec_0, \qquad t \in [0,T],
\end{equation}
where $\afullvec_0 \in \RR{N}$ with $\afullvec_0 =  [\Wfullvec^T \vfullvec ]^{-1} \Wfullvec^T \ufullvec_0.$ 
Note that Eq.~\ref{eq:FOM_generalized} is an $N$-dimensional ODE system and is simply Eq.~\ref{eq:FOM} expressed in a different coordinate system. It is further worth noting that, since $\Wfullvec$ and $\vfullvec$ are invertible (both are square matrices with linearly independent columns), one has $[\Wfullvec^T \vfullvec ]^{-1} \Wfullvec^T = \vfullvec^{-1}.$ This will not be the case for ROMs.

\section{Reduced-Order Models}\label{sec:ROM}
\subsection{Multiscale Formulation}
This subsection addresses task 2 in the aforementioned agenda. Reduced-order models seek a low-dimensional representation of the original high-fidelity model. To achieve this, we examine a multiscale formulation of Eq.~\ref{eq:FOM_generalized}. Consider sum decompositions of the trial and test space,
\begin{equation}
\Vfull = \Vcoarse \oplus \Vfine, \qquad \MC{W} = \MC{\tilde{W}} \oplus \MC{W}'.
\end{equation}
The space $\Vcoarse$ is referred to as the coarse-scale trial space, while $\Vfine$ is referred to as the fine-scale trial space. We refer to $\Wcoarse$ and $\Wfine$ in a similar fashion.
For simplicity, define $\Vcoarse$ to be the column space of the first $K$ basis vectors in $\vfullvec$ and $\Vfine$ to be the column space of the last $N-K$ basis vectors in $\vfullvec$. This approach is appropriate when the basis vectors are ordered in a hierarchical manner, as is the case with POD. Note the following properties of the decomposition:
\begin{enumerate}
\item The coarse-scale space is a subspace of $\MC{V}$, i.e., $\Vcoarse \subset \Vfull$.
\item The fine-scale space is a subspace of $\MC{V}$, i.e., $\Vfine \subset \Vfull$.
\item The fine and coarse-scale subspaces do not overlap, i.e.,  $\Vcoarse \cap \Vfine = \{ 0 \}.$
\item The fine-scale and coarse-scale subspaces are orthogonal, i.e., $\Vcoarse \perp \Vfine.$ This is due to the fact that the basis vectors that comprise $\mathbf{V}$ are orthonormal.
\end{enumerate}
For notational purposes, we make the following definitions for the trial and test spaces:
\begin{equation*}
\vfullvec \equiv \begin{bmatrix} \vcoarsevec & ; & \vfinevec \end{bmatrix}, \quad \Wfullvec \equiv \begin{bmatrix} \Wcoarsevec & ; & \Wfinevec \end{bmatrix}, 
\end{equation*}
where $[ \cdot \hspace{0.05 in}; \hspace{0.05 in} \cdot ]$ denotes the concatenation of two matrices and,
\begin{alignat*}{2}
&\vcoarsevec \equiv \begin{bmatrix}
\mathbf{v}_1 & \mathbf{v}_2 & \cdots & \mathbf{v}_K
\end{bmatrix},  && \quad \Vcoarse \defeq \text{Range}(\vcoarsevec) , \\
&\vfinevec  \equiv \begin{bmatrix}
\mathbf{v}_{K+1} & \mathbf{v}_{K+2} & \cdots &  \mathbf{v}_{N}
\end{bmatrix}, && \quad  \Vfine \defeq \text{Range}(\vfinevec). \\
&\Wcoarsevec \equiv \begin{bmatrix}
\mathbf{w}_1 & \mathbf{w}_2 & \cdots & \mathbf{w}_K
\end{bmatrix},  && \quad \Wcoarse \defeq \text{Range}(\Wcoarsevec) , \\
&\Wfinevec  \equiv \begin{bmatrix}
\mathbf{w}_{K+1} & \mathbf{w}_{K+2} & \cdots & \mathbf{w}_{N}
\end{bmatrix}, && \quad \Wfine \defeq \text{Range}(\Wfinevec) .
\end{alignat*}
The coarse and fine-scale states are defined as,
\begin{equation*}
\ucoarsevec(t) \defeq \sum_{i=1}^K \mathbf{v}_i a_i(t) \equiv \vcoarsevec \acoarsevec(t), \qquad \ufinevec(t) \defeq \sum_{i = K+1}^N \mathbf{v}_i a_i(t) \equiv \vfinevec \afinevec(t),
\end{equation*}
with $\ucoarsevec : [0,T] \rightarrow \tilde{\MC{V}}$, $\ufinevec : [0,T] \rightarrow \MC{V}', \acoarsevec: [0,T] \rightarrow \RR{K},$ and $\afinevec : [0,T] \rightarrow \RR{N-K}.$
\begin{comment}
We make similar definitions for the test space,
\begin{equation*}
\Wfullvec \trieq \begin{bmatrix} \Wcoarsevec & ; & \Wfinevec \end{bmatrix},   
\end{equation*}
where,
\begin{alignat*}{2}
&\Wcoarsevec \trieq \begin{bmatrix}
\mathbf{w}_1, \mathbf{w}_2, \hdots, \mathbf{w}_K
\end{bmatrix},  && \quad \text{Range}(\Wcoarsevec) \trieq \Wcoarse, \\
&\Wfinevec  \trieq \begin{bmatrix}
\mathbf{w}_{K+1}, \mathbf{w}_{K+2}, \hdots, \mathbf{w}_{N}
\end{bmatrix}, && \quad \text{Range}(\Wfinevec)\trieq \Wfine.
\end{alignat*}
\end{comment}
These decompositions allow Eq.~\ref{eq:FOM3} to be expressed as two linearly independent systems,
\begin{equation}\label{eq:FOM_VMS_coarse}
\Wcoarsevec^T \vcoarsevec \frac{d }{dt}\acoarsevec(t) + \Wcoarsevec^T \vfinevec \frac{d }{dt}\afinevec(t) =\Wcoarsevec^T  \mathbf{R}(\vcoarsevec \acoarsevec(t) + \vfinevec \afinevec(t)),
\end{equation}
\begin{equation}\label{eq:FOM_VMS_fine}
\Wfinevec^T \vcoarsevec \frac{d }{dt} \acoarsevec(t) + \Wfinevec^T \vfinevec \frac{d }{dt}\afinevec(t) =\Wfinevec^T  \mathbf{R}(\vcoarsevec \acoarsevec(t) + \vfinevec \afinevec(t)).
\end{equation}
Equation~\ref{eq:FOM_VMS_coarse} is referred to as the coarse-scale equation, while Eq.~\ref{eq:FOM_VMS_fine} is referred to as the fine-scale equation. It is important to emphasize that the system formed by Eqs.~\ref{eq:FOM_VMS_coarse} and~\ref{eq:FOM_VMS_fine} is still an exact representation of the original FOM.

The objective of ROMs is to solve the coarse-scale equation. The challenge encountered in this objective is that the evolution of the coarse-scales depends on the fine-scales. This is a type of ``closure problem" and must be addressed to develop a closed ROM.

\subsection{Reduced-Order Models}
As noted above, the objective of a ROM is to solve the (unclosed) coarse-scale equation. We now develop ROMs of Eq.~\ref{eq:FOM} by leveraging the multiscale decomposition presented above. This section addresses task 3 in the mathematical agenda.

 The most straightforward technique to develop a ROM is to make the approximation,
\begin{equation*}\label{eq:ufine_ansatz}
\ufinevec \approx \mathbf{0}.
\end{equation*}
This allows for the coarse-scale equation to be expressed as,
\begin{equation}\label{eq:FOM_VMS_coarse_2}
\Wcoarsevec^T \vcoarsevec \frac{d }{dt}\acoarsevec(t) =\Wcoarsevec^T  \mathbf{R}(\vcoarsevec \acoarsevec(t) ).
\end{equation}
Equation~\ref{eq:FOM_VMS_coarse_2} forms a $K$-dimensional reduced-order system (with $K \ll N$) and provides the starting point for formulating several standard ROM techniques. The Galerkin and Least-Squares Petrov--Galerkin ROMs are outlined in the subsequent subsections.

\subsubsection{The Galerkin Reduced-Order Model}
Galerkin projection is a common choice for producing a reduced set of ODEs. In Galerkin projection, the test basis is taken to be equivalent to the trial basis, i.e.  $\Wcoarsevec = \vcoarsevec$. The Galerkin ROM is then,
\begin{equation}\label{eq:GROM}
\vcoarsevec^T \frac{d }{dt}\ucoarsevec(t) = \vcoarsevec^T \mathbf{R}(\ucoarsevec(t)), \qquad \ucoarsevec(0) = \ucoarsevec_0, \qquad t \in [0,T].
\end{equation}
Galerkin projection can be shown to be optimal in the sense that it minimizes the $L^2$-norm of the FOM ODE residual over $\text{Range}(\vcoarsevec)$~\cite{carlberg_lspg}. As the columns of $\vcoarsevec$ no longer spans $\Vfull$, it is possible that the initial state of the full system, $\ufullvec_0$, may differ from the initial state of the reduced system, $\ucoarsevec_0$. For simplicity, however, it is assumed here that the initial conditions lie fully in the coarse-scale trial space, i.e. $\ufullvec_0 \in \Vcoarse$ such that,
\begin{equation}\label{eq:ROM_IC}
\ucoarsevec_0 = \ufullvec_0.
\end{equation}
 Note that this issue can be formally addressed by using an affine trial space to ensure that $\ucoarsevec_0 = \ufullvec_0$.

Equation~\ref{eq:GROM} can be equivalently written for the generalized coarse-scale coordinates, $\acoarsevec$,
\begin{equation}\label{eq:GROM_modal}
\frac{d }{dt}\acoarsevec(t) = \vcoarsevec^T \mathbf{R}(\vcoarsevec \acoarsevec(t)), \qquad \acoarsevec(0) = \acoarsevec_0, \qquad t \in [0,T],
\end{equation}
where $\acoarsevec_0 \in \RR{K}$ with  $\acoarsevec_0 =\vcoarsevec^T \ufullvec_0.$
Equation~\ref{eq:GROM_modal} is a $K$-dimensional ODE system (with $K\ll N$) and is hence of lower dimension than the FOM.  Note that, similar to Eq.~\ref{eq:FOM_generalized}, the projection via $\vcoarsevec^T$ would normally be $\big[ \vcoarsevec^T \vcoarsevec \big]^{-1} \vcoarsevec$. When $\vcoarsevec$ is constructed via POD, its columns are orthonormal and $\vcoarsevec^T \vcoarsevec = \mathbf{I}$; the projector has been simplified to reflect this. Non-orthonormal basis vectors will require the full computation of $\big[ \vcoarsevec^T \vcoarsevec \big]^{-1} \vcoarsevec$.

It is important to note that, in order to develop a computationally efficient ROM, some ``hyper-reduction'' method must be devised to reduce the cost associated with evaluating the matrix-vector product, $\vcoarsevec^T \mathbf{R}(\ucoarsevec(t))$. Gappy POD~\cite{everson_sirovich_gappy} and the (discrete) empirical interpolation method~\cite{eim,deim} are two such techniques. More details on hyper-reduction are provided in Appendix~\ref{appendix:hyper}.

When applied to unsteady non-linear problems, the Galerkin ROM is often inaccurate and, at times, unstable. Examples of this are seen in Ref.~\cite{carlberg_lspg_v_galerkin}. These issues motivate the development of more sophisticated reduced-order modeling techniques. 

\subsubsection{Petrov--Galerkin and Least-Squares Petrov--Galerkin Reduced-Order Models}
In the Petrov--Galerkin approach, the test space is  different from the trial space. Petrov--Galerkin approaches have a rich history in the finite element community~\cite{brooks_supg,hughes_petrovgalerkin} and can enhance the stability and robustness of a numerical method. In the context of reduced-order modeling for dynamical systems, the Least-Squares Petrov--Galerkin method (LSPG)~\cite{carlberg_lspg} is a popular approach. The LSPG approach is a ROM technique that seeks to minimize the fully discrete residual (i.e., the residual after spatial and temporal discretization) at each time-step. The LSPG method can be shown to be optimal in the sense that it minimizes the $L^2$-norm of the \textit{fully discrete} residual at each time-step over $\text{Range}(\Vcoarse)$. To illustrate the LSPG method, consider the algebraic system of equations for the FOM obtained after an implicit Euler temporal discretization,
\begin{equation}\label{eq:coarse_implicit_euler_0}
\frac{\ufullvec^{n} - \ufullvec^{n-1}  }{\Delta t} - \mathbf{R}(\ufullvec^{n}) = \mathbf{0},
\end{equation}
where $\ufullvec^n \in \RR{N}$ denotes the solution at the $n^{th}$ time-step.
The FOM will exactly satisfy Eq.~\ref{eq:coarse_implicit_euler_0}. The ROM, however, will not. The LSPG method minimizes the residual of Eq.~\ref{eq:coarse_implicit_euler_0} over each time-step. For notational purposes, we define the residual vector for the implicit Euler method,
\begin{equation*}
\mathbf{r}_{\text{IE}}: (\mathbf{y};\ufullvec^{n-1}) \mapsto  \frac{\mathbf{y} - \ufullvec^{n-1}  }{\Delta t} -  \mathbf{R}(\mathbf{y}).
\end{equation*}
The LSPG method is defined as follows,
\begin{equation*}
\ufullvec^n = \underset{\mathbf{y} \in \text{Range}(\vcoarsevec)  }{\text{arg min}}|| \mathbf{A}(\mathbf{y}) \mathbf{r}_{\text{IE}}(\mathbf{y};\ufullvec^{n-1}) ||_2^2,
\end{equation*}
where $\mathbf{A}(\cdot) \in \mathbb{R}^{z \times n}$ with $z \le N$ is a weighting matrix. The standard LSPG method takes $\mathbf{A} = \mathbf{I}$. For the implicit Euler time integration scheme (as well as various other implicit schemes) the LSPG approach can be shown to have an equivalent continuous representation using a Petrov--Galerkin projection~\cite{carlberg_lspg_v_galerkin}. For example, the LSPG method for any backward differentiation formula (BDF) time integration scheme can be written as a Petrov--Galerkin ROM with the test basis,
\begin{equation*}
\mathbf{\Wcoarsevec} =\big( \mathbf{I} - \alpha \Delta t \mathbf{J}[\ucoarsevec(t)] \big) \vcoarsevec,
\end{equation*}
where $\mathbf{J}[\ucoarsevec(t)] = \frac{\partial \RHS}{\partial \statedummy}(\ucoarsevec(t))$ is the Jacobian of the right-hand side function evaluated about the coarse-scale state and $\alpha$ is a constant, specific to a given scheme (e.g. $\alpha = 1$ for implicit Euler, $\alpha = \frac{2}{3}$ for BDF2, $\alpha = \frac{6}{11}$ for BDF3, etc.).
With this test basis, the LSPG ROM can be written as,
\begin{equation}\label{eq:LSPGROM}
 \vcoarsevec^T \bigg( \frac{d }{dt}\ucoarsevec(t) -   \mathbf{R}(\ucoarsevec(t)) \bigg) = \vcoarsevec^T \mathbf{J}^T [\ucoarsevec(t)]  \alpha  \Delta t\bigg(  \frac{d }{dt}\ucoarsevec(t) - \mathbf{R}(\ucoarsevec(t)) \bigg)  , \qquad \ucoarsevec(0) = \ucoarsevec_0, \qquad t \in [0,T].
\end{equation}
In writing Eq.~\ref{eq:LSPGROM}, we have coupled all of the terms from the standard Galerkin ROM on the left-hand side, and have similarly coupled the terms introduced by the Petrov--Galerkin projection on the right-hand side. One immediately observes that the LSPG approach is a residual-based method, meaning that the stabilization added by LSPG is proportional to the residual. The LSPG method is similar to the Galerkin/Least-Squares (GLS) approach commonly employed in the finite element community~\cite{hughes_GLS,hughes0}. This can be made apparent by writing Eq.~\ref{eq:LSPGROM} as,
\begin{equation*}
 \bigg(  \mathbf{v}_i ,  \frac{d }{dt}\ucoarsevec(t) -   \mathbf{R}(\ucoarsevec(t))   \bigg)= \bigg(\mathbf{J}  [\ucoarsevec(t)] \mathbf{v}_i   , \tau \big[  \frac{d }{dt}\ucoarsevec(t) - \mathbf{R}(\ucoarsevec(t)) \big] \bigg) , \qquad i = 1,2,\hdots,K,  
\end{equation*}
where $(\mathbf{a},\mathbf{b}) = \mathbf{a}^T \mathbf{b}$ and $\tau = \alpha \Delta t$ is the stabilization parameter. Compare the above to, say, Eq. 70 and 71 in Ref.~\cite{hughes0}. A rich body of literature exists on residual-based methods, and viewing the LSPG approach in this light helps establish connections with other methods.
We highlight several important aspects of LSPG. Remarks 1 through 3 are derived by Carlberg et al. in Ref.~\cite{carlberg_lspg_v_galerkin}:
\begin{enumerate}
\item The LSPG approach is inherently tied to the temporal discretization. For different time integration schemes, the ``stabilization" added by the LSPG method will vary. For optimal accuracy, the LSPG method requires an intermediary time-step size.
\item In the limit of $\Delta t \rightarrow 0$, the LSPG approach recovers a Galerkin approach.
\item For explicit time integration schemes, the LSPG and Galerkin approach are equivalent.
\item For backwards differentiation schemes, the LSPG approach is a type of GLS stabilization for non-linear problems. 
\item While commonalities exist between LSPG and multiscale approaches, the authors believe that the LSPG method should \textit{not} be viewed as a subgrid-scale model. The reason for this is that it is unclear how Eq.~\ref{eq:LSPGROM} can be derived from Eq.~\ref{eq:FOM_VMS_coarse}. This is similar to the fact that, in Ref.~\cite{hughes0}, \textit{adjoint} stabilization is viewed as a subgrid-scale model while GLS stabilization is not.  The challenge in deriving Eq.~\ref{eq:LSPGROM} from Eq.~\ref{eq:FOM_VMS_coarse} lies primarily in the fact that the Jacobian in Eq.~\ref{eq:LSPGROM} contains a transpose operator. We thus view LSPG as mathematical stabilization rather than a subgrid-scale model.
\end{enumerate}
While the LSPG approach has enjoyed much success for constructing ROMs of non-linear problems, remarks 1, 2, 3, and 5 suggest that improvements over the LSPG method are possible. Remark 1 suggests improvements in computational speed and accuracy are possible by removing sensitivity to the time-step size. Remark 3 suggests that improvements in computational speed and flexibility are possible by formulating a method that can be used with explicit time-stepping schemes. Lastly, remark 5 suggests that improvements in accuracy are possible by formulating a method that accounts for subgrid effects.

\subsection{Mori-Zwanzig Reduced-Order Models}\label{sec:MZ}
The optimal prediction framework formulated by Chorin et al.~\cite{ChorinOptimalPrediction,ChorinOptimalPredictionMemory,Chorin_book}, which is a (significant) reformulation of the Mori-Zwanzig (MZ) formalism of statistical mechanics, is a model order reduction tool that can be used to develop representations of the impact of the fine-scales on the coarse-scale dynamics. In this section, the optimal prediction framework is used to derive a compact approximation to the impact of the fine-scale POD modes on the evolution of the coarse-scale POD modes. For completeness, the optimal prediction framework is first derived in the context of the Galerkin POD ROM. It is emphasized that the content presented in Sections~\ref{sec:liouville} and \ref{sec:projOpsLangevin} is simply a formulation of Chorin's framework, with a specific projection operator, in the context of the Galerkin POD ROM. 


We pursue the MZ approach on a Galerkin formulation of Eq.~\ref{eq:FOM_generalized}. Before describing the formalism, it is beneficial to re-write the original FOM in terms of the generalized coordinates with the solution being defined implicitly as a function of the initial conditions,
\begin{equation}\label{eq:FOM_generalized_b}
\frac{d }{dt} \afullvec (\afullvec_0,t) =  \vfullvec^T \mathbf{R}(\vfullvec \afullvec (\afullvec_0,t)), \qquad \afullvec(0) = \afullvec_0, \qquad t \in [0,T],
\end{equation}
with $\afullvec: \mathbb{R}^N \times [0,T] \rightarrow \mathbb{R}^N$, $\afullvec \in \RR{N} \otimes \RRC{N} \otimes \timeSpace$ the time-dependent generalized coordinates, $\RRC{N}$ the space of (sufficiently smooth) functions acting on $\RR{N}$, $\timeSpace$ the space of (sufficiently smooth) functions acting on $[0,T]$, and $\afullvec_0 \in \mathbb{R}^N$ the initial conditions. Here, $\afullvec(\afullvec_0,t)$ is viewed as a function that maps from the coordinates $\afullvec_0$ (i.e., the initial conditions) and time to a vector in $\mathbb{R}^N$. It is assumed that the right-hand side operator $\mathbf{R}$ is continuously differentiable on $\mathbb{R}^N$.
\subsubsection{The Liouville Equation}\label{sec:liouville}
The starting point of the MZ approach is to transform the non-linear FOM (Eq.~\ref{eq:FOM_generalized_b}) into a linear partial differential equation. 
Equation~\ref{eq:FOM_generalized_b} can be written equivalently as the following partial differential equation in $\PDEspace$~\cite{ChorinOptimalPredictionMemory},
\begin{equation}\label{eq:Liouville}
\frac{\partial }{\partial t}v(\afullvec_0,t) = \MC{L} v(\afullvec_0,t); \qquad
v(\mathbf{a_0},0) = g(\mathbf{a}_0),
\end{equation}
where $v:  \mathbb{R}^N  \times [0,T] \rightarrow \mathbb{R}^{N_v}$ with $v \in \RR{N_v} \otimes \RRC{N} \otimes \timeSpace$ is a set of $N_v$ observables and $g: \mathbb{R}^N  \rightarrow \mathbb{R}^{N_v}$ is a state-to-observable map. The operator $\MC{L}$ is the Liouville operator, also known as the Lie derivative, and is defined by,
\begin{align*}
\MC{L} &: \vdummy \mapsto \bigg[ \frac{\partial  }{\partial \mathbf{a_0}} \vdummy \bigg] \vfullvec^T \mathbf{R}( \vfullvec \mathbf{a_0} ),\\
       &: \RR{q} \otimes \RRC{N} \rightarrow \RR{q} \otimes \RRC{N},
\end{align*}
for arbitrary q.
Equation~\ref{eq:Liouville} is referred to as the Liouville equation and is an exact statement of the original dynamics. The Liouville equation describes the solution to Eq.~\ref{eq:FOM_generalized_b} for \textit{all} possible initial conditions. The advantage of reformulating the system in this way is that the Liouville equation is linear, allowing for the use of superposition and aiding in the removal of the fine-scales.

The solution to Eq.~\ref{eq:Liouville} can be written as,
\begin{equation*}
v(\mathbf{a_0},t) = e^{t \MC{L}}  g(\mathbf{a_0}) .
\end{equation*}
The operator $e^{t \MC{L}}$, which has been referred to as a ``propagator", evolves the solution along its trajectory in phase-space~\cite{ZwanzigBook}. The operator $e^{t \MC{L}}$ has several interesting properties. Most notably, the operator can be ``pulled" inside of a non-linear functional~\cite{ZwanzigBook},
\begin{equation*}
e^{t \MC{L}}  g(\mathbf{a}_0) =   g( e^{t \MC{L}} \mathbf{a}_0).
\end{equation*}
This is similar to the composition property inherent to Koopman operators~\cite{Koopman}. With this property, the solution to Eq.~\ref{eq:Liouville} may be written as,
\begin{equation*}
v(\mathbf{a_0},t) =  g( e^{t \MC{L}}  \mathbf{a}_0).
\end{equation*}
The implications of $e^{t \MC{L}}$ are significant. It demonstrates that, given trajectories $\mathbf{a}(\mathbf{a_0},t)$, the solution $v$ is known for any observable $g$. 

Noting that $\MC{L}$ and $e^{t \MC{L}}$ commute, Eq.~\ref{eq:Liouville} may be written as,
\begin{equation*}
\frac{\partial }{\partial t} v(\xfullvec,t) = e^{t \MC{L}}  \MC{L} v(\xfullvec,0).
\end{equation*}
A set of partial differential equations for the resolved generalized coordinates can be obtained by taking $g(\mathbf{a_0}) = \acoarsevec_0$,
\begin{equation}\label{eq:Liouville_sg_res}
\frac{\partial }{\partial t}  e^{t \MC{L}} \acoarsevec_0 = e^{t \MC{L}}  \MC{L}  \acoarsevec_0.
\end{equation}
The remainder of the derivation is performed for $g(\mathbf{a_0}) = \acoarsevec_0$, thus $N_v = K$. 
\subsubsection{Projection Operators and the Generalized Langevin Equation}\label{sec:projOpsLangevin}
The objective now is to remove the dependence of Eq.~\ref{eq:Liouville_sg_res} on the fine-scale variables. 
Similar to the VMS decomposition, $\icspace$ can be decomposed into resolved and unresolved subspaces,
\begin{equation*}
\icspace = \icspacecoarse \oplus \icspacefine,
\end{equation*} 
with $\icspacecoarse$ being the space of all functions of the resolved coordinates, $\xcoarsevec$, and $\icspacefine$ the complementary space.
The associated projection operators are defined as $\mathcal{P}: \icspace \rightarrow \icspacecoarse$ and $\mathcal{Q} = I - \mathcal{P}$. Various types of projections are possible, and here we consider,
\begin{equation*}
\MC{P}f( \afullvec_0 ) = \int_{\RR{N}} f( \afullvec_0 ) \delta(\afinevec_0) d \afinevec_0,
\end{equation*}
which leads to
\begin{equation*}
\MC{P}f(\afullvec_0 )= f([\acoarsevec_0;\mathbf{0}]).
\end{equation*}

The projection operators can be used to split the Liouville equation,
\begin{equation}\label{eq:Liouville_sg_split}
\frac{\partial }{\partial t}  e^{t \MC{L}} \xcoarsevec = e^{t \MC{L}}  \MC{PL} \xcoarsevec +  e^{t \MC{L}}\MC{QL}\xcoarsevec.
\end{equation}
The objective now is to remove the dependence of the right-hand side of Eq.~\ref{eq:Liouville_sg_split} on the fine-scales, $\xfinevec$ (i.e. $\MC{QL}\xcoarsevec$). This may be achieved by Duhamel's principle,
\begin{equation}\label{eq:duhamel}
e^{t \mathcal{L}} = e^{t \mathcal{Q} \mathcal{L}} + \int_0^t e^{(t - s)\mathcal{L}} \mathcal{P}\mathcal{L} e^{s \mathcal{Q} \mathcal{L}} ds.
\end{equation}
Inserting Eq.~\ref{eq:duhamel} into Eq.~\ref{eq:Liouville_sg_split}, the generalized Langevin equation is obtained,
\begin{equation}\label{eq:MZ_Identity}
\frac{\partial }{\partial t}  e^{t \MC{L}} \xcoarsevec =   \underbrace{e^{t\MC{L}}\MC{PL} \xcoarsevec}_{\text{Markovian}} +   \underbrace{e^{t\MC{QL}}\MC{QL}  \xcoarsevec}_{\text{Noise}} + 
 \underbrace{ \int_0^t e^{{(t - s)}\mathcal{L}} \mathcal{P}\mathcal{L} e^{s \mathcal{Q} \mathcal{L}} \MC{QL}\xcoarsevec ds}_{\text{Memory}}.
\end{equation}
By the definition of the initial conditions (Eq.~\ref{eq:ROM_IC}), the noise-term is zero and we obtain,
\begin{equation}\label{eq:MZ_Identity3}
\frac{\partial}{\partial t}    e^{t \MC{L}} \xcoarsevec = e^{t\MC{L}}\MC{PL}  \xcoarsevec+   \int_0^t e^{{(t - s)}\mathcal{L}} \mathcal{P}\mathcal{L} e^{s \mathcal{Q} \mathcal{L}} \MC{QL}  \xcoarsevec ds.
\end{equation}
The system described in Eq.~\ref{eq:MZ_Identity} is precise and not an approximation to the original ODE system. For notational purposes, define,
\begin{equation}\label{eq:kerndef}
\mathbf{K}(\xcoarsevec,t) \equiv \MC{PL}e^{t\MC{QL}}\MC{QL}\xcoarsevec.
\end{equation}
The term $\mathbf{K}: \mathbb{R}^K \times [0,T] \rightarrow \mathbb{R}^K$ with $\mathbf{K} \in \RR{K} \otimes  \icspacecoarse \otimes \timeSpace$ is referred to as the memory kernel. 

Using the identity $e^{t \MC{L}} \MC{PL}\xcoarsevec = \vcoarsevec^T \mathbf{R}(\ucoarsevec(t))$ and Definition~\eqref{eq:kerndef}, Equation~\ref{eq:MZ_Identity3} can be written in a more transparent form,
\begin{equation}\label{eq:MZ_Identity_VMS}
\vcoarsevec^T \bigg( \frac{\partial }{ \partial t}\ucoarsevec(t) - \mathbf{R}(\ucoarsevec(t)) \bigg) = \int_0^t \mathbf{K}(\acoarsevec(t-s),s) ds,
\end{equation}
Note that the time derivative is represented as a partial derivative due to the Liouville operators embedded in the memory.

The derivation up to this point has cast the original full-order model in generalized coordinates (Eq.~\ref{eq:FOM_generalized_b}) as a linear PDE. Through the use of projection operators and Duhamel's principle, an \textit{exact} equation (Eq.~\ref{eq:MZ_Identity_VMS}) for the coarse-scale dynamics \textit{only in terms of the coarse-scale variables} was then derived. The effect of the fine-scales on the coarse-scales appeared as a memory integral. This memory integral may be thought of as the closure term that is required to exactly account for the unresolved dynamics. 
 
\subsubsection{The $\tau$-model and the Adjoint Petrov--Galerkin Method}\label{sec:tau-model}
The direct evaluation of the memory term in Eq.~\ref{eq:MZ_Identity_VMS} is, in general, computationally intractable. To gain a reduction in computational cost, an approximation to the memory must be devised. A variety of such approximations exist, and here we outline the $\tau$-model~\cite{parish_dtau,BarberThesis}. The $\tau$-model can be interpreted as the result of assuming that the memory is driven to zero in finite time and approximating the integral with a quadrature rule. This can be written as a two-step approximation,
\begin{equation*}
\int^t_0 \mathbf{K}(\acoarsevec(t-s),s) ds \approx \int^t_{t-\tau} \mathbf{K}(\acoarsevec(t-s),s) ds \approx \tau \mathbf{K}(\acoarsevec(t),0).
\end{equation*}
Here, $\tau \in \RR{}$ is a stabilization parameter that is sometimes referred to as the ``memory length." It is typically static and user-defined, though methods of dynamically calculating it have been developed in \cite{parish_dtau}. The \textit{a priori} selection of $\tau$ and sensitivity of the model output to this selection are discussed later in this manuscript. 

The term $\mathbf{K}(\acoarsevec(t),0)$ can be shown to be~\cite{parishVMS},
\begin{equation*}
\mathbf{K}(\acoarsevec(t),0) =  \vcoarsevec^T \mathbf{J}[\ucoarsevec(t)]  \Pifine \mathbf{R}(\ucoarsevec(t)),
\end{equation*}
where $\Pifine$ is the ``orthogonal projection operator," defined as $\Pifine \equiv \big(\mathbf{I}-\vcoarsevec \vcoarsevec^T\big)$. We define the corresponding coarse-scale projection operator as $\Picoarse \equiv \vcoarsevec \vcoarsevec^T$. The coarse-scale equation with the $\tau$-model reads,
\begin{equation}\label{eq:MZ_coarse_tau_NL}
\vcoarsevec^T \bigg( \frac{d }{dt}\ucoarsevec(t) - \mathbf{R}(\ucoarsevec(t)) \bigg) = \tau \vcoarsevec^T \mathbf{J}[\ucoarsevec(t)]  \Pifine \mathbf{R}(\ucoarsevec(t)).
\end{equation}
Equation~\ref{eq:MZ_coarse_tau_NL} provides a closed equation for the evolution of the coarse-scales. The left-hand side of Eq.~\ref{eq:MZ_coarse_tau_NL} is the standard Galerkin ROM, and the right-hand side can be viewed as a subgrid-scale model. 

When compared to existing methods, the inclusion of the $\tau$-model leads to a method that is analogous to a non-linear formulation of the \textit{adjoint} stabilization technique developed in the finite element community. The ``adjoint" terminology arises from writing Eq.~\ref{eq:MZ_coarse_tau_NL} in a Petrov--Galerkin form,
\begin{equation}\label{eq:adjoint_Galerkin}
 \bigg[ \bigg( \mathbf{I} + \tau \Pifine^T \mathbf{J}^T[\ucoarsevec(t)]\bigg)  \vcoarsevec \bigg]^T  \bigg( \frac{d }{dt}\ucoarsevec(t) - \mathbf{R}(\ucoarsevec(t)) \bigg) = \mathbf{0}.
\end{equation}
It is seen that Eq.~\ref{eq:adjoint_Galerkin} involves taking the inner product of the coarse-scale ODE with a test-basis that contains the adjoint of the coarse-scale Jacobian. Unlike GLS stabilization, adjoint stabilization can be derived from the multiscale equations~\cite{hughes0}. Due to the similarity of the proposed method with adjoint stabilization techniques, as well as the LSPG terminology, the complete ROM formulation will be referred to as the Adjoint Petrov--Galerkin (APG) method. 

\subsubsection{Comparison of APG and LSPG}
The APG method displays similarities to LSPG. From Eq.~\ref{eq:adjoint_Galerkin}, it is seen that the test basis for the APG ROM is given by,
\begin{equation}\label{eq:MZ_testbasis}
\Wcoarsevec_{A}  = \bigg( \mathbf{I} + \tau \Pifine^T \mathbf{J}^T[\ucoarsevec]\bigg)  \vcoarsevec .
\end{equation}
Recall the LSPG test basis for backward differentiation schemes,
\begin{equation}\label{eq:LSPG_testbasis}
\Wcoarsevec_{LSPG} = \big( \mathbf{I} - \alpha \Delta t \mathbf{J}[\ucoarsevec] \big) \vcoarsevec.
\end{equation}
Comparing Eq.~\ref{eq:LSPG_testbasis} to Eq.~\ref{eq:MZ_testbasis}, we can draw several interesting comparisons between the LSPG and APG method. Both contain a time-scale: $\tau$ for APG and $\alpha \Delta t$ for LSPG. Both include Jacobians of the non-linear function $\mathbf{R}(\ucoarsevec)$. The two methods differ in the presence of the orthogonal projection operator in APG, a transpose on the Jacobian, and a sign discrepancy on the Jacobian. These last two differences are consistent with the discrepancies between GLS and adjoint stabilization methods used in the finite element community. See, for instance, Eqs. 71 and 73 in Ref~\cite{hughes0}.

\section{Analysis}\label{sec:analysis}
This section presents theoretical analyses of the Adjoint Petrov--Galerkin method. Specifically, error and eigenvalue analyses are undertaken for linear time-invariant (LTI) systems.
Section~\ref{sec:error_bound} derives \textit{a priori} error bounds for the Galerkin and Adjoint Petrov--Galerkin ROMs. Conditions under which the APG ROM may be more accurate than the Galerkin ROM are discussed. Section~\ref{sec:selecttau} outlines the selection of the parameter $\tau$ that appears in APG. 
\subsection{A Priori Error Bounds}\label{sec:error_bound}
We now derive \textit{a priori} error bounds for the Galerkin and Adjoint Petrov--Galerkin method for LTI systems. Define $\ufullvec_F$ to be the solution to the FOM, $\ucoarsevec_G$ to be the solution to the Galerkin ROM, and $\ucoarsevec_A$ the solution to the Adjoint Petrov--Galerkin ROM. The full-order solution, Galerkin ROM, and Adjoint Petrov--Galerkin ROMs obey the following dynamical systems,
\begin{equation}\label{eq:fom_error}
\frac{d }{dt}\ufullvec_F(t) =  \mathbf{R}(\ufullvec_F(t)), \qquad \ufullvec_F(0) = \ufullvec_0,
\end{equation}
\begin{equation}\label{eq:grom_error}
\frac{d }{dt}\ucoarsevec_G(t) = \mathbb{P}_G \mathbf{R}(\ucoarsevec_G(t)), \qquad \ufullvec_G(0) = \ufullvec_0,
\end{equation}
\begin{equation}\label{eq:ag_error}
\frac{d }{dt}\ucoarsevec_A(t) = \mathbb{P}_{A} \mathbf{R}(\ucoarsevec_A(t)), \qquad \ufullvec_A(0) = \ufullvec_0,
\end{equation}
where the Galerkin and Adjoint Petrov--Galerkin projections are, respectively,
\begin{equation*}
\mathbb{P}_G = \Picoarse, \qquad \mathbb{P}_A =  \Picoarse \big[ \mathbf{I} + \tau \mathbf{J}[\ucoarsevec_A] \Pifine \big].    
\end{equation*}
The residual of the full-order model is defined as,
\begin{equation*}
\mathbf{r}_F : \ufullvec \mapsto \frac{d \ufullvec}{dt} - \mathbf{R}(\ufullvec).
\end{equation*}
We define the error in the Galerkin and Adjoint Petrov--Galerkin method as,
\begin{equation*}
\error_G \trieq \ufullvec_F - \ucoarsevec_G, \qquad \error_A \defeq \ufullvec_F - \ucoarsevec_A.   
\end{equation*}
Similarly, the coarse-scale error is defined as,
\begin{equation*}
\errorcoarse_G \trieq \Picoarse \ufullvec_F - \ucoarsevec_G, \qquad \tilde{\error}_A \defeq \Picoarse  \ufullvec_F - \ucoarsevec_A.   
\end{equation*}
In what follows, we assume Lipschitz continuity of the right-hand side function: there exists a constant $\kappa > 0$ such that $\forall \mathbf{x},\mathbf{y} \in \mathbb{R}^N$,
\begin{equation*}
\norm{ \mathbf{R}(\mathbf{x}) - \mathbf{R}(\mathbf{y})} \le \kappa \norm{ \mathbf{x} - \mathbf{y}}. 
\end{equation*}
To simplify the analysis, the Adjoint Petrov--Galerkin projection is approximated to be stationary in time. Note that the Galerkin projection is stationary in time.  For clarity, we suppress the temporal argument on the states when possible in the proofs.
\begin{theorem}
\textit{A priori} error bounds for the Galerkin and Adjoint Petrov--Galerkin ROMs are, respectively,
\begin{equation}\label{eq:g_nlbound} \norm{\error_G(t)} \le \int_0^t e^{ \norm{\mathbf{P}_G} \kappa s }\norm{\big[\mathbf{I} - \mathbb{P}_G \big]  \mathbf{R}( \ufullvec_F(t-s) ) } ds .\end{equation}
\begin{equation}\label{eq:ag_nlbound} \norm{\error_A(t)} \le \int_0^t e^{\norm{\mathbf{P}_A} \kappa s }\norm{\big[\mathbf{I} - \mathbb{P}_A \big]  \mathbf{R}(\ufullvec_F(t-s))} ds .\end{equation}
\end{theorem}

\begin{proof}
We prove only Eq.~\ref{eq:ag_nlbound} as Eq.~\ref{eq:g_nlbound} is obtained through the same arguments. Following~\cite{carlberg_lspg_v_galerkin}, start by subtracting Eq.~\ref{eq:ag_error} from Eq.~\ref{eq:fom_error}, and adding and subtracting $\mathbb{P}_A\mathbf{R} ( \ufullvec_F$),
\begin{equation*}\label{eq:ea_lti_1}
\frac{d \error_A}{dt} = \mathbf{R} (\ufullvec_F) + \mathbb{P}_A \mathbf{R} (\ufullvec_F) - \mathbb{P}_A \mathbf{R} (\ufullvec_F)  - \mathbb{P}_A \mathbf{R}(\ucoarsevec_G), \qquad \error_A(0) = \mathbf{0}.
\end{equation*}
Taking the $L^2$-norm,
\begin{equation*}\label{eq:ea_lti_2}
\norm{ \frac{d \error_A}{dt} } = \norm{ \mathbf{R} (\ufullvec_F) + \mathbb{P}_A \mathbf{R} (\ufullvec_F) - \mathbb{P}_A \mathbf{R} (\ufullvec_F)  - \mathbb{P}_A \mathbf{R}(\ucoarsevec_G)}.
\end{equation*}
Applying the triangle inequality,
\begin{equation*}\label{eq:ea_lti_3}
\norm{ \frac{d \error_A}{dt} }\le \norm{\big[\mathbf{I} - \mathbb{P}_A \big]  \mathbf{R}( \ufullvec_F) }  + \norm{  \mathbb{P}_A \big( \mathbf{R} \ufullvec_F)    -  \mathbf{R}(\ucoarsevec_G) \big) } .
\end{equation*}
Invoking the assumption of Lipschitz continuity,
\begin{equation*}\label{eq:ea1}
\norm{ \frac{d \error_A}{dt} }\le \norm{\big[\mathbf{I} - \mathbb{P}_A \big]  \mathbf{R}( \ufullvec_F)}  +\norm{ \mathbb{P}_A} \kappa  \norm{\error_A}.
\end{equation*}
Noting that $\frac{d \norm{ \error_A } }{dt} \le \norm{ \frac{d \error_A}{dt}}$ 
we have
\footnote{ 
$$\frac{d \norm{ \error_A } }{dt} = \frac{1}{\norm{\error_A}} \error_A^T \frac{d \error_A}{dt} \le  \norm{ \frac{1}{\norm{\error_A}} \error_A} \norm{ \frac{d \error_A}{dt}} \le \norm{ \frac{d \error_A}{dt}}
$$},
\begin{equation}\label{eq:ea2}
\frac{d  \norm{\error_A }}{dt} \le \norm{\big[\mathbf{I} - \mathbb{P}_A \big]  \mathbf{R}( \ufullvec_F)}  +\norm{ \mathbb{P}_A} \kappa  \norm{\error_A}.
\end{equation}
An upper bound on the error for the Adjoint Petrov--Galerkin method is then obtained by solving Eq.~\ref{eq:ea2} for $\norm{\error_A}$, which yields,
\begin{equation*}
\norm{ \error_A(t) } \le \int_0^t e^{\norm{ \mathbb{P}_A } \kappa s }\norm{\big[\mathbf{I} - \mathbb{P}_A \big]  \mathbf{R} (\ufullvec_F(t-s))} ds .
\end{equation*}

 \end{proof}
Note that, for the Adjoint Petrov--Galerkin method, the error bound provided in Eq.~\ref{eq:ag_nlbound} is not truly an \textit{a priori} bound as  $\mathbb{P}_A$ is a function of $\ucoarsevec_A$.
Equation~\ref{eq:g_nlbound} (and~\ref{eq:ag_nlbound}) indicates an exponentially growing error and contains two distinct terms. The term $e^{\norm{\mathbb{P}_A} \kappa s }$ indicates the exponential growth of the error in time.  The second term of interest is $\norm{\big[\mathbf{I} - \mathbb{P}_A \big]  \mathbf{R} (\ufullvec_F(t))}$. This term corresponds to the error introduced at time $t$ due to projection.  It is important to note that the first term controls how the error will grow in time, while the second term controls how much error is added at a given time. 

Unfortunately, for general non-linear systems, \textit{a priori} error analysis provides minimal insight beyond what was just mentioned. To obtain a more intuitive understanding of the APG method, error analysis in the case that $\mathbf{R}(\cdot)$ is a linear time-invariant operator is now considered. 

\begin{theorem}
Let $\mathbf{R}(\ufullvec) = \mathbf{A}\ufullvec$ be a linear time-invariant operator. Error bounds for the coarse-scales in the Galerkin and APG ROMs, are, respectively,
\begin{equation}\label{eq:g_ltibound_inplane}
\norm{\errorcoarse_G(t) } \le  \norm{\mathbf{S}_G} \norm{\mathbf{S}_G^{-1} }\int_0^t  \norm{ e^{ \Lambda_{G} (t-s)}  } \norm{ \vcoarsevec^T  \mathbb{P}_G \mathbf{r}_F (\ucoarsevec_F(s)) } ds ,
\end{equation}
\begin{equation}\label{eq:ag_ltibound_inplane}
\norm{\errorcoarse_A(t) } \le \norm{\mathbf{S}_A} \norm{\mathbf{S}_A^{-1} } \int_0^t \norm{ e^{ \Lambda_{A}(t-s)} } \norm{\vcoarsevec^T   \mathbb{P}_A \mathbf{r}_F (\ucoarsevec_F (s)) }ds ,
\end{equation}
where $\Lambda_{G}$ is a diagonal matrix of the eigenvalues of $\vcoarsevec^T \mathbb{P}_G \mathbf{A} \vcoarsevec$, while $\mathbf{S}_G$ are the eigenvectors of $\vcoarsevec^T \mathbb{P}_G \mathbf{A} \vcoarsevec$. Similar definitions hold for $\Lambda_{A}$ and $\mathbf{S}_A$ using the APG projection. 
\end{theorem}
\begin{proof}
We prove only Eq.~\ref{eq:ag_ltibound_inplane} as Eq.~\ref{eq:g_ltibound_inplane} is obtained through the same arguments. Subtracting Eq.~\ref{eq:ag_error} from Eq.~\ref{eq:fom_error}, and adding and subtracting $\mathbb{P}_A\mathbf{A} \ufullvec_F$,
$$
\frac{d \error_A}{dt} = \big[\mathbf{I} - \mathbb{P}_A \big]  \mathbf{A} \ufullvec_F  +  \mathbb{P}_A \mathbf{A} \error_A, \qquad \mathbf{e}_A(0) = \mathbf{0}.
$$
We decompose the error into coarse and fine-scale components $\error_A = \errorcoarse_A + \errorfine_A.$ Since $\ufinevec_A = \mathbf{0}$ for all time, $\errorfine_A = \ufinevec_{F}$.
An equation for the coarse-scale error is obtained by left multiplying by $\vcoarsevec^T \mathbb{P}_A$,
$$
\vcoarsevec^T \mathbb{P}_A \frac{d \errorcoarse_A}{dt} + \vcoarsevec^T \mathbb{P}_A \frac{d \errorfine_A}{dt} = \vcoarsevec^T \mathbb{P}_A \big[\mathbf{I} - \mathbb{P}_A \big]  \mathbf{A} \ufullvec_F  +  \vcoarsevec^T \mathbb{P}_A  \mathbb{P}_A \mathbf{A} \errorcoarse_A + \vcoarsevec^T \mathbb{P}_A \mathbb{P}_A \mathbf{A} \ufinevec_{F}.
$$
Now express the coarse-scale error as,
$$\errorcoarse_A = \vcoarsevec \errorcoarsegen_A,$$
where $\errorcoarsegen_A: [0,T] \rightarrow \RR{K}$ are the coarse-scale error generalized coordinates.
Rearranging gives,
$$
\frac{d \errorcoarsegen_A}{dt}  = \vcoarsevec^T \mathbb{P}_A \big[\mathbf{I} - \mathbb{P}_A \big]  \mathbf{A} \ufullvec_F  +  \vcoarsevec^T \mathbb{P}_A  \mathbb{P}_A \mathbf{A} \errorcoarse_A + \vcoarsevec^T \mathbb{P}_A \mathbb{P}_A \bigg( \mathbf{A} \ufinevec_{F} - \frac{d \ufinevec_F}{dt} \bigg).
$$
Noting that $\mathbb{P}_A^2 = \mathbb{P}_A$, and that the first term on the right-hand side is zero, the equation reduces to,
\begin{equation}\label{eq:lti_inplane}
\frac{d \errorcoarsegen_A}{dt} = \vcoarsevec^T \mathbb{P}_A \mathbf{A} \vcoarsevec \errorcoarsegen_A + \vcoarsevec^T \mathbb{P}_A \bigg(  \mathbf{A} \ufinevec_{F} -  \frac{d \ufinevec_F}{dt}\bigg).
\end{equation}
Equation~\ref{eq:lti_inplane} is a first-order linear inhomogeneous differential equation that can be solved analytically. A standard way to do this is by performing the eigendecomposition,
\begin{equation}\label{eq:eigendecompostion}
\vcoarsevec^T \mathbb{P}_A \mathbf{A} \vcoarsevec = \mathbf{S_A}\Lambda_A \mathbf{ S_A^{-1} },
\end{equation}
where $\mathbf{S}_A \in \mathbb{R}^{K \times K}$ are the eigenvectors and $\Lambda_A$ is a diagonal matrix of eigenvalues. Left multiplying Eq.~\ref{eq:lti_inplane} by $\mathbf{S_A}^{-1}$ and inserting the eigendecomposition~\ref{eq:eigendecompostion}, Eq.~\ref{eq:lti_inplane} 
becomes,
 \begin{equation}\label{eq:eg_lti_2}
\frac{d \mathbf{w}_A}{dt} =  \Lambda_A \mathbf{w}_A + \mathbf{S_A}^{-1} \vcoarsevec^T \mathbb{P}_A \bigg(  \mathbf{A} \ufinevec_{F} -  \frac{d \ufinevec_F}{dt}\bigg),
\end{equation}
where $\mathbf{w}_A = \mathbf{S}^{-1} \errorcoarsegen_A$. 
The above is a set of decoupled first order inhomogenous differential equation that have the solution,
$$
\mathbf{w}_A(t) = \int_0^t e^{ \Lambda_A s}  \mathbf{S_A}^{-1} \vcoarsevec^T \mathbb{P}_A \bigg(  \mathbf{A} \ufinevec_{F}(t-s) -  \frac{d }{dt}\ufinevec_F(t-s)\bigg) ds. 
$$
Left multiplying by $\vcoarsevec \mathbf{S_A}$ to obtain an equation for the coarse-scale error,
$$
\errorcoarse_A(t) = \int_0^t \vcoarsevec  \mathbf{S_A} e^{ \Lambda_A (t-s)}  \mathbf{S_A}^{-1} \vcoarsevec^T \mathbb{P}_A \bigg(  \mathbf{A} \ufinevec_{F}(s) -  \frac{d }{dt}\ufinevec_F(s)\bigg) ds.
$$
By definition of the FOM,
$$\frac{d \ucoarsevec_F}{dt} + \frac{d \ufinevec_F}{dt} = \mathbf{A}\ucoarsevec_F + \mathbf{A}\ufinevec_F.
$$
Thus,
$$
\errorcoarse_A(t) = \int_0^t \vcoarsevec  \mathbf{S_A} e^{ \Lambda_A (t-s)}  \mathbf{S_A}^{-1} \vcoarsevec^T \mathbb{P}_A \bigg(\frac{d }{dt}\ucoarsevec_F(s) - \mathbf{A} \ucoarsevec_F(s) \bigg) ds.
$$
Using the definition of the full-order residual,
$$
\errorcoarse_A(t) = \int_0^t \vcoarsevec  \mathbf{S_A} e^{ \Lambda_A (t-s)}  \mathbf{S_A}^{-1} \vcoarsevec^T \mathbb{P}_A \mathbf{r}_F(\ucoarsevec_F(s) )ds.
$$
Note that, although $\mathbf{r}_F(\ufullvec_F)$ is exactly zero, this is no longer guaranteed when evaluated at the projected solution $\ucoarsevec_F$. Taking the norm and using the sub-multiplicative property,
$$
\norm{ \errorcoarse_A(t) } \le \norm{\mathbf{S_A}} \norm{  \mathbf{S_A}^{-1}  }\int_0^t \norm{ e^{ \Lambda_A (t-s)}   } \norm{ \vcoarsevec^T \mathbb{P}_A \mathbf{r}_F(\ucoarsevec_F(s) ) }ds.
$$
\end{proof}
Similar to the non-linear case, Eqns~\ref{eq:g_ltibound_inplane} and~\ref{eq:ag_ltibound_inplane} contain two distinct terms. The first term corresponds to the exponential growth in time, $e^{\Lambda_A s}$. In contrast to the non-linear case, however, the arguments in the exponential contain the eigenvalues of $\vcoarsevec^T \mathbb{P}_A \mathbf{A} \vcoarsevec$. When these eigenvalues are negative, the term $e^{\Lambda_A s}$  can diminish the growth of error in time. The second term of interest is given by  $\norm{ \vcoarsevec^T \mathbb{P}_A \mathbf{r}_F(\ucoarsevec(t) ) }$. This term corresponds to the error introduced due to the unresolved scales at time $t$. This term is a statement of the error introduced due to the closure problem. 

\begin{corollary}\label{corollary:resid_bound}
If $\norm{ \vcoarsevec^T \mathbb{P}_A \mathbf{r}(\ucoarsevec_F(t)) } \le \norm{\vcoarsevec^T \mathbb{P}_G \mathbf{r}(\ucoarsevec_F(t))}$,  then the upper bound on the error introduced due to the fine-scales is less for APG than for the Galerkin Method.
\end{corollary}
\begin{proof}
By inspection, when $\norm{ \vcoarsevec^T \mathbb{P}_A \mathbf{r}_F(\ucoarsevec_F(t)) } \le \norm{\vcoarsevec^T \mathbb{P}_G \mathbf{r}_F(\ucoarsevec_F(t))}$, the error introduced in the integrand in Eqn.~\ref{eq:ag_ltibound_inplane} due to the coarse-scale residual is less than that in Eqn.~\ref{eq:g_ltibound_inplane}.
\end{proof}

\begin{theorem}\label{theorem:residualbounds}
Let $\mathbf{R}(\ufullvec) = \mathbf{A} \ufullvec$ be a linear time-invariant operator. Upper bounds of the Galerkin and Adjoint Petrov-Galerkin projections of the full-order coarse-scale residual, are, respectively,
\begin{align*}
&\norm{\vcoarsevec^T \mathbb{P}_G \mathbf{r}(\ucoarsevec_F(t)) } \le \norm{ \vcoarsevec^T\mathbf{A} \int_0^{\tau} e^{\Pifine \mathbf{A} \zeta } \Pifine \mathbf{A} \ucoarsevec_F(t - \zeta) d\zeta } +  \norm{ \vcoarsevec^T\mathbf{A} \int_{\tau}^t e^{\Pifine \mathbf{A} \zeta } \Pifine \mathbf{A} \ucoarsevec_F(t - \zeta) d\zeta }, \\
&\norm{ \vcoarsevec^T \mathbb{P}_A \mathbf{r}(\ucoarsevec_F(t)) } \le \norm{\vcoarsevec^T \mathbf{A} \int_0^{\tau} e^{\Pifine \mathbf{A} \zeta } \Pifine \mathbf{A} \ucoarsevec_F(t - \zeta) d\zeta   - \tau \vcoarsevec^T  \mathbf{A} \Pifine \mathbf{A} \ucoarsevec_F(t) } + \norm{\vcoarsevec^T  \mathbf{A} \int_{\tau}^{t} e^{\Pifine \mathbf{A} \zeta } \Pifine \mathbf{A} \ucoarsevec_F(t - \zeta) d\zeta }.
\end{align*}
\end{theorem}
\begin{proof}
The result will be proved first for the Galerkin method and then for the Adjoint Petrov-Galerkin method. By definition of the residual and the Galerkin projector,
$$
\mathbb{P}_G \mathbf{r}_F(\ucoarsevec_F) = \Picoarse \frac{d \ucoarsevec_F}{dt} - \Picoarse \mathbf{A} \ucoarsevec_F.$$
By definition of the FOM,
$$\Picoarse \bigg[ \frac{d \ucoarsevec_F}{dt} + \frac{d \ufinevec_F}{dt} - \mathbf{A}\ucoarsevec_F - \mathbf{A} \ufinevec_F \bigg] = \mathbf{0}.$$
Noting that $\Pifine \ucoarsevec_F = \mathbf{0}$ and rearranging,
\begin{equation}\label{eq:galerkin_resid}
\mathbb{P}_G \mathbf{r}_F(\ucoarsevec_F) = \Picoarse \mathbf{A} \ufinevec_F.
\end{equation}
The evolution equation of the fine-scales is defined by the FOM projected onto the fine-scale space,
$$\frac{d \ufinevec_F}{dt} = \Pifine \mathbf{A}\ucoarsevec_F + \Pifine \mathbf{A} \ufinevec_F, \qquad \ufinevec_F(0) = \mathbf{0}.$$
The above equation is a first-order nonhomogeneous linear system of differential equations. By the definition of the initial conditions, the solution to the fine-scales reduces to,
\begin{equation*}
\ufinevec(t) = \int_0^{t} e^{\Pifine \mathbf{A} (t-s) } \Pifine \mathbf{A} \ucoarsevec_F(s) ds,
\end{equation*}
where $ e^{\Pifine \mathbf{A} }$ is the matrix exponential. Introducing a change of variables, $\zeta = t-s$, this becomes,
\begin{equation}\label{eq:ufinesol}
    \ufinevec(t) = \int_0^{t} e^{\Pifine \mathbf{A} \zeta } \Pifine \mathbf{A} \ucoarsevec_F(t-\zeta) d\zeta.
\end{equation}
Substituting Eq.~\ref{eq:ufinesol} into Eq.~\ref{eq:galerkin_resid},
\begin{equation}\label{eq:galerkin_term}
\mathbb{P}_G \mathbf{r}_F(\ucoarsevec_F(t)) = \Picoarse \mathbf{A} \int_0^{t} e^{\Pifine \mathbf{A} \zeta } \Pifine \mathbf{A} \ucoarsevec_F(t-\zeta) d\zeta.
\end{equation}
Left multiplying by $\vcoarsevec^T$, noting that $\vcoarsevec^T \Picoarse = \vcoarsevec^T$, and taking the norm,
$$\norm{\vcoarsevec^T \mathbb{P}_G \mathbf{r}_F(\ucoarsevec_F(t)) } = \norm{ \vcoarsevec^T\mathbf{A} \int_0^{t} e^{\Pifine \mathbf{A} \zeta } \Pifine \mathbf{A} \ucoarsevec_F(t-\zeta) d\zeta }.$$
Breaking the integral up into two intervals and applying the triangle inequality, the desired result for Galerkin is obtained,
$$\norm{\vcoarsevec^T \mathbb{P}_G \mathbf{r}_F(\ucoarsevec_F(t)) } \le \norm{ \vcoarsevec^T\mathbf{A} \int_0^{\tau} e^{\Pifine \mathbf{A} \zeta } \Pifine \mathbf{A} \ucoarsevec_F(t-\zeta) d\zeta } +  \norm{ \vcoarsevec^T\mathbf{A} \int_{\tau}^t e^{\Pifine \mathbf{A} \zeta } \Pifine \mathbf{A} \ucoarsevec_F(t-\zeta) d\zeta }.$$
For APG,
$$\mathbb{P}_A \mathbf{r}_F(\ucoarsevec_F(t)) = \Picoarse \frac{d }{dt}\ucoarsevec_F(t) - \Picoarse \mathbf{A} \ucoarsevec_F(t)  - \tau \Picoarse \mathbf{A} \Pifine \mathbf{A} \ucoarsevec_F(t).$$
The first two terms on the right-hand side are the Galerkin term from Eq.~\ref{eq:galerkin_term}. Substituting this into the above equation results in,
$$\mathbb{P}_A \mathbf{r}_F(\ucoarsevec_F(t)) = \Picoarse \mathbf{A} \int_0^{t} e^{\Pifine \mathbf{A} \zeta } \Pifine \mathbf{A} \ucoarsevec_F(t - \zeta) d\zeta   - \tau \Picoarse \mathbf{A} \Pifine \mathbf{A} \ucoarsevec_F(t).$$
Left multiplying by $\vcoarsevec^T$ and taking the norm, 
$$\norm{ \vcoarsevec^T \mathbb{P}_A \mathbf{r}_F(\ucoarsevec_F(t)) } =\norm{\vcoarsevec^T \mathbf{A} \int_0^{t} e^{\Pifine \mathbf{A} \zeta } \Pifine \mathbf{A} \ucoarsevec_F(t - \zeta) d\zeta   - \tau \vcoarsevec^T \Picoarse \mathbf{A} \Pifine \mathbf{A} \ucoarsevec_F(t) }.$$
Breaking the integral up into two intervals, applying the triangle inequality, and noting that $\vcoarsevec^T \Picoarse = \vcoarsevec^T$, the desired result for APG is obtained,
$$\norm{ \vcoarsevec^T \mathbb{P}_A \mathbf{r}_F(\ucoarsevec_F(t)) } \le \norm{\vcoarsevec^T \mathbf{A} \int_0^{\tau} e^{\Pifine \mathbf{A} \zeta } \Pifine \mathbf{A} \ucoarsevec_F(t - \zeta) d\zeta   - \tau \vcoarsevec^T \mathbf{A} \Pifine \mathbf{A} \ucoarsevec_F(t) } + \norm{\vcoarsevec^T \mathbf{A} \int_{\tau}^{t} e^{\Pifine \mathbf{A} \zeta } \Pifine \mathbf{A} \ucoarsevec_F(t - \zeta) d\zeta }.$$

\end{proof}

Theorem~\ref{theorem:residualbounds} provides an upper bound on the error introduced due to the closure problem for both the Galerkin and APG ROMs. These error bounds were obtained by analytically solving for the fine-scales as a function of coarse-scales. This led to a memory integral involving the past history of the coarse-scales. The memory integral expressing the solution to the fine-scales was then split into two intervals: one from $[0,\tau]$ and one from $[\tau,t]$. The APG method can be interpreted as approximating the integral over the first interval with a quadrature rule and ignoring the second interval. The Galerkin ROM ignores both intervals. Next, Theorem~\ref{theorem:apg_error} shows that, for sufficiently small $\tau$, the APG approximation is more accurate than the Galerkin approximation. This result is intuitive since, for sufficiently small $\tau$, the quadrature rule used in deriving APG is accurate.

\begin{theorem}\label{theorem:apg_error}
In the limit of $\tau \rightarrow 0^+$ the error bound provided in Theorem~\ref{theorem:residualbounds} is less per unit $\tau$ in APG than in Galerkin,
\begin{multline}\label{eq:proof_residerror_conclusion}
 \lim_{\substack{\tau \rightarrow 0^+} } \frac{1}{\tau} \bigg[ \norm{\vcoarsevec^T \mathbf{A} \int_0^{\tau} e^{\Pifine \mathbf{A} \zeta } \Pifine \mathbf{A} \ucoarsevec_F(t - \zeta) d\zeta   - \tau  \vcoarsevec^T \mathbf{A} \Pifine \mathbf{A} \ucoarsevec_F(t)  } + \norm{\vcoarsevec^T \mathbf{A} \int_{\tau}^{t} e^{\Pifine \mathbf{A} \zeta } \Pifine \mathbf{A} \ucoarsevec_F(t - \zeta) d\zeta } \bigg] \le \\
\lim_{\substack{\tau \rightarrow 0^+} } \frac{1}{\tau} \bigg[ \norm{ \vcoarsevec^T\mathbf{A} \int_0^{\tau} e^{\Pifine \mathbf{A} \zeta } \Pifine \mathbf{A} \ucoarsevec_F(t - \zeta) d\zeta } +  \norm{ \vcoarsevec^T\mathbf{A} \int_{\tau}^t e^{\Pifine \mathbf{A} \zeta } \Pifine \mathbf{A} \ucoarsevec_F(t - \zeta) d\zeta } \bigg].
\end{multline}


\begin{comment}

$$
\lim_{\substack{\tau \rightarrow 0^+} }\norm{\bigg[ \int_0^\tau e^{\Pifine \mathbf{A} s}  \Pifine \mathbf{A} \ucoarsevec_F(t-s) ds  - \tau \Picoarse \mathbf{A} \Pifine \mathbf{A} \ucoarsevec_F(t) }\le \lim_{\substack{\tau \rightarrow 0^+} } \norm{\bigg[ \int_0^\tau e^{\Pifine \mathbf{A} s}  \Pifine \mathbf{A} \ucoarsevec_F(t-s) ds } 
$$
\end{comment}

\end{theorem}

\begin{proof}

Define the difference between the APG and Galerkin error incurred per unit $\tau$ as,
\begin{multline}\label{eq:err_diff}
\mathbf{\overline{\delta}}(\tau) \defeq  \frac{1}{\tau} \bigg[ \norm{\vcoarsevec^T \mathbf{A} \int_0^{\tau} e^{\Pifine \mathbf{A} \zeta } \Pifine \mathbf{A} \ucoarsevec_F(t - \zeta) d\zeta   - \tau  \vcoarsevec^T \mathbf{A} \Pifine \mathbf{A} \ucoarsevec_F(t)  } + \norm{\vcoarsevec^T \mathbf{A} \int_{\tau}^{t} e^{\Pifine \mathbf{A} \zeta } \Pifine \mathbf{A} \ucoarsevec_F(t - \zeta) d\zeta } \bigg]  \\ -
 \frac{1}{\tau} \bigg[ \norm{ \vcoarsevec^T\mathbf{A} \int_0^{\tau} e^{\Pifine \mathbf{A} \zeta } \Pifine \mathbf{A} \ucoarsevec_F(t - \zeta) d\zeta } +  \norm{ \vcoarsevec^T\mathbf{A} \int_{\tau}^t e^{\Pifine \mathbf{A} \zeta } \Pifine \mathbf{A} \ucoarsevec_F(t - \zeta) d\zeta } \bigg].
\end{multline}
After canceling terms, we equivalently have
\begin{equation}\label{eq:delta1}
\mathbf{\overline{\delta}}(\tau) \defeq  \frac{1}{\tau}\norm{\vcoarsevec^T \mathbf{A} \int_0^{\tau} e^{\Pifine \mathbf{A} \zeta } \Pifine \mathbf{A} \ucoarsevec_F(t - \zeta) d\zeta   - \tau  \vcoarsevec^T  \mathbf{A} \Pifine \mathbf{A} \ucoarsevec_F(t)  } - \frac{1}{\tau}\norm{ \vcoarsevec^T\mathbf{A} \int_{0}^{\tau} e^{\Pifine \mathbf{A} \zeta } \Pifine \mathbf{A} \ucoarsevec_F(t - \zeta) d\zeta }.
\end{equation}
Apply the rectangle rule to alternatively write the integrals in the above as,
\begin{equation}\label{eq:rec_rule}
\int_0^\tau  e^{\Pifine \mathbf{A} \zeta } \Pifine \mathbf{A} \ucoarsevec_F(t-\zeta) \zeta  = \tau  \Pifine \mathbf{A} \ucoarsevec_F(t)  + \frac{d}{d\zeta} \big(e^{\Pifine \mathbf{A}  \zeta}  \Pifine \mathbf{A} \ucoarsevec_F(t-\zeta)\big)[c] \frac{\tau^2}{2},
\end{equation}
for some $c$ in the interval $[0,\tau$]. Inserting Eq.~\ref{eq:rec_rule} into Eq.~\ref{eq:delta1},
$$\mathbf{\overline{\delta}}(\tau) =  \frac{1}{\tau}\norm{\vcoarsevec^T \mathbf{A}  \frac{d}{d\zeta} \big(e^{\Pifine \mathbf{A}  \zeta}  \Pifine \mathbf{A} \ucoarsevec_F(t-\zeta)\big)[c] \frac{\tau^2}{2}  } - \frac{1}{\tau}\norm{  \tau \vcoarsevec^T \mathbf{A} \Pifine \mathbf{A} \ucoarsevec_F(t)  + \vcoarsevec^T \mathbf{A}  \frac{d}{d\zeta} \big(e^{\Pifine \mathbf{A}  \zeta}  \Pifine \mathbf{A} \ucoarsevec_F(t-\zeta)\big)[c]\frac{\tau^2}{2} } .$$
Factoring out $\tau$, which is non-negative by definition,
$$\mathbf{ \overline{\delta}}(\tau) =  \norm{\vcoarsevec^T \mathbf{A}  \frac{d}{d\zeta} \big(e^{\Pifine \mathbf{A}  \zeta}  \Pifine \mathbf{A} \ucoarsevec_F(t-\zeta)\big)[c] \frac{\tau}{2}  } - \norm{\vcoarsevec^T \mathbf{A}   \Picoarse \mathbf{A} \Pifine \mathbf{A} \ucoarsevec_F(t)  + \vcoarsevec^T \mathbf{A}  \frac{d}{d\zeta} \big(e^{\Pifine \mathbf{A}  \zeta}  \Pifine \mathbf{A} \ucoarsevec_F(t-\zeta)\big)\frac{\tau}{2} }.$$
Taking the limit as $\tau \rightarrow 0^+$ and noting the norm is always non-negative,
$$\lim_{\substack{\tau \rightarrow 0^+}  } \overline{\delta}(\tau) =  - \norm{ \vcoarsevec^T \mathbf{A}  \Picoarse \mathbf{A} \Pifine \mathbf{A} \ucoarsevec_F(t)  } \le 0.$$
From Eq.~\ref{eq:err_diff} and the algebraic limit theorem, this implies that, 
\begin{multline}\label{eq:proof_residerror}
 \lim_{\substack{\tau \rightarrow 0^+} } \frac{1}{\tau} \bigg[ \norm{\vcoarsevec^T \mathbf{A} \int_0^{\tau} e^{\Pifine \mathbf{A} \zeta } \Pifine \mathbf{A} \ucoarsevec_F(t - \zeta) d\zeta   - \tau  \vcoarsevec^T \Picoarse \mathbf{A} \Pifine \mathbf{A} \ucoarsevec_F(t)  } + \norm{\vcoarsevec^T \mathbf{A} \int_{\tau}^{t} e^{\Pifine \mathbf{A} \zeta } \Pifine \mathbf{A} \ucoarsevec_F(t - \zeta) d\zeta } \bigg] \le \\
\lim_{\substack{\tau \rightarrow 0^+} } \frac{1}{\tau} \bigg[ \norm{ \vcoarsevec^T\mathbf{A} \int_0^{\tau} e^{\Pifine \mathbf{A} \zeta } \Pifine \mathbf{A} \ucoarsevec_F(t - \zeta) d\zeta } +  \norm{ \vcoarsevec^T\mathbf{A} \int_{\tau}^t e^{\Pifine \mathbf{A} \zeta } \Pifine \mathbf{A} \ucoarsevec_F(t - \zeta) d\zeta } \bigg].
\end{multline}

\end{proof}
Theorem~\ref{theorem:apg_error} shows that, for sufficiently small $\tau$, an upper bound on the \textit{a priori} error introduced due to the closure problem in APG is \textit{less} than in Galerkin. This suggests that APG may be more accurate than Galerkin. The results of Theorem~\ref{theorem:apg_error} will be discussed more in Section~\ref{sec:theorem_discussion}. While Theorem~\ref{theorem:apg_error} shows that, for sufficiently small $\tau$, APG may be more accurate than Galerkin, it does not give insight into appropriate values of $\tau$. Theorem~\ref{theorem:errorbound_symmetric} and Corollary~\ref{corollary:errorbound_symmetric} seek to provide  insight into this issue.
\begin{theorem}\label{theorem:errorbound_symmetric}
Let $\mathbf{R}(\ufullvec) = \mathbf{A}\ufullvec$ be a linear time-invariant operator. If $\mathbf{A}$ is self-adjoint with all real negative eigenvalues, then $\lambda_{A_i} \ge \lambda_{G_i}$, where $\lambda_{A_1} \ge \lambda_{A_2} \ge \hdots \ge \lambda_{A_K}$ are the $K$ eigenvalues of $\vcoarsevec^T \mathbb{P}_A \mathbf{A} \vcoarsevec $ and $ \lambda_{G_1} \ge  \lambda_{G_2} \ge \hdots \ge \lambda_{G_K}$ are the $K$ eigenvalues of $\vcoarsevec^T \mathbb{P}_G \mathbf{A} \vcoarsevec$.

\end{theorem}

\begin{proof}
Expanding the Adjoint Petrov-Galerkin operator,
$$ \vcoarsevec^T \mathbb{P}_A \mathbf{A} \vcoarsevec = \vcoarsevec^T \Picoarse \mathbf{A} \vcoarsevec + \tau \vcoarsevec^T \Picoarse \mathbf{A} \Pifine \mathbf{A} \vcoarsevec.$$
Noting that the first term on the right-hand side is just the Galerkin term,
$$ \vcoarsevec^T \mathbb{P}_A \mathbf{A} \vcoarsevec = \vcoarsevec^T \mathbb{P}_G \mathbf{A} \vcoarsevec + \tau \vcoarsevec^T \Picoarse \mathbf{A} \Pifine \mathbf{A} \vcoarsevec.$$
As all matrices are self-adjoint, the Weyl inequalities can be used to relate the eigenvalues of $\vcoarsevec^T \mathbb{P}_A \mathbf{A} \vcoarsevec$ to $\vcoarsevec^T \mathbb{P}_G\mathbf{A} \vcoarsevec$. For $K \times K$ Hermitian positive semi-definite matrices $\mathbf{M}, \mathbf{H},$ and $\mathbf{P}$, where $\mathbf{M} = \mathbf{H} + \mathbf{P}$, the Weyl inequalities state,
$$\lambda_{M_i} \ge \lambda_{H_i},$$
where $\lambda_{M_1} \ge \lambda_{M_2} \ge \hdots \ge \lambda_{M_K}$ are the $K$ eigenvalues of $\mathbf{M}$ and $\lambda_{H_1} \ge \lambda_{H_2} \ge \hdots \ge \lambda_{H_K}$ are the $K$ eigenvalues of $\mathbf{H}$. In the present context, as $\vcoarsevec^T \Picoarse \mathbf{A} \Pifine \mathbf{A} \vcoarsevec (=\vcoarsevec^T \mathbf{A} \Pifine \mathbf{A} \vcoarsevec$) is positive semi-definite, it follows from the Weyl inequalities that,
$$\lambda_{A_i} \ge \lambda_{G_i}.$$  
\end{proof}

\begin{corollary}\label{corollary:errorbound_symmetric}
Let $\mathbf{R}(\ufullvec) = \mathbf{A}\ufullvec$ be a linear time-invariant operator. If $\mathbf{A}$ is self-adjoint with all real negative eigenvalues, then $\lambda_{A_1} \le 0$ for $\tau \le \frac{|\gamma_1|}{\gamma_N^2}$, where $\gamma_1$ and $\gamma_N$ are the largest and smallest eigenvalues of $\mathbf{A}$, respectively, and $\lambda_{A_1}$ is the largest eigenvalue of $\vcoarsevec^T \mathbb{P}_A \mathbf{A}\vcoarsevec$.

\end{corollary}

\begin{proof}

We again expand the APG projection and simplify,
$$
\vcoarsevec^T \mathbb{P}_A \mathbf{A} \vcoarsevec = \vcoarsevec^T \Picoarse (\mathbf{I} + \tau A \Pifine) \mathbf{A} \vcoarsevec = \vcoarsevec^T (\mathbf{A} + \tau \mathbf{A} \Pifine \mathbf{A}) \vcoarsevec
$$
Now, let $\lambda_{A_1} \ge \lambda_{A_2} \ge \hdots \ge \lambda_{A_K}$ be the $K$ eigenvalues of $\vcoarsevec^T (\mathbf{A} + \tau \mathbf{A} \Pifine \mathbf{A}) \vcoarsevec$, and let  $\kappa_1 \ge \kappa_2 \ge \hdots \ge \kappa_N$ be the $N$ eigenvalues of $\mathbf{A} + \tau \mathbf{A} \Pifine \mathbf{A}$. Recognizing that $\vcoarsevec$ is semi-orthogonal, we apply the Poincar\'e separation theorem, 
\begin{equation}
    \kappa_i \ge \lambda_{A_i} \ge \kappa_{N-K+i}, \quad \text{for } i = 1,2,\hdots,K
\end{equation}
To ensure that $\lambda_{A_1} \le 0$, we need only show under what conditions $\kappa_1 \le 0$. Thus, we seek bounds on $\tau$ such that $\kappa_1 \le 0$. 

As $\mathbf{A}$ and $\tau \mathbf{A}\Pifine \mathbf{A}$ are both self-adjoint, we can use the Weyl inequalities to bound the smallest eigenvalue of $\mathbf{A} + \tau \mathbf{A} \Pifine \mathbf{A}$. Let $\chi_1 \ge \chi_2 \ge \hdots \ge \chi_N$ be the $N$ eigenvalues of  $\mathbf{A} \Pifine \mathbf{A}$, thus,
\begin{equation}\label{eq:weyl_large}
\kappa_1 \le \gamma_1 + \tau \chi_1,
\end{equation}
To ensure that $\kappa_1 \le 0$, we specify that $\lambda_1 + \tau \chi_1 \le 0$. Noting that $\gamma_1 \le 0$ leads to the inequality,
\begin{equation}\label{eq:ineq_1}
    \lambda_{A_1} \le 0; \qquad \forall \; \tau \le \frac{|\gamma_1|}{\chi_1}
\end{equation}

To write the bounds on $\tau$ in Eq.~\ref{eq:ineq_1} into a more intuitive form, the $\chi_i$ can be related to $\lambda_i$ through the the Weyl inequalities. Recall that $\chi_i$ are the eigenvalues of,
$$\mathbf{A} \Pifine \mathbf{A}  =  \mathbf{A}^2 - \mathbf{A} \Picoarse \mathbf{A}.$$
As $\mathbf{A} \Picoarse \mathbf{A}$ is positive semi-definite, one has,
\begin{equation*}\label{eq:weyl3}
\chi_1 \le \gamma_N^2.
\end{equation*}
Inserting the above inequality into Eq.~\ref{eq:ineq_1} obtains the desired result,
\begin{equation}\label{eq:tau_bound}
\lambda_{A_1} \le 0; \qquad \forall \; \tau \le \frac{| \gamma_1 | }{\gamma_N^2}.
\end{equation}
This upper bound on $\tau$ is more conservative than Eq.~\ref{eq:ineq_1}, though both are equally valid.

\end{proof}

\subsubsection{Discussion}\label{sec:theorem_discussion}
Theorems~\ref{theorem:apg_error} and~\ref{theorem:errorbound_symmetric} contain three interesting results that are worth discussing. First, as discussed in Corollary~\ref{corollary:resid_bound}, Theorem~\ref{theorem:apg_error} shows that, in the limit $\tau \rightarrow 0^+$, the upper-bound provided in Theorem~\ref{theorem:residualbounds} on the error introduced at time $t$ in the APG ROM due to the fine-scales is less than that introduced in the Galerkin ROM (per unit $\tau$). This is an appealing result as APG is derived as a subgrid-scale model. Two remarks are worth making regarding this result. First, while the result was demonstrated in the limit $\tau \rightarrow 0^+$, it will hold so long as the norm of the truncation error in the quadrature approximation is less than the norm of the integral it is approximating; i.e. the approximation is doing a better job than neglecting the integral entirely. Second, the result derived in Theorem~\ref{theorem:apg_error} does not \textit{directly} translate to showing that,
\begin{equation}\label{eq:resid_discussion}
\norm{\vcoarsevec^T \mathbb{P}_A \mathbf{r}_F(\ucoarsevec_F(t)) } \le \norm{\vcoarsevec^T \mathbb{P}_G \mathbf{r}_F(\ucoarsevec_F(t)) }.
\end{equation}
This is a consequence of Theorem~\ref{theorem:residualbounds}, in which the integral that defines the fine-scale solution was split into two intervals. The APG ROM attempts to approximate the first integral, while the Galerkin ROM ignores both terms. Theorem~\ref{theorem:apg_error} showed that, in the limit $\tau \rightarrow 0^+$, the APG approximation to the first integral is better than in the case of Galerkin (i.e., the APG approximation is better than no approximation). The only time that the result provided in Theorem~\ref{theorem:apg_error} will \textit{not} translate to APG providing a better approximation to the \textit{entire} integral (and thus proving Eq.~\ref{eq:resid_discussion}), is when integration over the second interval ``cancels out" the integration over the first interval. To make this idea concrete, consider the integral,
$$\int_0^{2 \pi} \sin(x) dx = \int_0^{\pi} \sin(x) dx  + \int_{\pi}^{2 \pi} \sin(x) dx .$$
Clearly, $\int_0^{2 \pi} \sin(x) dx = 0$. If the entire integral was approximated using just the interval $0$ to $\pi$ (which is analogous to APG), then one would end up with the approximation  $\int_0^{2 \pi} \sin(x) dx \approx \int_0^{\pi} \sin(x) dx = 2$. Alternatively, if one were to ignore the integral entirely (which is analogous to Galerkin) and make the approximation $\int_0^{2 \pi} \sin(x) dx \approx 0$  (which in this example is exact), a better approximation would be obtained.

 The next interesting result is presented in Theorem~\ref{theorem:errorbound_symmetric}, where it is shown that for a self-adjoint system with negative eigenvalues, the eigenvalues associated with the APG ROM error equation are \textit{greater} than the Galerkin ROM. This implies that APG is \textit{less} dissipative than Galerkin, and means that errors may be slower to decay in time. Thus, while Theorem~\ref{theorem:apg_error} shows that the \textit{a priori} contributions to the error due to the closure problem may be smaller in the APG ROM than in the Galerkin ROM, the errors that \textit{are} incurred may be slower to decay. Finally, Corollary~\ref{corollary:errorbound_symmetric} shows that, for self-adjoint systems with negative eigenvalues, the bounds on the parameter $\tau$ such that all eigenvalues associated with the evolution of the error in APG remain negative depends on the spectral content of the Jacobian of $\mathbf{A}$. This observation has been made heuristically in Ref~\cite{parishMZ1}. Although the upper bound on $\tau$ in Eq.~\ref{eq:tau_bound} is very conservative due to the repeated use of inequalities, it provides insight into the selection and behaviour of $\tau$.

\subsection{Selection of Memory Length $\tau$}\label{sec:selecttau}
The APG method requires the specification of the parameter $\tau$. Theorem~\ref{theorem:errorbound_symmetric} showed that, for a self-adjoint linear system, bounds on the value of $\tau$ are related to the eigenvalues of the Jacobian of the full-dimensional right-hand side operator. While such bounds provide intuition into the behavior of $\tau$, they are not particularly useful in the selection of an optimal value of $\tau$ as they 1.) are conservative due to repeated use of inequalities and 2.) require the eigenvalues of the full right-hand side operator, which one does not have access to in a ROM. Further, the bounds were derived for a self-adjoint linear system, and the extension to non-linear systems is unclear. 

In practice, it is desirable to obtain an expression for $\tau$ using only the coarse-scale Jacobian, $\vcoarsevec^T \mathbf{J}[\ucoarsevec]\vcoarsevec$.
In Ref.~\cite{parishMZ1}, numerical evidence showed a strong correlation between the optimal value of $\tau$ and this coarse-scale Jacobian. Based on this numerical evidence and the analysis in the previous section, the following heuristic for selecting $\tau$ is used:
\begin{equation}\label{eq:taueq}
\tau = \frac{C}{\rho (\vcoarsevec^T \mathbf{J}[\ucoarsevec] \vcoarsevec)},
\end{equation}
where $C$ is a model parameter and $\rho(\cdot)$ indicates the spectral radius. In Ref.~\cite{parishMZ1}, $C$ was reported to be $0.2.$  In the numerical experiments presented later in this manuscript, the sensitivity of APG to the value of $\tau$ and the validity of Eq.~\ref{eq:taueq} are examined. 

Similar to the selection of $\tau$ in the APG method, the LSPG method requires the selection of an appropriate time-step~\cite{carlberg_lspg_v_galerkin}. In practice, this fact can be problematic as finding an optimal time-step for LSPG which minimizes error may result in a small time-step and, hence, an expensive simulation. The selection of the parameter $\tau$, on the other hand, does not impact the computational cost of the APG ROM.

\section{Implementation and Computational Cost of the Adjoint Petrov--Galerkin Method}\label{sec:cost}

This section details the implementation of the Adjoint Petrov--Galerkin ROM for simple time integration schemes. Algorithms for explicit and implicit time integration schemes are provided, and the approximate cost of each method in floating-point operations (FLOPs) is analyzed. Here, a FLOP refers to any floating-point addition or multiplication; no distinction is made between the computational cost of either operation. The notation used here is as follows: $N$ is the full-order number of degrees of freedom, $K$ is the number of modes retained in the POD basis, and $\omega N$ is the number of FLOPs required for one evaluation of the right-hand side, $\mathbf{R}(\ucoarsevec(t))$. For sufficiently complex problems, $\omega$ is usually on the order of $\MC{O}(10) < \omega < \MC{O}(1000)$. The analysis presented in this section does not consider hyper-reduction. The analysis can be approximately extended to hyper-reduction by replacing the full-order degrees of freedom with the dimension of the hyper-reduced right-hand side.\footnote{An accurate cost-analysis for hyper-reduced ROMs should consider over sampling of the right-hand side and the FOM stencil.}

\subsection{Explicit Time Integration Schemes}
This section explores the cost of the APG method within the scope of explicit time integration schemes. For simplicity, the analysis is carried out only for the explicit Euler scheme. The computational cost of more sophisticated time integration methods, such as Runge-Kutta and multistep schemes, is generally a  proportional scaling of the cost of the explicit Euler scheme. Algorithm~\ref{alg:alg_apg_exp} provides the step-by-step procedure for performing an explicit Euler update to the Adjoint Petrov--Galerkin ROM. Table~\ref{tab:alg_apg_exp} provides the approximate floating-point operations for the steps reported in Algorithm~\ref{alg:alg_apg_exp}. The algorithm for an explicit update to the Galerkin ROM, along with the associated FLOP counts, is provided in Algorithm~\ref{alg:alg_g_exp} and Table~\ref{tab:alg_g_exp} in Appendix~\ref{appendix:algorithms}. As noted previously, LSPG reverts to the Galerkin method for explicit schemes, and so is not detailed in this section. Table~\ref{tab:alg_apg_exp} shows that, in the case that $K \ll N$ (standard for a ROM) and $\omega \gg 1$ (sufficiently complex right-hand side), the Adjoint Petrov--Galerkin ROM is approximately twice as expensive as the Galerkin ROM.


\begin{algorithm}
\caption{Algorithm for an explicit Euler update for the APG ROM}
\label{alg:alg_apg_exp}
Input: $\acoarsevec^n$\;
\newline
Output: $\acoarsevec^{n+1}$\;
\newline
Steps:
\begin{enumerate}
\item Compute the state from the generalized coordinates, $\ucoarsevec^n =\vcoarsevec \acoarsevec^{n+1}$
\item Compute the right-hand side from the state, $\mathbf{R}(\ucoarsevec^n)$
\item Compute the projection of the right-hand side, $\Picoarse \mathbf{R}(\ucoarsevec^n) = \vcoarsevec \vcoarsevec^T \mathbf{R}(\ucoarsevec^n)$
\item Compute the orthogonal projection of the right-hand side, $\Pifine \mathbf{R}(\ucoarsevec^n) =  \mathbf{R}(\ucoarsevec^n) -  \Picoarse \mathbf{R}(\ucoarsevec^n)$
\item Compute the action of the Jacobian on $\Pifine \mathbf{R}(\ucoarsevec^n)$ using either of the two following strategies:
    \begin{enumerate}
    \item Finite difference approximation:
    \begin{equation*}
    \mathbf{J}[\ucoarsevec^n] \Pifine \mathbf{R}(\ucoarsevec^n) \approx \frac{1}{\epsilon} \Big[ \mathbf{R}\big(\ucoarsevec^n + \epsilon \Pifine  \mathbf{R}(\ucoarsevec^n) \big) - \mathbf{R}(\ucoarsevec^n  ) \Big], 
    \end{equation*}
    where $\epsilon$ is a small constant value, usually  $\sim \MC{O}(10^{-5})$.
    \item Exact linearization:
    \begin{equation*}
    \mathbf{J}[\ucoarsevec^n] \Pifine \mathbf{R}(\ucoarsevec^n) = \mathbf{R'}[\ucoarsevec^n](\Pifine \mathbf{R}(\ucoarsevec^n)),
    \end{equation*}
    where $\mathbf{R}'[\ucoarsevec^n]$ is right-hand side operator linearized about $\ucoarsevec^n$.
    \end{enumerate}
\item Compute the full right-hand side:  $\mathbf{R}(\ucoarsevec^n) +  \tau \mathbf{J}[\ucoarsevec^n] \Pifine \mathbf{R}(\ucoarsevec^n)$

\item Project: $\vcoarsevec^T \bigg[  \mathbf{R}(\ucoarsevec^n) +  \tau \mathbf{J}[\ucoarsevec^n] \Pifine \mathbf{R}(\ucoarsevec^n) \bigg]$
\item Update the state $\acoarsevec^{n+1} = \acoarsevec^n + \Delta t \vcoarsevec^T \bigg[  \mathbf{R}(\ucoarsevec^n) +  \tau\mathbf{J}[\ucoarsevec^n] \Pifine \mathbf{R}(\ucoarsevec^n) \bigg]$
\end{enumerate}
\end{algorithm}

\begin{table}
\begin{tabular}{p{7cm} p{8cm}}
\hline
Step in Algorithm~\ref{alg:alg_apg_exp}& Approximate FLOPs \\
\hline
1    & $2 N K - N$  \\
2    & $\omega N$  \\
3    & $4 N K - N - K$  \\
4    & $N $  \\
5    & $(\omega + 4)N  $  \\
6    & $2N $  \\
7    & $2NK - K $  \\
8    & $2K $  \\
\hline
Total    & $8 N K + (2\omega + 5) N$ \\
Total for Galerkin Method & $4 N K + (\omega-1) N + K$ \\
\hline
\end{tabular}
\caption{Approximate floating-point operations for an explicit Euler update to the Adjoint Petrov--Galerkin method reported in Algorithm~\ref{alg:alg_apg_exp}. The total FLOP count for the Galerkin ROM with an explicit Euler update is additionally reported for comparison. A full description of the Galerkin update is provided in Appendix~\ref{appendix:algorithms}.}
\label{tab:alg_apg_exp}
\end{table}

\subsection{Implicit Time Integration Schemes}
This section evaluates the computational cost of the Galerkin, Adjoint Petrov--Galerkin, and Least-Squares Petrov--Galerkin methods for implicit time integration schemes. For non-linear systems, implicit time integration schemes require the solution of a non-linear algebraic system at each time-step. Newton's method, along with a preferred linear solver, is typically employed to solve the system. For simplicity, the analysis provided in this section is carried out for the implicit Euler time integration scheme along with Newton's method to solve the non-linear system. Before proceeding, the full-order residual, Galerkin residual, and APG residual at time-step $(n + 1)$ are denoted as,
\begin{align*}
 &\mathbf{r}(\vcoarsevec \acoarsevec^{n+1}) = \vcoarsevec \acoarsevec^{n+1} - \vcoarsevec \acoarsevec^n - \Delta t \mathbf{R}(\vcoarsevec \acoarsevec^{n+1}),\\
&\mathbf{r}_G(\acoarsevec^{n+1}) = \acoarsevec^{n+1} - \acoarsevec^n - \Delta t \vcoarsevec^T \mathbf{R}(\vcoarsevec \acoarsevec^{n+1}),\\
&\mathbf{r}_{A}(\acoarsevec^{n+1}) = \acoarsevec^{n+1} - \acoarsevec^n - \Delta t \vcoarsevec^T\bigg[\mathbf{R}(\vcoarsevec \acoarsevec^{n+1}) + \tau \mathbf{J}[\ucoarsevec]\Pifine \mathbf{R}(\vcoarsevec \acoarsevec^{n+1} ) \bigg].
\end{align*}
As the future state, $\acoarsevec^{n+1}$, is unknown, we denote an intermediate state, $\acoarsevec_k$, that is updated after every Newton iteration until some convergence criterion is met. 
Newton's method is defined by the iteration,
\begin{equation}\label{eq:newton_linear}
    \frac{\partial \mathbf{r}(\acoarsevec_k)}{\partial \acoarsevec_k} \big[ \acoarsevec_{k+1} - \acoarsevec_k\big] = - \mathbf{r}(\acoarsevec_k).
\end{equation}
Newton's method solves Eq.~\ref{eq:newton_linear} for the change in the state, $\acoarsevec_{k+1} - \acoarsevec_k$, for $k = 1,2,\hdots$, until the residual converges to a sufficiently-small number. For a ROM, both the assembly and solution of this linear system is the dominant cost of an implicit method.

Two methods are considered for the solution to the non-linear algebraic system arising from implicit time discretizations of the G and APG ROMs: Newton's method with direct Gaussian elimination and Jacobian-Free Newton-Krylov GMRES. The Gauss-Newton method with Gaussian elimination is considered for the solution to the least-squares problem arising in LSPG.

Algorithm~\ref{alg:alg_apg_imp} provides the step-by-step procedures for performing an implicit Euler update to the Adjoint Petrov--Galerkin ROM with the use of Newton's method and Gaussian elimination. Table~\ref{tab:alg_apg_imp} provides the approximate floating-point operations for the steps reported in these algorithms. Analogous results for the Galerkin and LSPG ROMs are reported in Algorithms~\ref{alg:alg_g_imp} and~\ref{alg:alg_LSPG}, and Tables~\ref{tab:alg_g_imp} and~\ref{tab:alg_LSPG} in Appendix~\ref{appendix:algorithms}.
In the limit that $K \ll N$ and $\omega \gg 1$, the total FLOP counts reported show that APG is twice as expensive as both the LSPG and Galerkin ROMs. It is observed that the dominant cost for all three methods lies in the computation of the low-dimensional residual Jacobian. Computation of the low-dimensional Jacobian requires $K$ evaluations of the unsteady residual. Depending on values of $\omega$, $N$, and $K$, this step can consist of over $50\%$ $(K \ll N)$ of the CPU time.\footnote{It is noted that the low-dimensional Jacobian can be computed in parallel.}

To avoid the cost of computing the low-dimensional Jacobian required in the linear solve at each Netwon step, the Galerkin and APG ROMs can make use of Jacobian-Free Netwon-Krylov (JFNK) methods to solve the linear system, opposed to direct methods such as Gaussian elimination. JFNK methods are iterative methods that allow one to circumvent the expense associated with computing the full low-dimensional Jacobian. Instead, JFNK methods only compute the \textit{action} of the Jacobian on a vector at each iteration of the linear solve. This can drastically decrease the cost of the implicit solve. JFNK utilizing the Generalized Minimal Residual (GMRES) method~\cite{gmres}, for example, is guaranteed to converge to the solution $\acoarsevec_k$ in at most $K$ iterations. It takes $K$ residual evaluations just to form the Jacobian required for direct methods.

The LSPG method is formulated as a non-linear least-squares problem. The use of Jacobian-free methods to solve non-linear least-squares problems is significantly more challenging. The principle issue encountered in attempting to use Jacobian-free methods for such applications it that one requires the action of the \textit{transpose} of the residual Jacobian on a vector. This quantity cannot be computed via a standard finite difference approximation or linearization. It is only recently that true Jacobian-free methods have been utilized for solving non-linear least-squares problems. In Ref~\cite{nlls_JacobianFree}, for example, automatic differentiation is utilized to compute the action of the transposed Jacobian on a vector. Due to the challenges associated with Jacobian-free methods for non-linear least-squares problems, this method is not considered here as a solution technique for LSPG.

Algorithm~\ref{alg:alg_apg_jfnk} and Table~\ref{tab:alg_apg_jfnk} report the algorithm and FLOPs required for an implicit Euler update to APG using JFNK GMRES. The term $\eta \le K$ is the number of iterations needed for convergence of the GMRES solver at each Newton iteration. For a concise presentation, the same update for the Galerkin ROM is not presented. Figure~\ref{fig:implicitcost} shows the ratio of the cost of the various implicit ROMs as compared to the Galerkin ROM solved with Gaussian elimination. The standard LSPG method is seen to be approximately the same cost of Galerkin, while APG is seen to be approximately $2$x the cost of Galerkin. The success of the JFNK methods depends on the number of GMRES iterations required for convergence. If $\eta = K$, which is the maximum number of iterations required for GMRES, the cost of JFNK methods is seen to be the same as their direct-solve counterparts. For cases where JFNK converges at a rate of $\eta < K,$ the iterative methods out-perform their direct-solve counterparts.

The analysis presented here shows that, for a given basis dimension, the Adjoint Petrov--Galerkin ROM is approximately twice the cost of the Galerkin ROM for both implicit and explicit solvers. In the implicit case, the APG ROM utilizing a direct linear solver is approximately 2x the cost of LSPG. It was highlighted, however, that APG can be solved via JFNK methods. For cases where one either doesn't have access to the full Jacobian, or the full Jacobian can't be stored, JFNK methods can significantly decrease the ROM cost. The use of JFNK methods within the LSPG approach is more challenging due to the presence of the transpose of the residual Jacobian. Lastly it is noted that, although hyper-reduction can decrease the cost of a residual evaluation, it does not entirely alleviate the cost of forming the Jacobian.

\begin{figure}
\begin{center}
\begin{subfigure}[t]{0.65\textwidth}
\includegraphics[trim={0cm 0cm 0cm 0cm},clip,width=1.\linewidth]{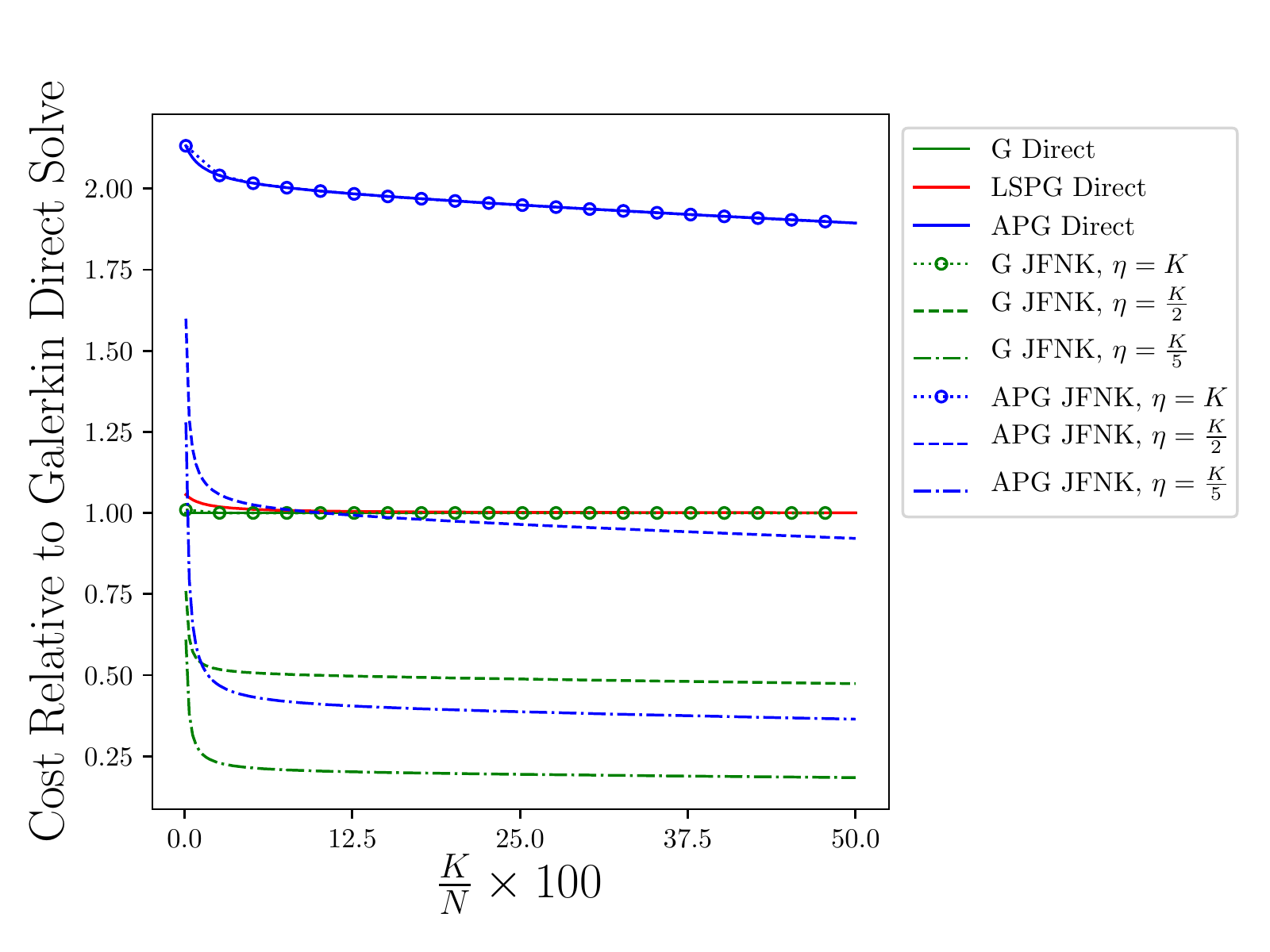}
\end{subfigure}
\end{center}
\caption{Estimates of the G, APG, and LSPG reduced-order models for an implicit Euler update. This plot is generated for values of $N=1000$ and $\omega = 50$, and $\eta = \{K,K/2,K/5\},$ where $\eta$ is the total number of iterations required for the GMRES solver at each Newton step.}
\label{fig:implicitcost}
\end{figure}

\begin{algorithm}
\caption{Algorithm for an implicit Euler update for the APG ROM using Newton's Method with Gaussian Elimination}
\label{alg:alg_apg_imp}
Input: $\acoarsevec^n$, residual tolerance $\xi$ \;
\newline
Output: $\acoarsevec^{n+1}$\;
\newline
Steps:
\begin{enumerate}
\item Set initial guess, $\acoarsevec_k$
\item  Loop while $\mathbf{r}^k > \xi$
\begin{enumerate}
    \item Compute the state from the generalized coordinates, $\ucoarsevec_k = \vcoarsevec \acoarsevec_k$
    \item Compute the right-hand side from the full state, $\mathbf{R}(\ucoarsevec_k)$
    \item Compute the projection of the right-hand side, $\Picoarse \mathbf{R}(\ucoarsevec^n) = \vcoarsevec \vcoarsevec^T \mathbf{R}(\ucoarsevec^n)$
    \item Compute the orthogonal projection of the right-hand side, $\Pifine \mathbf{R}(\ucoarsevec_k) =  \mathbf{R}(\ucoarsevec_k) -  \Picoarse \mathbf{R}(\ucoarsevec_k)$
    \item Compute the action of the right-hand side Jacobian on $\Pifine \mathbf{R}(\ucoarsevec_k)$, as in Alg.~\ref{alg:alg_apg_exp}.
    \item Compute the modified right-hand side, $ \mathbf{R}(\vcoarsevec \acoarsevec_k) + \tau \mathbf{J}[\ucoarsevec]\Pifine \mathbf{R}(\vcoarsevec \acoarsevec_k )$
    \item Project the modified right-hand side,  $\vcoarsevec^T\Big[\mathbf{R}(\vcoarsevec \acoarsevec_k) + \tau \mathbf{J}[\ucoarsevec]\Pifine \mathbf{R}(\vcoarsevec \acoarsevec_k ) \Big]$
    \item Compute the APG residual, $\mathbf{r}_A(\acoarsevec_k) = \acoarsevec_k - \acoarsevec^n - \Delta t \vcoarsevec^T\Big[\mathbf{R}(\vcoarsevec \acoarsevec_k) + \tau \mathbf{J}[\ucoarsevec]\Pifine \mathbf{R}(\vcoarsevec \acoarsevec_k ) \Big]$
    \item Compute the residual Jacobian, $\frac{\partial \mathbf{r}(\acoarsevec_k)}{\partial \acoarsevec_k}$
    \item Solve the linear system via Gaussian Elimination: $\frac{\partial \mathbf{r}(\acoarsevec_k)}{\partial \acoarsevec_k} \Delta  \acoarsevec = - \mathbf{r}(\acoarsevec_k)$
    \item Update the state: $\acoarsevec_{k+1} = \acoarsevec_k + \Delta \acoarsevec$
    \item $k = k + 1$
\end{enumerate}
\item Set final state, $\acoarsevec^{n+1} = \acoarsevec_k$
\end{enumerate}
\end{algorithm}

\begin{table}[]
\centering
\begin{tabular}{p{7cm} p{8cm}}
\hline
Step in Algorithm~\ref{alg:alg_apg_imp}& Approximate FLOPs \\
\hline
2a    & $2 N K - N $ \\
2b    & $ \omega N $ \\
2c    & $4 N K - N - K$  \\
2d    & $ N $  \\
2e    & $ (\omega + 4) N $ \\
2f    & $ 2N $ \\
2g    & $ 2NK - K $ \\
2h    & $ 3K $ \\
2i    & $ (2\omega + 5) NK + K^2 + 8NK^2 $ \\
2j    & $ K^3 $ \\
2k    & $ K $ \\
\hline
Total & $ (2\omega + 5)N + 2K + (2\omega + 13) NK + K^2 + 8NK^2 + K^3 $ \\
Galerkin ROM FLOP count & $ (\omega - 1)N + 3K + (\omega + 3)NK + 2K^2 + 4NK^2 + K^3 $ \\
LSPG ROM FLOP count& $ (\omega + 2)N + (\omega + 6) NK - K^2 + 4NK^2 + K^3 $ \\

\end{tabular}
\caption{Approximate floating-point operations for one Newton iteration for the implicit Euler update to the Adjoint Petrov--Galerkin method reported in Algorithm~\ref{alg:alg_apg_imp}. FLOP counts for the Galerkin ROM and LSPG ROM with an implicit Euler update are additionally reported for comparison. A full description of the Galerkin and LSPG ROM updates are provided in Appendix~\ref{appendix:algorithms}}
\label{tab:alg_apg_imp}
\end{table}

\begin{algorithm}
\caption{Algorithm for an implicit Euler update for the APG ROM using JFNK GMRES}
\label{alg:alg_apg_jfnk}
Input: $\acoarsevec^n$, residual tolerance $\xi$ \;
\newline
Output: $\acoarsevec^{n+1}$\;
\newline
Steps:
\begin{enumerate}
\item Set initial guess, $\acoarsevec_k$
\item  Loop while $\mathbf{r}^k > \xi$
\begin{enumerate}

    \refstepcounter{enumii}\item[$(a\text{--}h)$] Compute steps 2a through 2h in Algorithm~\ref{alg:alg_apg_imp}
    \setcounter{enumii}{8}
    \item Solve the linear system,  $\frac{\partial \mathbf{r}(\acoarsevec_k)}{\partial \acoarsevec_k} \Delta \acoarsevec_k = \mathbf{r}_A(\acoarsevec_k)$ using Jacobian-Free GMRES
    \item Update the state: $\acoarsevec_{k+1} = \acoarsevec_k + \Delta \acoarsevec$
    \item $k = k + 1$
\end{enumerate}
\item Set final state, $\acoarsevec^{n+1} = \acoarsevec_k$
\end{enumerate}
\end{algorithm}

\begin{table}[]
\centering
\begin{tabular}{p{6cm} p{9cm}}
\hline
Step in Algorithm~\ref{alg:alg_apg_jfnk}& Approximate FLOPs \\
\hline
2a    & $2 N K - N $ \\
2b    & $ \omega N $ \\
2c    & $4 N K - N - K$  \\
2d    & $ N $  \\
2e    & $ (\omega + 4) N $ \\
2f    & $ 2N $ \\
2g    & $ 2NK - K $ \\
2h    & $ 3K $ \\
2i    & $ (2\omega + 5) N \eta + K \eta  + 8N K \eta  + \eta^2 K $\\
2k    & $ K $ \\
\hline
Total & $ \big((2\eta + 2) \omega + 5\eta + 5) \big)N + (\eta^2 + \eta + 2)K + (8\eta + 8) NK $ \\
\end{tabular}
\caption{Approximate floating-point operations for one Newton iteration for the implicit Euler update to the Adjoint Petrov--Galerkin method using Jacobian-Free GMRES reported in Algorithm~\ref{alg:alg_apg_jfnk}.}
\label{tab:alg_apg_jfnk}
\end{table}

\section{Numerical Examples}\label{sec:numerical}

Applications of the APG method are presented  for ROMs of compressible flows: the 1D Sod shock tube problem and 2D viscous flow over a cylinder. In both problems, the test bases are chosen via POD. The shock tube problem highlights the improved stability and accuracy of the APG method over the standard Galerkin ROM, as well as improved performance over the LSPG method. The impact of the choice of $\tau$ (APG) and $\Delta t$ (LSPG) time-scales are also explored. The cylinder flow experiment examines a more complex problem and assesses the predictive capability of APG in comparison with Galerkin and LSPG ROMs. The effect of the choice of $\tau$ on simulation accuracy is further explored.  

\subsection{Example 1: Sod Shock Tube with reflection}
The first case considered is the Sod shock tube, described in more detail in \cite{sod}. The experiment simulates the instantaneous bursting of a diaphragm separating a closed chamber of high-density, high-pressure gas from a closed chamber of low-density, low pressure gas. This generates a strong shock, a contact discontinuity, and an expansion wave, which reflect off the shock tube walls at either end and interact with each other in complex ways. The system is described by the one-dimensional compressible Euler equations with the initial conditions,
\begin{comment}
\begin{equation}\label{eq:euler_1D}
    \frac{\partial \ufullvec}{\partial t} + \frac{\partial \mathbf{f}}{\partial x} = 0, \quad
    \ufullvec = 
    \begin{Bmatrix} \rho \\ \rho u \\ \rho E \end{Bmatrix}, \quad 
    \mathbf{f} = \begin{Bmatrix} \rho u \\ \rho u^2 + p \\  u(\rho E + p) \end{Bmatrix}.
\end{equation}
The problem setup is given by the initial conditions,
\end{comment} 

\begin{equation*}
\rho = 
\begin{cases} 
      1 & x\leq 0.5 \\
      0.125 & x > 0.5 
   \end{cases},
\qquad
p = 
\begin{cases} 
      1 & x\leq 0.5 \\
      0.1 & x > 0.5 
   \end{cases},
\qquad
u = 
\begin{cases} 
      0 & x\leq 0.5 \\
      0 & x > 0.5 
   \end{cases},
\end{equation*}  
with $x \in [0,1]$. Impermeable wall boundary conditions are enforced at x = 0 and x = 1.

\subsubsection{Full-Order Model}
The 1D compressible Euler equations are solved using a finite volume method and explicit time integration. The domain is partitioned into 1,000 cells of uniform width. The finite volume method uses the first-order Roe flux~\cite{roescheme} at the cell interfaces. A strong stability-preserving RK3 scheme~\cite{SSP_RK3} is used for time integration. The solution is evolved for $t \in [0.0,1.0]$ with a time-step of $\Delta t = 0.0005$, ensuring CFL$\leq 0.75$ for the duration of the simulation. The solution is saved every other time-step, resulting in 1,000 solutions snapshots for each conserved variable.

\subsubsection{Solution of the Reduced-Order Model}
Using the FOM data snapshots, trial bases for the ROMs are constructed via the proper orthogonal decomposition (POD) approach. A separate basis is constructed for each conserved variable. The complete basis construction procedure is detailed in Appendix~\ref{appendix:basisconstruction}. Once a coarse-scale trial basis $\vcoarsevec$ is built, a variety of ROMs are evaluated according to the following formulations:

\begin{enumerate}
\item Galerkin ROM:
\begin{equation*}\label{eq:galerkin_ROM}
\vcoarsevec^T \bigg( \frac{d \ucoarsevec}{dt} - \mathbf{R}(\ucoarsevec) \bigg) = 0, \qquad t \in [0,1].
\end{equation*}

\item Adjoint Petrov--Galerkin ROM:
\begin{equation*}\label{eq:MZPG_ROM}
\vcoarsevec^T\bigg(\mathbf{I} + \tau  \mathbf{J}[\ucoarsevec] \Pifine \bigg)  \bigg( \frac{d \ucoarsevec}{dt} - \mathbf{R}(\ucoarsevec) \bigg) = 0 , \qquad t \in [0,1].
\end{equation*}
\textit{Remark: The Adjoint Petrov--Galerkin ROM requires specification of $\tau$.}

\item Least-Squares Petrov--Galerkin ROM (Implicit Euler Time Integration):
\begin{equation*}
\ufullvec^n = \underset{\mathbf{y} \in \text{Range}(\vcoarsevec)  }{\text{arg min}}\norm{ \frac{\mathbf{y} - \ucoarsevec^{n-1}}{\Delta t} - \mathbf{R}(\mathbf{y}) }^2, \qquad \text{for } n = 1,2,\hdots , \text{ceil}\big(\frac{1}{\Delta t}\big).
\end{equation*}
\textit{Remark: The LSPG approach is strictly coupled to the time integration scheme and time-step.}


\end{enumerate}

\subsubsection{Numerical Results}
The first case considered uses $50$ basis vectors each for the conserved variables $\rho, \rho u,$ and $\rho E$. The total dimension of the reduced model is thus $K = 150$. Roughly 99.9--99.99\% of the POD energy is captured by this 150-mode basis. In fact, 99\% of the energy is contained in the first 5-12 modes of each conserved variable. 

The Adjoint Petrov--Galerkin ROM requires specification of the memory length $\tau$. Similarly, LSPG requires the selection of an appropriate time-step. The sensitivity of both methods to this selection will be discussed later in this section. 
The simulation parameters are provided in Table~\ref{tab:sod_tab1}.

Density profiles at $t = 0.25$ and $t = 1.0$ for explicit Galerkin and APG ROMs, along with an implicit LSPG ROM, are displayed in Fig.~\ref{fig:sod_density}. All three ROMs are capable of reproducing the shock tube density profile in Fig.~\ref{fig:sod_density_0p25}; a normal shock propagates to the right and is followed closely behind by a contact discontinuity, while an expansion wave propagates to the left. All three methods exhibit oscillations at $x = 0.5$, the location of the imaginary burst diaphragm, and near the shock at $x = 0.95$. At $t = 1.0$, when the shock has reflected from the right wall and interacted with the contact discontinuity, much stronger oscillations are present, particularly near the reflected shock at $x = 0.45$. These oscillations are reminiscent of Gibbs phenomenon, and are an indicator of the inability to accurately reconstruct sharp gradients. The Galerkin ROM exhibits the largest oscillations of the ROMs considered, while LSPG exhibits the smallest.

\begin{table}
\centering
\begin{tabular}{ l  l  l  l  l}\hline
  ROM Type & Time Scheme & $\Delta t$ & $\tau$ & $\int ||e||_2 dt$\\ \hline
  Galerkin  & SSP-RK3    & 0.0005 & N/A     & 1.5752 \\
  Galerkin  & Imp. Euler & 0.0005 & N/A     & 1.0344 \\
  APG       & SSP-RK3    & 0.0005 & 0.00043 & 1.0637 \\
  APG       & Imp. Euler & 0.0005 & 0.00043 & 0.8057 \\
  APG       & Imp. Euler & 0.001  & 0.00043 & 0.7983 \\
  LSPG      & Imp. Euler & 0.0005 & N/A     & 1.1668 \\
  LSPG      & Imp. Euler & 0.001  & N/A     & 1.4917 \\ \hline
\end{tabular}
\caption{Computational details for Sod shock tube ROM cases, $K = 150$}
\label{tab:sod_tab1}
\end{table}

\begin{comment}
\begin{figure}
    \centering
    \begin{minipage}{0.49\linewidth}
    \includegraphics[trim={0cm 0cm 0cm 0cm},clip,width=1.\linewidth]{sodFigs/PODspectrumZoomed.png}
    \caption{Sod shock tube POD energy spectrum}
    \label{fig:pod_spectrum}
    \end{minipage}\hfill
    \begin{minipage}{0.49\linewidth}
    \includegraphics[trim={0cm 0cm 0cm 0cm},clip,width=1.\linewidth]{sodFigs/err_consVars_k50_apgExplicit.png}
    \caption{Conserved variable error profiles, APG w/ SSP-RK3 time integration, $K = 150$, $\Delta t = 0.0005$}
    \label{fig:sod_error_consVars}
    \end{minipage}\hfill
\end{figure}
\end{comment}

\begin{figure}
\begin{center}
\begin{subfigure}[t]{0.49\textwidth}
\includegraphics[trim={0cm 0cm 0cm 0cm},clip,width=1.\linewidth]{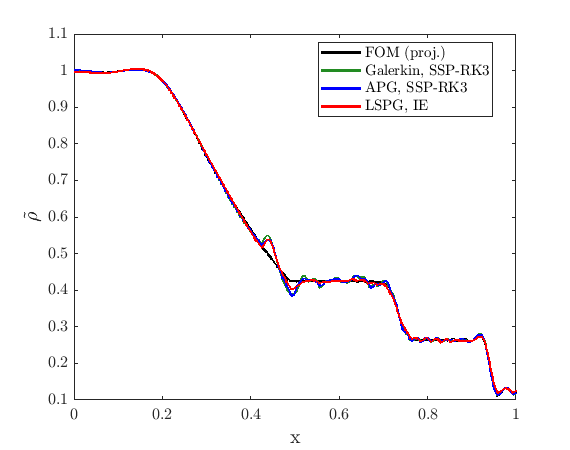}
\caption{$t=0.25$.}
\label{fig:sod_density_0p25}
\end{subfigure}
\begin{subfigure}[t]{0.49\textwidth}
\includegraphics[trim={0cm 0cm 0cm 0cm},clip,width=1.\linewidth]{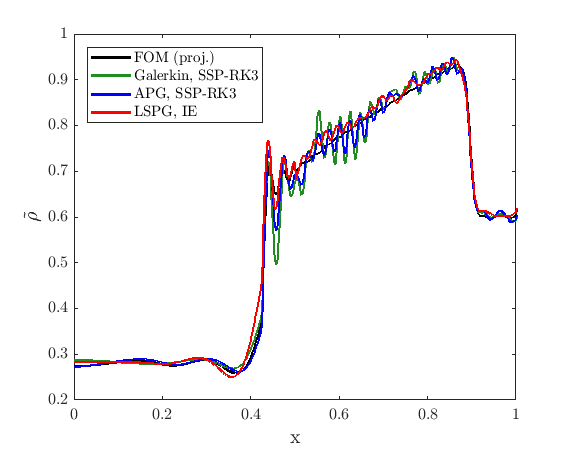}
\caption{$t=1$.}
\label{fig:sod_density_1p0}
\end{subfigure}
\end{center}
\caption{Density profiles for the Sod shock tube with $K = 150$, $\Delta t = 0.0005$.}
\label{fig:sod_density}
\end{figure}

Figure~\ref{fig:sod_error} shows the evolution of the error for all of the ROMs listed in Table~\ref{tab:sod_tab1}. The $L^2$-norm of the error is computed as,
\begin{equation*}
||e||_2 = \sqrt{ \sum_{i=1}^{1000} \Big[(\tilde{\rho}_{i,ROM} - \tilde{\rho}_{i,FOM} )^2 + (\widetilde{\rho u}_{i,ROM} - \widetilde{\rho u}_{i,FOM} )^2  + (\widetilde{\rho E}_{i,ROM} - \widetilde{\rho E}_{i,FOM})^2 \Big] }.
\end{equation*}
 Here, the subscript $i$ denotes each finite volume cell. The FOM values used for error calculations are projections of the FOM data onto $\Vcoarse$, e.g. $\tilde{\rho}_{FOM} = \Picoarse \rho_{FOM}$. This error measure provides a fair upper bound on the accuracy of the ROMs, as the quality of the ROM is generally dictated by the richness of the trial basis and the projection of the FOM data is the maximum accuracy that can be reasonably hoped for.

In Figure~\ref{fig:sod_error}, it is seen that the APG ROM exhibits improved accuracy over the Galerkin ROM. The LSPG ROM for $\Delta t = 0.0005$ performs slightly better than the explicit Galerkin ROM, and worse than the implicit Galerkin ROM. Increasing the time-step to $\Delta t = 0.001$ results in a significant increase in error for the LSPG ROM. This is due to the fact that the performance of LSPG is influenced by the time-step. For a trial basis containing much of the residual POD energy, LSPG will generally require a very small time-step to improve accuracy; this sensitivity will be explored later. Lastly, it is observed that the APG ROM is \textit{not} significantly affected by the time-step. The APG ROM with $\Delta t = 0.001$ shows moderately increased error prior to $t = 0.3$ and similar error afterwards when compared against the $\Delta t = 0.0005$ APG ROM case. 

\begin{figure}
\begin{center}
\begin{subfigure}[t]{0.49\textwidth}
\includegraphics[trim={0cm 0cm 0cm 0cm},clip,width=1.\linewidth]{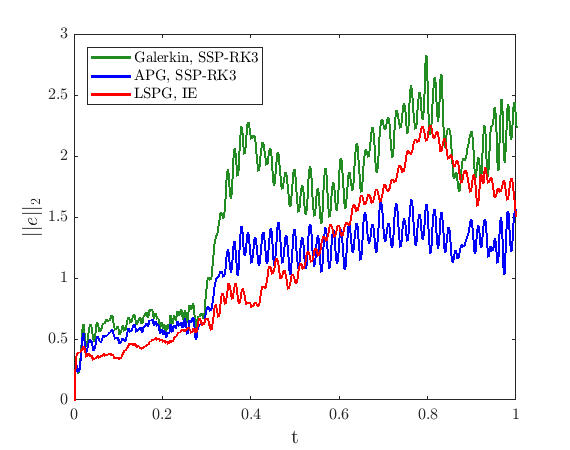}
\caption{Explicit Galerkin/APG, implicit LSPG, $\Delta t = 0.0005$}
\label{fig:sod_err_explicit}
\end{subfigure}
\begin{subfigure}[t]{0.49\textwidth}
\includegraphics[trim={0cm 0cm 0cm 0cm},clip,width=1.\linewidth]{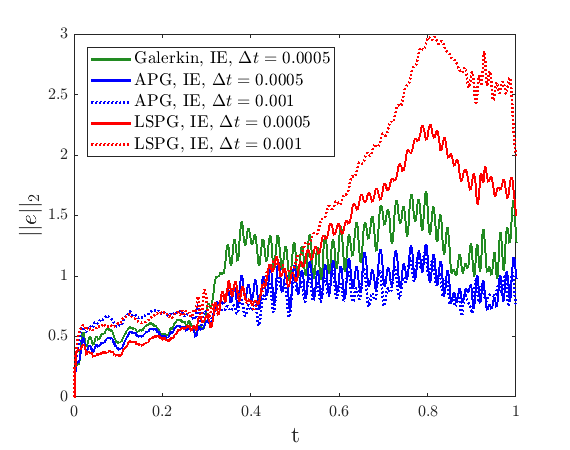}
\caption{All implicit, various $\Delta t$}
\label{fig:sod_err_implicit}
\end{subfigure}
\end{center}
\caption{$L^2$-norm error profiles for the Sod shock tube with $150$ basis vectors.}
\label{fig:sod_error}
\end{figure}

Figure~\ref{fig:sod_mode_study} studies the effect of the number of modes retained in the trial basis on the stability and accuracy, over the range $K = 60,75,90,\ldots,180$. Missing data points indicate an unstable solution. Values of $\tau$ for the APG ROMs are again selected by user choice. The most striking feature of these plots is the fact that even though the explicit Galerkin ROM is unstable for $K \leq 135$ and the implicit Galerkin ROM is unstable for $K \leq 75$, the APG and LSPG ROMs are stable for all cases. Furthermore, the APG and LSPG ROMs are capable of achieving stability with a time-step twice as large as that of the Galerkin ROM. The cost of the APG and LSPG ROMs are effectively halved, but they are still able to stabilize the simulation. Interestingly, the Galerkin and APG ROMs both exhibit abrupt peaks in error at $K = 120$, while the LSPG ROMs do not. The exact cause of this is unknown, but displays that a monotonic decrease in error with enrichment of the trial space is not guaranteed. 

Several interesting comparisons between APG and LSPG arise from Figure~\ref{fig:sod_mode_study}. First, with the exception of the $K = 120$ case, Fig.~\ref{fig:sod_err_explicit} shows that the APG ROM with explicit time integration exhibits accuracy comparable to that of the LSPG ROM with implicit time integration. As can be seen in comparing Tables~\ref{tab:alg_apg_exp} and~\ref{tab:alg_LSPG}, the cost of APG with explicit time integration is significantly lower than the cost of LSPG. This is an attractive feature of APG, as it is able to use inexpensive explicit time integration while LSPG is restricted to implicit methods. Additionally, we draw attention to the poor performance of LSPG at high $K$ for a moderate time-step in Fig.~\ref{fig:sod_modeSens_implicit}. Increasing the time-step to $\Delta t = 0.001$ to decrease simulation cost only exacerbates this issue; as the trial space is enriched, LSPG requires a smaller time-step to yield accurate results. If we wish to improve the LSPG solution for $K = 150$, we must decrease the time-step below that of the FOM. The accuracy of the APG ROM does not change when the time-step is doubled from $\Delta t = 0.0005$ to $\Delta t = 0.001$. This halves the cost of the APG ROM with no significant drawbacks.  

\begin{figure}
\begin{center}
\begin{subfigure}[t]{0.49\textwidth}
\includegraphics[trim={0cm 0cm 0cm 0cm},clip,width=1.\linewidth]{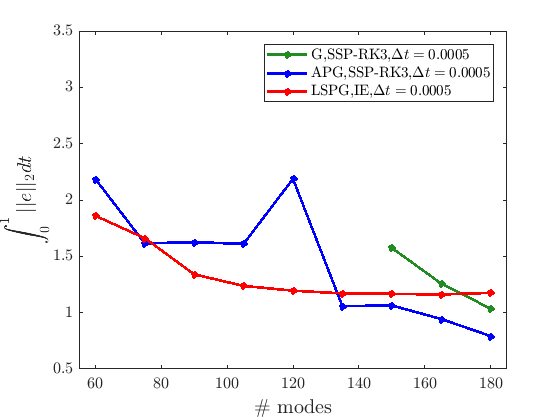}
\caption{Explicit Galerkin/APG, implicit LSPG, fixed $\Delta t$}
\label{fig:sod_modeSens_explicit}
\end{subfigure}
\begin{subfigure}[t]{0.49\textwidth}
\includegraphics[trim={0cm 0cm 0cm 0cm},clip,width=1.\linewidth]{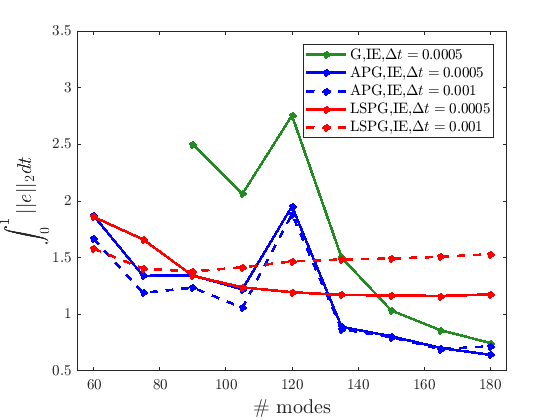}
\caption{All implicit, various $\Delta t$}
\label{fig:sod_modeSens_implicit}
\end{subfigure}
\end{center}
\caption{Integrated error mode sensitivity study for the Sod shock tube.}
\label{fig:sod_mode_study}
\end{figure}

\subsubsection{Optimal Memory Length Investigations}
As mentioned previously, the success of LSPG is tied to the physical time-step and the time integration scheme, while the parameter $\tau$ in the APG method may be chosen independently from these factors. In minimizing ROM error, finding an optimal value of $\tau$ for the APG ROM may permit the choice of a much larger time-step than the optimal LSPG time-step. Further, the APG method may be applied with explicit time integration schemes, which are generally much less expensive than the implicit methods which LSPG is restricted to. To demonstrate this, the APG ROM and LSPG ROM with $K=150$ are simulated for a variety of time scales ($\tau$ for APG and $\Delta t$ for LSPG). 

Figure~\ref{fig:sod_memlength_sensitivity} shows the integrated error of the ROMs versus the relevant time scale. For this case, the optimal value of $\Delta t$ for LSPG is less than $0.0001$ and is not shown. The optimal value of $\tau$ is not greatly affected by the choice of time integration scheme (implicit or explicit) or time-step. Furthermore, because $\tau$ can be chosen independently from $\Delta t$ for APG, the APG ROM can produce low error at a much larger time-step ($\Delta t = 0.001$) than the optimal time-step for the LSPG ROM. This highlights the fact that the ``optimal" LSPG ROM may be computationally expensive due to a small time-step, whereas the ``optimal" APG ROM requires only the specification of $\tau$ and can use, potentially, much larger time-steps than LSPG. It has to be mentioned, however, that the choice of $\tau$  has an impact on performance --- selections of $\tau$ larger than those plotted in Fig.~\ref{fig:sod_memlength_sensitivity} caused the ROM to lose stability. 

Also discussed previously, computing an optimal value of $\tau$ \textit{a priori} may be linked to the spectral radius of the coarse-scale Jacobian, (i.e. $\rho \big( \vcoarsevec^T \mathbf{J}[\ucoarsevec_0]\vcoarsevec \big)$, not to be confused with the physical density $\rho$). We consider the APG ROM for basis sizes of  $K=30,60,90,150,180,240$. For each case, an ``optimal" $\tau$ is found by minimizing the misfit between the ROM solution and the projected FOM solution. The misfit is defined as follows,
\begin{equation}\label{eq:misfit}
\mathcal{J}(\tau) = \sum_{i=1}^{200} ||e(\tau,t = i10\Delta t)||_2.
\end{equation}
Equation~\ref{eq:misfit} corresponds to summing the $L^2$-norm of the error at every $10^{th}$ time-step. Figure~\ref{fig:sod_tau_specRad} shows the resulting optimal $\tau$ for each case plotted against the inverse of the spectral radius of the coarse-scale Jacobian evaluated at $t=0$. 
The strong linear correlation suggests that a near-optimal value of $\tau$ may be chosen by evaluating the spectral radius, $\rho \big( \vcoarsevec^T \mathbf{J}[\ucoarsevec_0]\vcoarsevec \big)$ and using the above linear relationship to $\tau$. 

Two points are emphasized here:
\begin{enumerate}
\item The spectral radius plays an important role in both implicit and explicit time integrators and is often the determining factor in the choice of the time-step. Theoretical analysis on the stability of explicit methods (and convergence of implicit methods) shows a similar dependence to the spectral radius. Choosing the memory length to be $\tau = \Delta t$ is one simple heuristic that may be used.
\item 
While a linear relationship between $\tau$ and the spectral radius of the coarse-scale Jacobian has been observed in every problem the authors have examined, the slope of the fit is somewhat problem dependent. For the purpose of reduced-order modeling, however, this is only a minor inconvenience as an appropriate 
value of $\tau$ can be selected by assessing the performance of the ROM on the training set, i.e., on the simulation used to construct the POD basis. 
\item Finally, more complex methods may be used to compute $\tau$. A method to dynamically compute $\tau$ based on Germano's identity, for instance, was proposed in~\cite{parish_dtau} in the context of the simulation of turbulent flows with Fourier-Galerkin methods. Extension of this technique to projection-based ROMs and the development of additional techniques to select $\tau$ will be the subject of future work.
\end{enumerate}

\begin{figure}
    \centering
    \begin{minipage}{0.49\linewidth}
    \includegraphics[trim={0cm 0cm 0cm 0cm},clip,width=1.\linewidth]{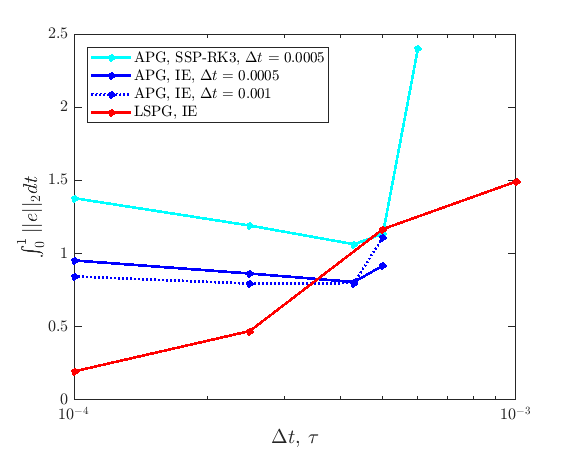}
    \caption{Error as a function of time scale, $K = 150$.}
    \label{fig:sod_memlength_sensitivity}
    \end{minipage}\hfill
    \begin{minipage}{0.49\linewidth}
    \includegraphics[trim={0cm 0cm 0cm 0cm},clip,width=1.\linewidth]{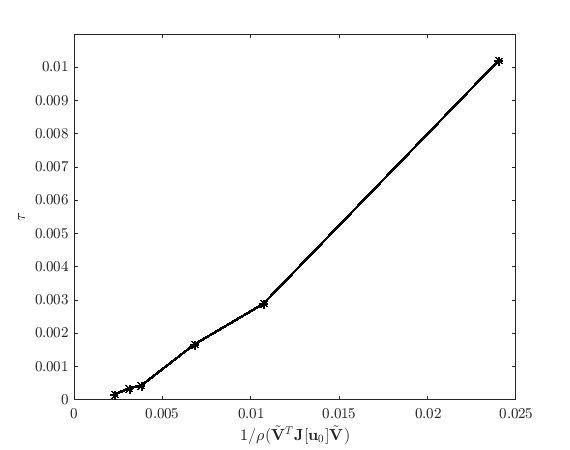}
    \caption{Optimal $\tau$ as a function of the spectral radius evaluated at $t=0$.}
    \label{fig:sod_tau_specRad}
    \end{minipage}\hfill
\end{figure}

\subsection{Example 2: Flow Over Cylinder}\label{sec:cylinder}
The second case considered is viscous compressible flow over a circular cylinder. The flow is described by the two-dimensional compressible Navier-Stokes equations. A Newtonian fluid and a calorically perfect gas are assumed.
\subsubsection{Full-Order Model}\label{sec:cylinder_fom}
The compressible Navier-Stokes equations are solved using a discontinuous Galerkin (DG) method and explicit time integration. Spatial discretization with the discontinuous Galerkin method leads to a semi-discrete system of the form,
\begin{equation*}
 \frac{d \ufullvec}{dt} = \mathbf{M}^{-1}\mathbf{f}(\ufullvec), \qquad \ufullvec(t=0) = \ufullvec_0,
\end{equation*}
where $\mathbf{M}\in \mathbb{R}^{N \times N}$ is a block diagonal mass matrix and $\mathbf{f}(\ufullvec) \in \mathbb{R}^N$ is a vector containing surface and volume integrals. Thus, using the notation defined in Eq.~\ref{eq:FOM}, the right-hand side operator for the DG discretization is defined as,
$$\mathbf{R}(\ufullvec) = \mathbf{M}^{-1}\mathbf{f}(\ucoarsevec).$$

For the flow over cylinder problem considered in this section, a single block domain is constructed in polar coordinates by uniformly discretizing in $\theta$ and by discretizing in the radial direction by,
$$r_{i+1} = r_i + r_i (R_g - 1),$$
where $R_g$ is a stretching factor and is defined by,
$$R_g = r_{max}^{1/N_r}.$$

The DG method utilizes the Roe flux at the cell interfaces and uses the first form of Bassi and Rebay~\cite{BR1} for the viscous fluxes. Temporal integration is again performed using a strong stability preserving RK3 method. Far-field boundary conditions and linear elements are used. Details of the FOM are presented in Table~\ref{tab:cylinder}.

\begin{table}
\centering
\begin{tabular}{|c | c | c | c | c | c | c | c | c | c |}\hline
 $r_{max}$ & $N_r$ & $N_{\theta}$ & $p_r$ & $p_{\theta}$ & $\Delta t$  & Mach & $a_{\infty}$ & $p_{\infty}$ & $T_{\infty}$ \\ \hline
 60 & 80 & 80 &  3 & 3  & $5e-3$ & 0.2 & $1.0$ & $1.0$ & $\gamma^{-1}$ \\\hline
\end{tabular}
\caption{Details used for flow over cylinder problem. In the above table, $N_r$ and $N_{\theta}$ are the number of cells in the radial and $\theta$ direction, respectively. Similarly, $p_r$ and $p_{\theta}$ are the polynomial orders in the radial and $\theta$ direction. Lastly, $a_{\infty}$, $p_{\infty}$, and $T_{\infty}$ are the free-stream speed of sound, pressure, and temperature.}
\label{tab:cylinder}
\end{table}


\subsubsection{Solution of the Full-Order Model and Construction of the ROM Trial Space}\label{sec:cyl_romsteps}
Flow over a cylinder at Re=$100,200,$ and $300$, where Re=$ \rho_{\infty} U_{\infty} D / \mu$ is the Reynolds number, are considered. These Reynolds numbers give rise to the well studied von K\'arm\'an vortex street. Figure~\ref{fig:vonkarman} shows the FOM solution at Re=100 for several time instances to illustrate the vortex street. 

The FOM is used to construct the trial spaces used in the ROM simulations. The process used to construct these trial spaces is as follows:
\begin{enumerate}
    \item Initialize FOM simulations at Reynold's numbers of Re=$100,200,$ and $300.$ The Reynold's number is controlled by raising or lowering the viscosity.
    \item Time-integrate the FOM at each Reynolds number until a statistically steady-state is reached.
    \item Once the flow has statistically converged to a steady state, reset the time coordinate to be $t=0$, and solve the FOM for $t \in [0,100]$.
    \item Take snapshots of the FOM solution obtained from Step 3 at every $t=0.5$ time units over a time-window of $t \in [0,100]$ time units, for a total of $200$ snapshots at each Reynolds number. This time window corresponds to roughly two cycles of the vortex street, with $100$ snapshots per cycle.
    \item  Assemble the snapshots from each case into one global snapshot matrix of dimension $N \times 600$. This snapshot matrix is used to construct the trial subspace through POD. Note that only one set of basis functions for all conserved variables is constructed. 
    \item Construct trial spaces of dimension $N\times 11$, $N\times 43$, and $N \times 87$. These subspace dimensions correspond to an energy criterion of $99, 99.9,$ and $99.99\%$. The different trial spaces are summarized in Table~\ref{tab:rom_basis_1}.
\end{enumerate}
\begin{figure}
\begin{center}
\begin{subfigure}[t]{0.05\textwidth}
\includegraphics[trim={0cm 0cm 0cm 0cm},clip,width=1.\linewidth]{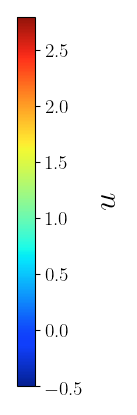}
\label{fig:vonkarman0}
\end{subfigure}
\begin{subfigure}[t]{0.3\textwidth}
\includegraphics[trim={0cm 0cm 0cm 0cm},clip,width=1.\linewidth]{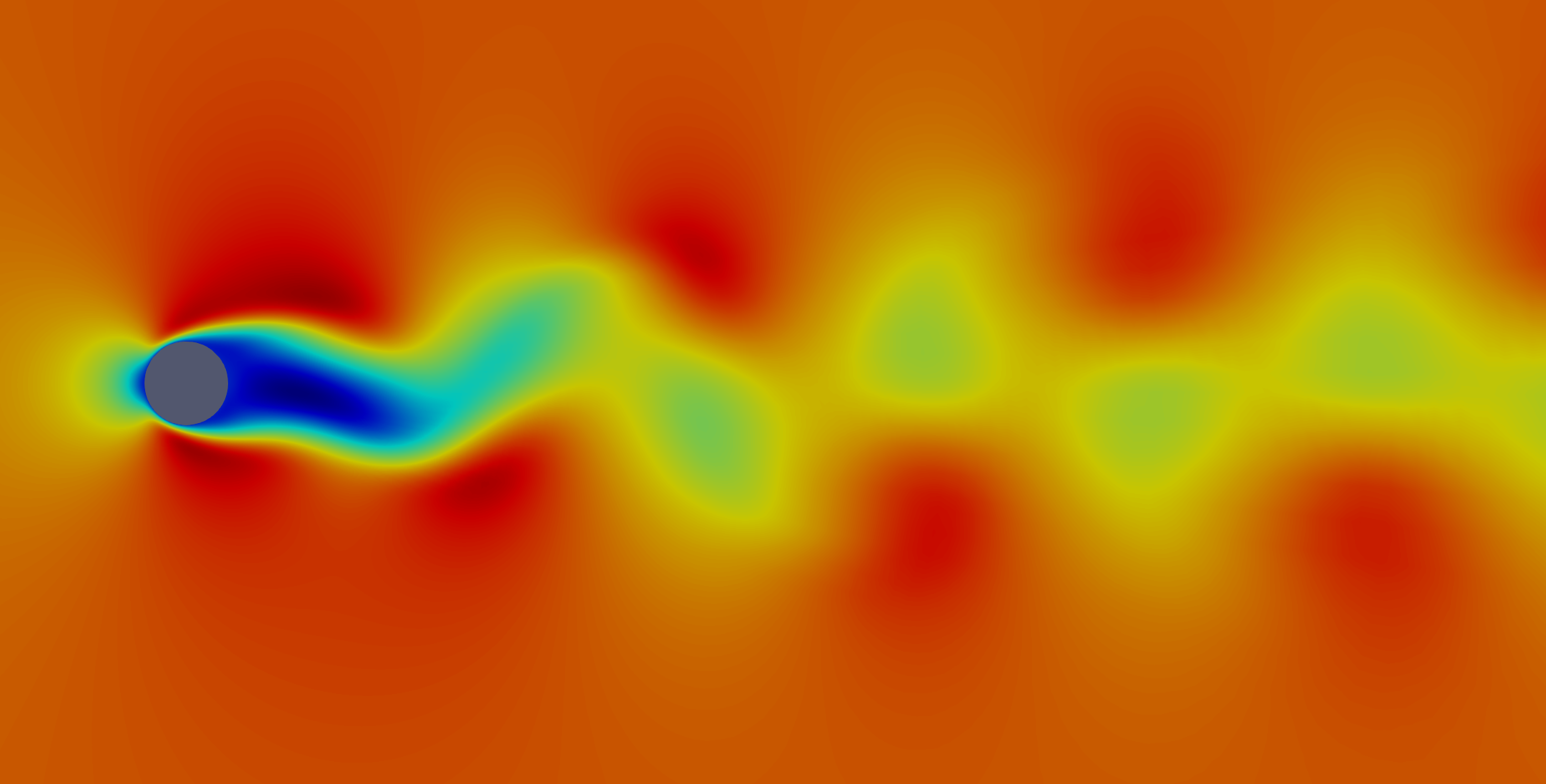}
\caption{$x$-velocity at $t=0.0$}
\label{fig:vonkarman1}
\end{subfigure}
\begin{subfigure}[t]{0.3\textwidth}
\includegraphics[trim={0cm 0cm 0cm 0cm},clip,width=1.\linewidth]{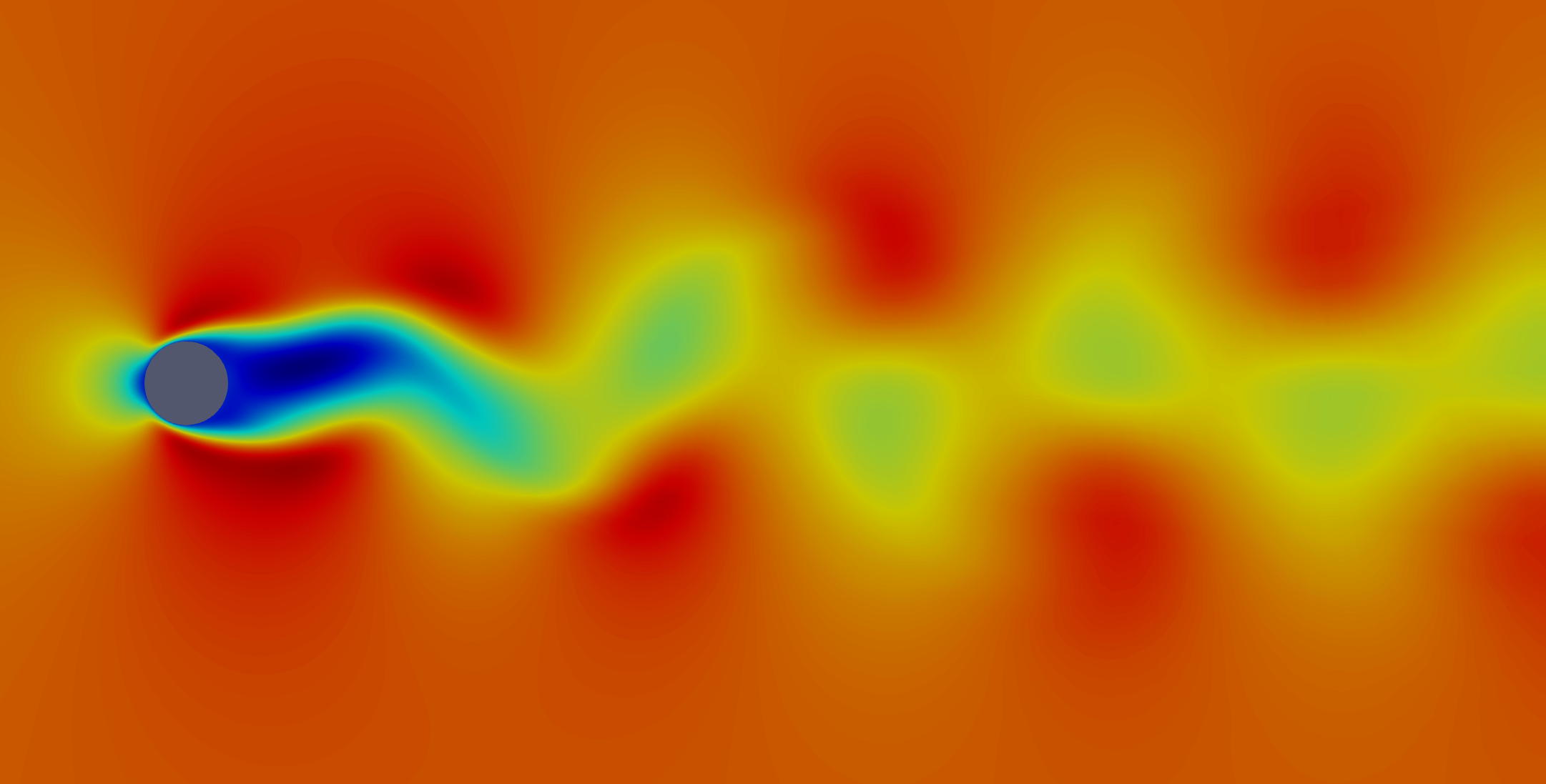}
\caption{$x$-velocity at $t=25.0$}
\end{subfigure}
\begin{subfigure}[t]{0.3\textwidth}
\includegraphics[trim={0cm 0cm 0cm 0cm},clip,width=1.\linewidth]{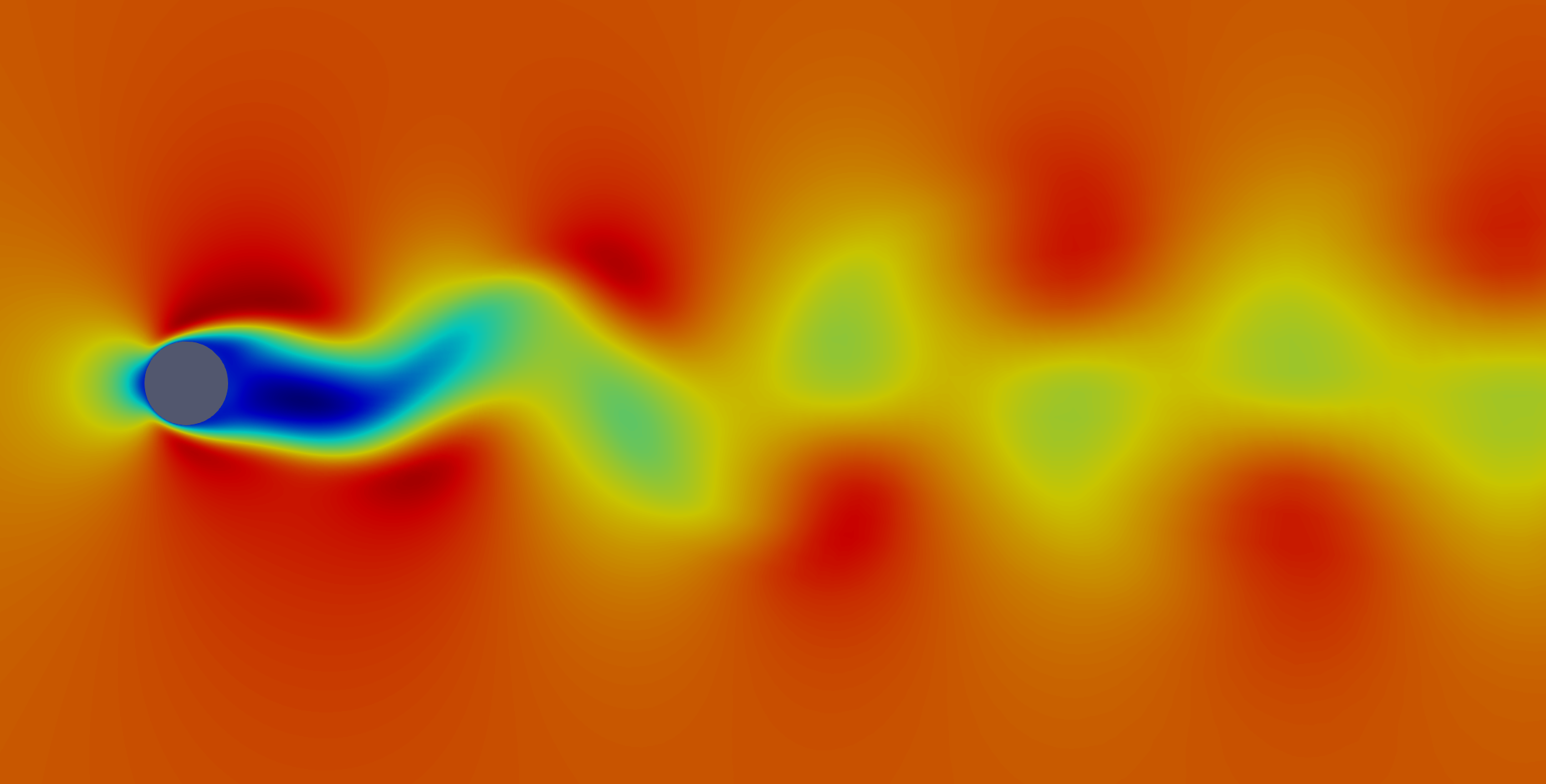}
\caption{$x$-velocity at $t=50.0$}
\label{fig:vonkarman3}
\end{subfigure}
\end{center}
\caption{Evolution of the von K\'arm\'an vortex street at Re=100.}
\label{fig:vonkarman}
\end{figure}

\subsubsection{Solution of the Reduced-Order Models}
The G ROM, APG ROM, and LSPG ROMs are considered. Details on their implementation are as follows:
\begin{enumerate}
\item Galerkin ROM: The Galerkin ROM is evolved in time using both explicit and implicit time integrators. In the explicit case, a strong stability RK3 method is used. In the implicit case, Crank-Nicolson time integration is used. The non-linear algebraic system is solved using SciPy's Jacobian-Free Netwon-Krylov solver. LGMRES is employed as the linear solver. The convergence tolerance for the max-norm of the residual is set at the default ftol=6e-6. 

\item Adjoint Petrov-Galerkin ROM: The APG ROM is evolved in time using the same time integrators as the Galerkin ROM. The extra term appearing in the APG model is computed via finite difference\footnote{It is noted that, when used in conjunction with a JFNK solver that utilizes finite difference to approximate the action of the Jacobian on a vector, approximating the extra RHS term in APG via finite difference leads to computing the finite difference approximation of a finite difference approximation. While not observed in the examples presented here, this can have a detrimental effect on accuracy and/or convergence.} with a step size of $\epsilon = $1e-5. Unless noted otherwise, the memory length $\tau$ is selected to be $\tau = \frac{0.2}{\rho(\vcoarsevec^T \mathbf{J}[\ucoarsevec(0)] \vcoarsevec ) } .$ The impact of $\tau$ on the numerical results is considered in the subsequent sections. 

\item LSPG ROM: The LSPG ROM is formulated from an implicit Crank-Nicolson temporal discretization. The resulting non-linear least-squares problem is solved using SciPy's least-squares solver with the `dogbox' method~\cite{scipy_leastsquares_dogbox}. The tolerance on the change to the cost function is set at ftol=1e-8. The tolerance on the change to the generalized coordinates is set at xtol=1e-8. The SciPy least-squares solver is comparable in speed to our own least-squares solver that utilizes the Gauss-Newton method with a thin QR factorization to solve the least-squares problem. The SciPy solver, however, was observed to be more robust in driving down the residual than the basic Gauss-Newton method with QR factorization, presumably due to SciPy's inclusion of trust-regions, and hence results are reported with the SciPy solver.
\end{enumerate}
All ROMs are initialized with the solution of the Re=$100$ FOM at time $t=0$, the $x$-velocity of which is shown in Figure~\ref{fig:vonkarman1}.


\begin{table}[]
\begin{center}
\begin{tabular}{c c c c}
\hline
Basis \# & Trial Basis Dimension ($K$) &  Energy Criteria  & $\tau$ (Adjoint Petrov-Galerkin) \\
\hline
1    & $ 11$ & $99\%$  & $1.0$ \\
2    & $ 42$ & $99.9\%$  & $0.3$   \\
3    & $ 86$ & $99.99\%$  & $0.1$  \\
\hline
\end{tabular}
\caption{Summary of the various basis dimensions used for Example 2}
\label{tab:rom_basis_1}
\end{center}
\end{table}

\subsubsection{Reconstruction of Re=100 Case}\label{sec:re100_rec}
Reduced-order models of the Re=100 case are first considered. This case was explicitly used in the construction of the POD basis and tests the ability of the ROM to reconstruct previously ``seen" dynamics. Unless otherwise noted, the default time-step for all ROMs is taken to be $\Delta t = 0.5$. The values of $\tau$ used in the APG ROMs are selected from the spectral radius heuristic and are given in Table~\ref{tab:rom_basis_1}. Figures~\ref{subfig:re100a} and~\ref{subfig:re100b} show the lift coefficient as well as the mean squared error (MSE) of the full-field ROM solutions for the G ROM, APG ROM, and LSPG ROMs for Basis \#2, while Figure~\ref{subfig:re100c} shows the integrated MSE for $t \in [0,200]$ for Basis \#1, 2, and 3. Figure~\ref{subfig:re100d} shows the integrated error as a function of relative CPU time for the various ROMs. The relative CPU time is defined with respect to the FOM, which is integrated with an explicit time-step 100 times lower than the ROMs. The lift coefficients predicted by all three ROMs are seen to overlay the FOM results. The mean squared error shows that, for a given trial basis dimension, the APG ROM is more accurate than both the Galerkin and LSPG ROMs. This is the case for both explicit and implicit time integrators. As shown in Figure~\ref{subfig:re100c}, the APG ROM converges at a similar rate to the G ROM as the dimension of the trial space grows. For Basis \#2 and \#3, the implicit time-marching schemes are slightly less accurate than the explicit time-marching schemes. Finally, Figure~\ref{subfig:re100d} shows that, for a given CPU time, the G ROM with explicit time-marching produces the least error. The APG ROM with explicit time-marching is the second-best performing method. In the implicit case, both the G and APG ROMs lead to lower error at a given CPU time than LSPG. This decrease in cost is due to the fact that the G and APG ROMs utilize Jacobian-Free Netwon-Krylov solvers. As discussed in Section~\ref{sec:cost}, it is much more challenging for LSPG to utilize Jacobian-Free methods. Due to the increased cost associated with implicit solvers, only explicit time integration is used for the G ROM and APG ROM beyond this point. 
\begin{figure}
\begin{center}
\begin{subfigure}[t]{0.48\textwidth}
\includegraphics[trim={0cm 0cm 0cm 0cm},clip,width=1.\linewidth]{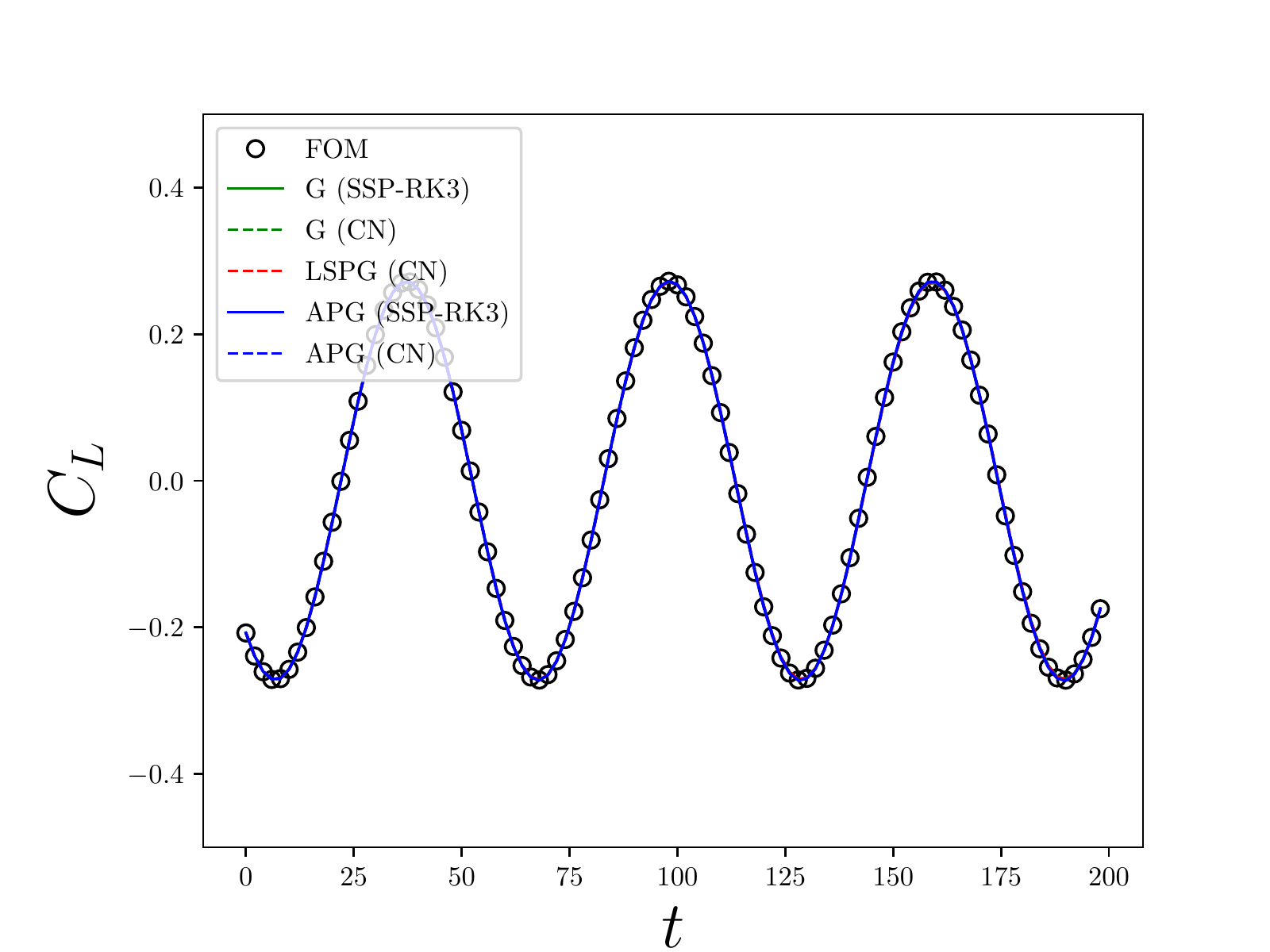}
\caption{Lift Coefficient as a function of time for Basis \# 2.}
\label{subfig:re100a}
\end{subfigure}
\begin{subfigure}[t]{0.48\textwidth}
\includegraphics[trim={0cm 0cm 0cm 0cm},clip,width=1.\linewidth]{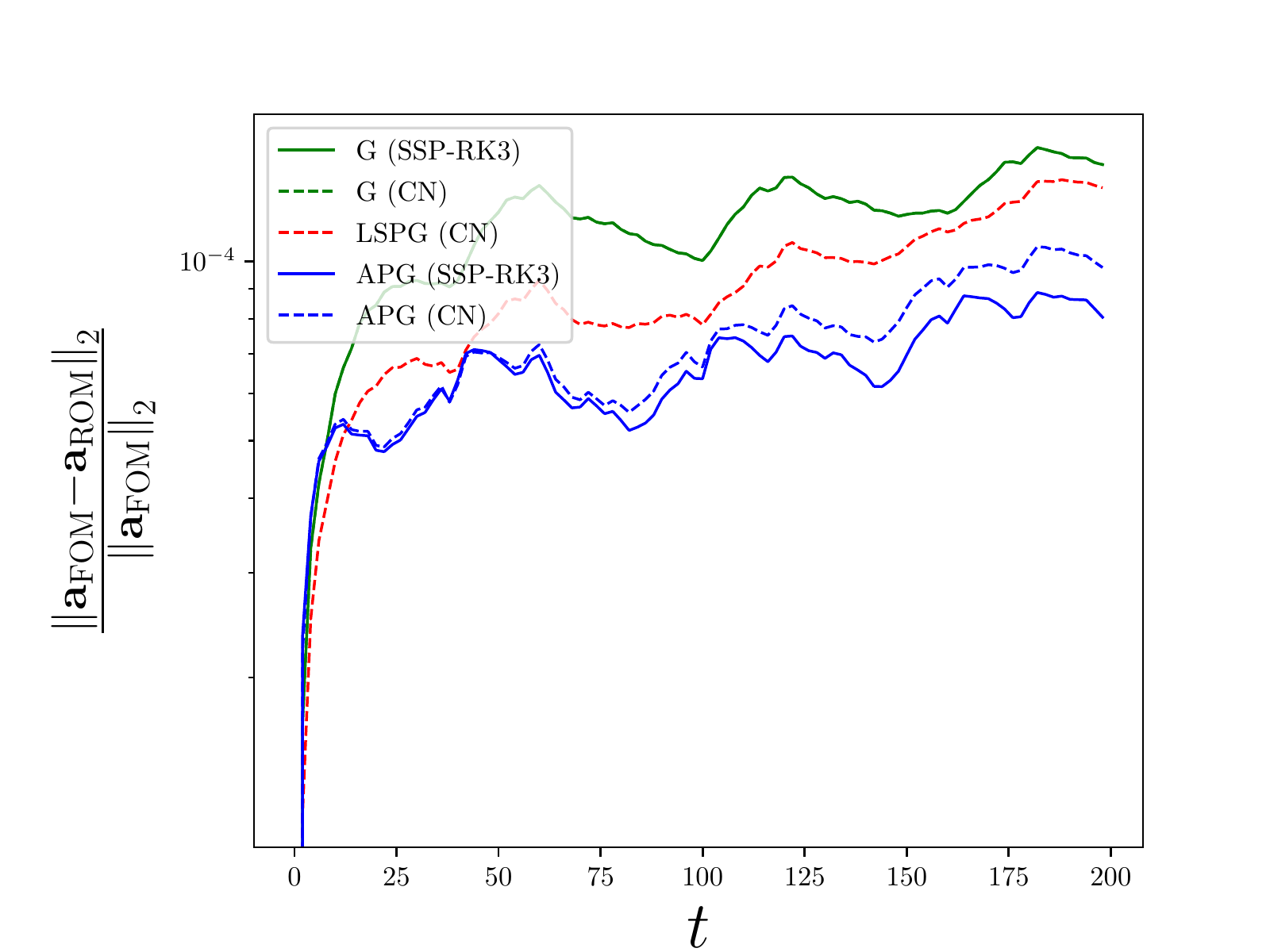}
\caption{Normalized error as a function of time for Basis \# 2.}
\label{subfig:re100b}
\end{subfigure}
\begin{subfigure}[t]{0.48\textwidth}
\includegraphics[trim={0cm 0cm 0cm 0cm},clip,width=1.\linewidth]{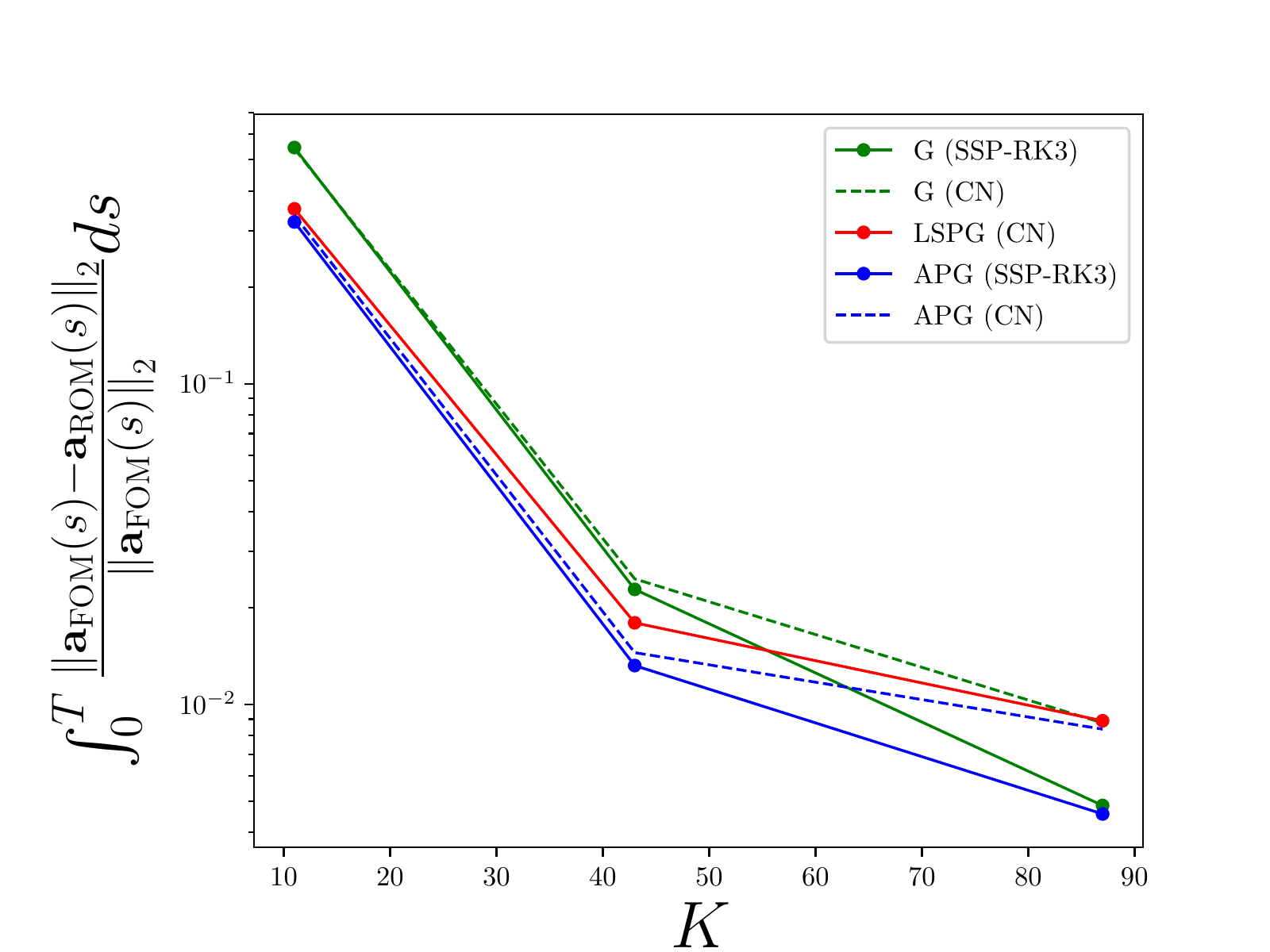}
\caption{Integrated normalized error vs trial subspace dimension.}
\label{subfig:re100c}
\end{subfigure}
\begin{subfigure}[t]{0.48\textwidth}
\includegraphics[trim={0cm 0cm 0cm 0cm},clip,width=1.\linewidth]{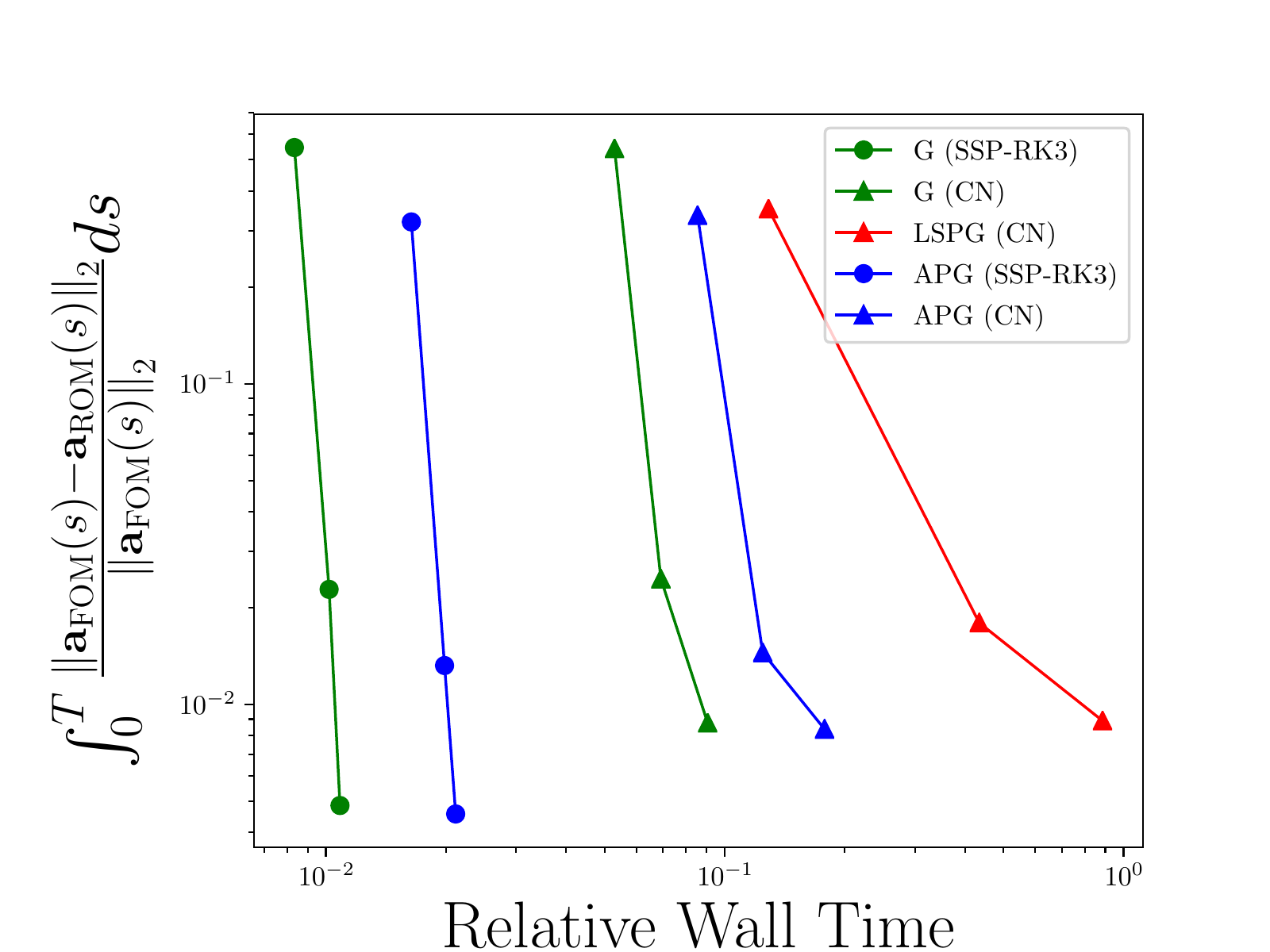}
\caption{Integrated normalized error vs relative CPU time.}
\label{subfig:re100d}
\end{subfigure}
\end{center}
\caption{ROM results for flow over cylinder at Re=100. The lift coefficient is defined as $C_L = \frac{2L}{\rho U_{\infty}^2 D}$, with $L$ being the integrated force on the cylinder perpendicular to the free-stream velocity vector.}
\label{fig:re100}
\end{figure}

Next, we investigate the sensitivity of the different ROMs to the time-step size. Reduced-order models of the Re=$100$ case using Basis \#2 are solved using time-steps of $\Delta t = \big[0.1,0.2,0.5,1]$. Note that the largest time-step considered is $200$ times larger than the FOM time-step, thus reducing the temporal dimensionality of the problem by 200 times. The mean-squared error of each ROM solution is shown in Figure~\ref{fig:re100_dtvary}. The G and APG ROMs are stable for all time-steps considered. Further, it is seen that varying the time-step has a minimal effect on the accuracy of the G and APG ROMs. In contrast, the accuracy of LSPG deteriorates if the time-step grows too large. This is due to the fact that, as shown in Ref.~\cite{carlberg_lspg_v_galerkin}, the stabilization added by LSPG depends on the time-step size. Optimal accuracy of the LSPG method requires an intermediate time-step. The ability of the APG and G ROMs to take large time-steps without a significant degradation in accuracy allows for significant computational savings. This advantage is further amplified when large time-steps can be taken with an explicit solver, as is the case here. This will be discussed in more detail in Section~\ref{sec:cylhyper}.
\begin{figure}
\begin{center}
\begin{subfigure}[t]{0.49\textwidth}
\includegraphics[trim={0cm 0cm 0cm 0cm},clip,width=1.\linewidth]{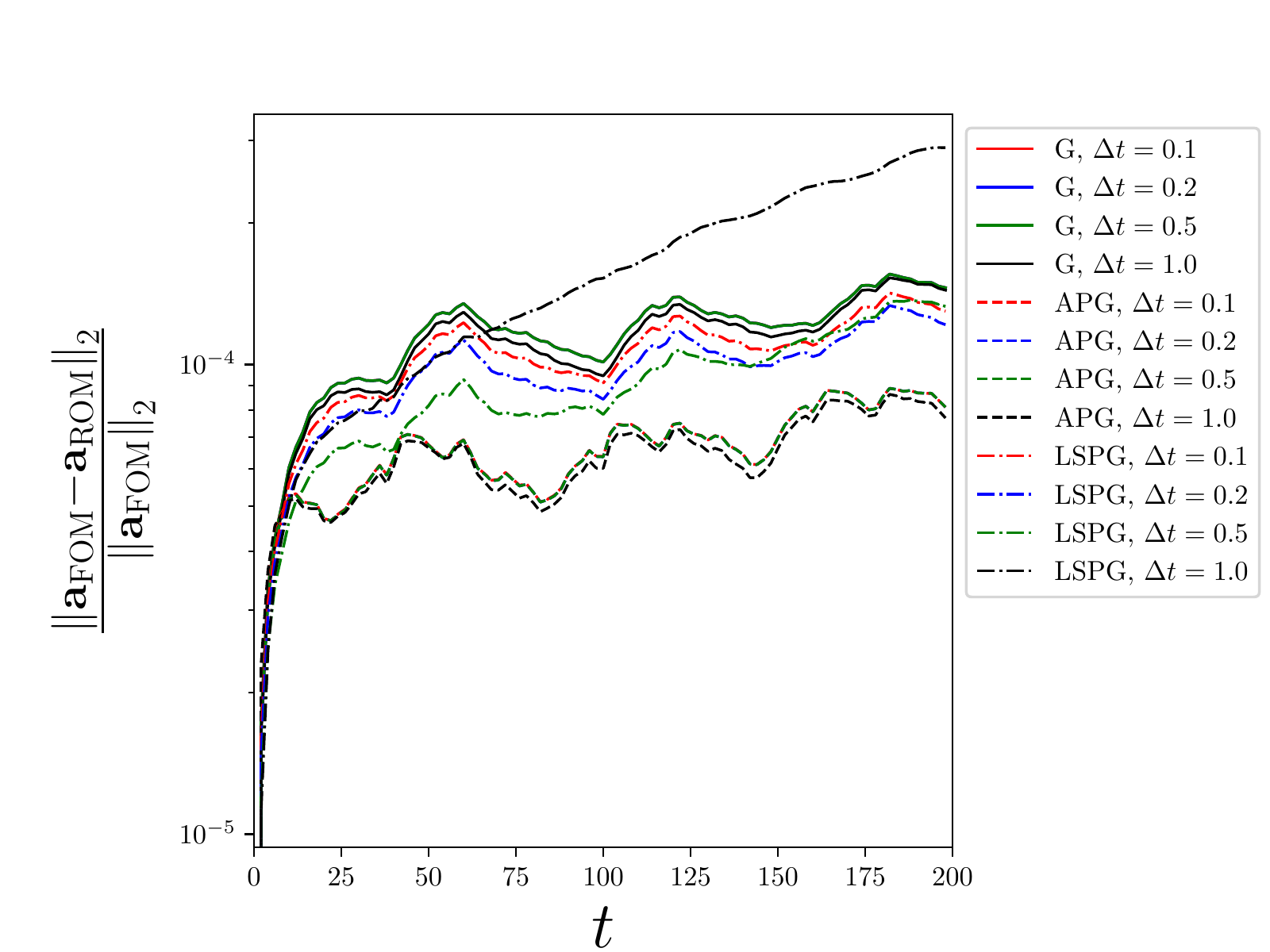}
\caption{Normalized error as a function of time.}
\end{subfigure}
\begin{subfigure}[t]{0.49\textwidth}
\includegraphics[trim={0cm 0cm 0cm 0cm},clip,width=1.\linewidth]{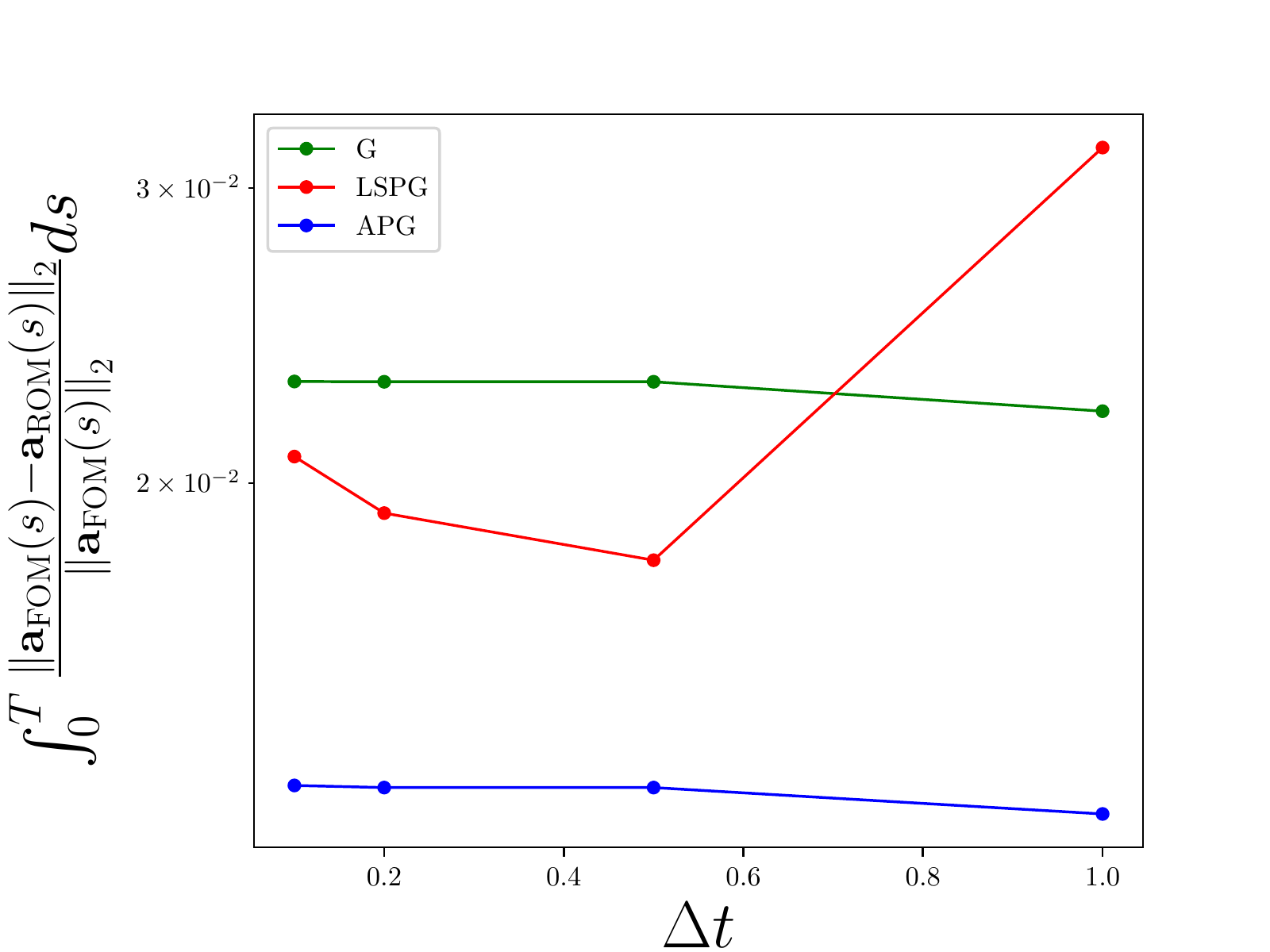}
\caption{Integrated normalized error as a function of the time-step.}
\end{subfigure}
\end{center}
\caption{Results for time-step study of flow over cylinder at Re$=100$.}
\label{fig:re100_dtvary}
\end{figure}

Lastly, we numerically investigate the sensitivity of APG to the parameter $\tau$ by running simulations for $\tau=[0.001,0.01,0.1,0.3,0.5,1.]$. All simulations are run at $\Delta t = 0.5$. The results of the simulations are shown in Figure~\ref{fig:re100_tau}. It is seen that, for all values of $\tau$, the APG ROM produces a better solution than the G ROM. The lowest error is observed for an intermediate value of $\tau$, in which case the APG ROM leads to over a $50\%$ reduction in error from the G ROM. As $\tau$ approaches zero, the APG ROM solution approaches the Galerkin ROM solution. Convergence plots for LSPG as a function of $\Delta t$ are additionally shown in Figure~\ref{subfig:re100b_tau}. It is seen that the optimal time-step in LSPG is similar to the optimal value of $\tau$ in APG.

\begin{figure}
\begin{center}
\begin{subfigure}[t]{0.45\textwidth}
\includegraphics[trim={0cm 0cm 0cm 0cm},clip,width=1.\linewidth]{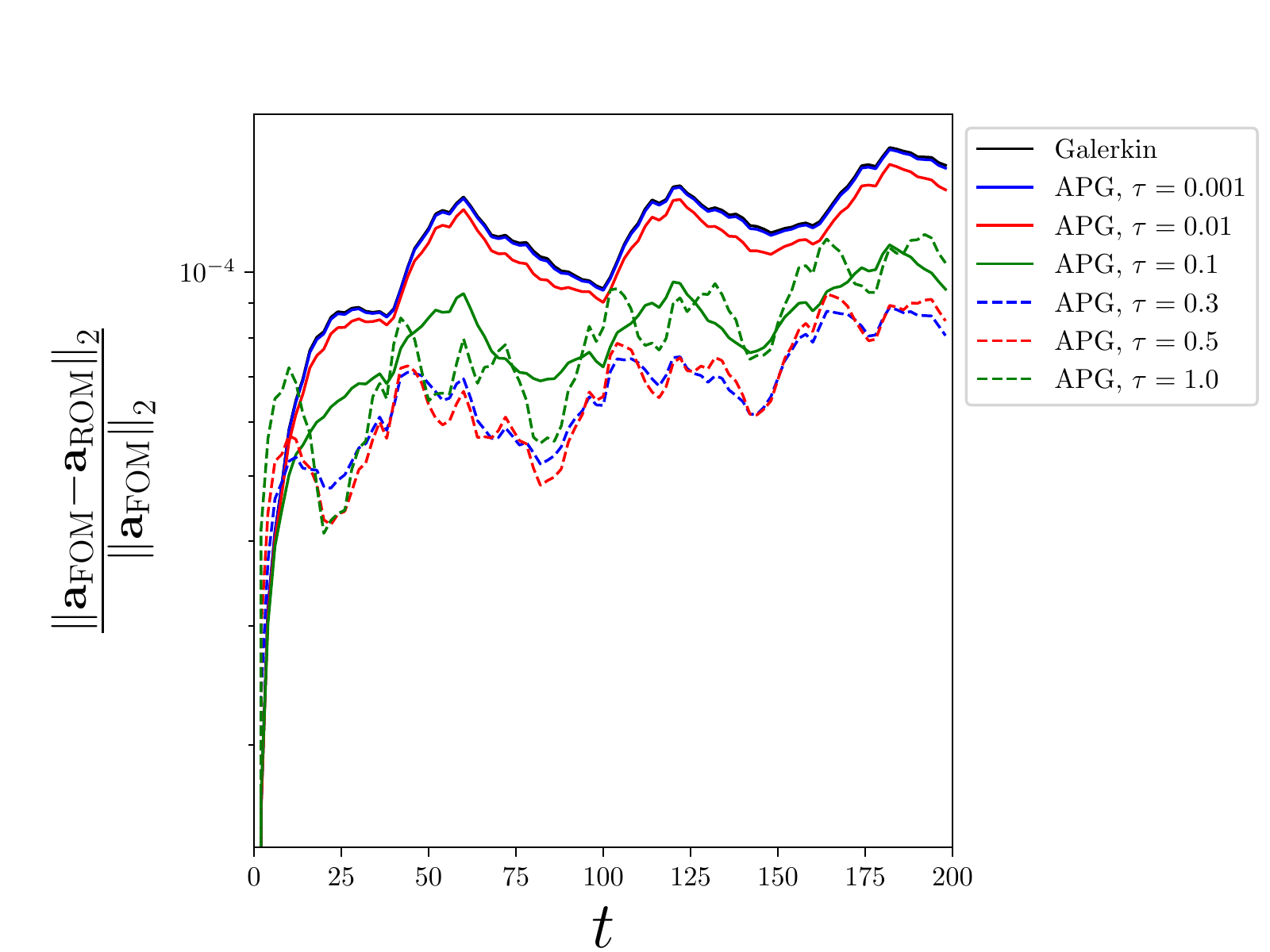}
\caption{Normalized error as a function of time}
\label{subfig:re100a_tau}
\end{subfigure}
\begin{subfigure}[t]{0.45\textwidth}
\includegraphics[trim={0cm 0cm 0cm 0cm},clip,width=1.\linewidth]{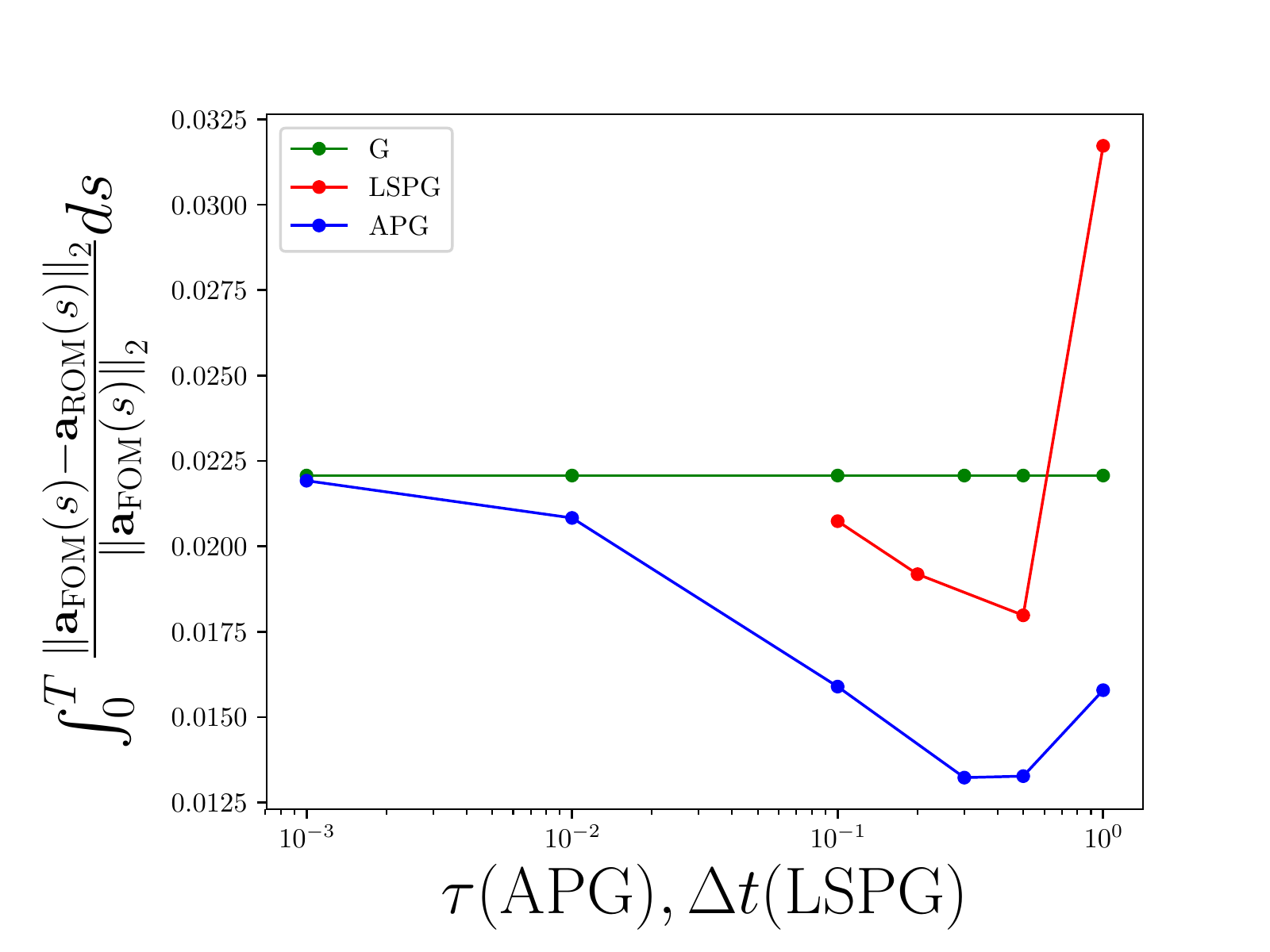}
\caption{Integrated normalized error as a function of $\tau$. Results for how the time-step $\Delta t$ impacts LSPG are included for reference.}
\label{subfig:re100b_tau}
\end{subfigure}
\end{center}
\caption{Summary of numerical results investigating the impact of the parameter $\tau$ on the performance of the Adjoint Petrov-Galerkin method.}
\label{fig:re100_tau}
\end{figure}

\subsubsection{Parametric Study of Reynolds Number Dependence}
Next, the ability of the different ROMs to interpolate between between different Reynolds numbers is studied. Simulations at Reynolds numbers of Re=100,150,200,250, and 300 with Basis $\#2$ and $\#3$ are performed.  All cases are initialized from the Re=100 simulation. Note that the trial spaces in the ROMs were constructed from statistically steady-state FOM simulations of Re=100,200,300. The Reynolds number is modified by changing the viscosity. 

Figure~\ref{fig:ROM_summary} summarizes the amplitude of the lift coefficient signal as well as the shedding frequency for the various methods. The values reported in Figure~\ref{fig:ROM_summary} are computed from the last 150s of the simulations\footnote{Not all G ROMs reached a statistically steady state over the time window considered}. The Galerkin ROM is seen to do poorly in predicting the lift coefficient amplitude for both Basis $\#2$ and Basis $\#3$. Unlike in the Re=100 case, enhancing the basis dimension does not improve the performance of the ROMs. Both the LSPG and APG ROMs are seen to offer much improved predictions over the Galerkin and ROM. This result is promising, as the ultimate goal of reduced-order modeling is to provide predictions in new regimes.

The results presented in this example highlight the shortcomings of the Galerkin ROM. To obtain results that are even qualitatively correct for the Re=\{150,200,250,300\} cases, either APG or LSPG must be used. As reported in Figure~\ref{subfig:re100d}, explicit APG is over an order of magnitude faster than LSPG, and implicit APG with a JFNK solver is anywhere from 2x to 5x faster than LSPG. Therefore, APG is the best-performing method for this example. 



 \begin{figure}
\begin{center}
\begin{subfigure}[t]{0.49\textwidth}
\includegraphics[trim={0cm 0cm 0cm 0cm},clip,width=1.\linewidth]{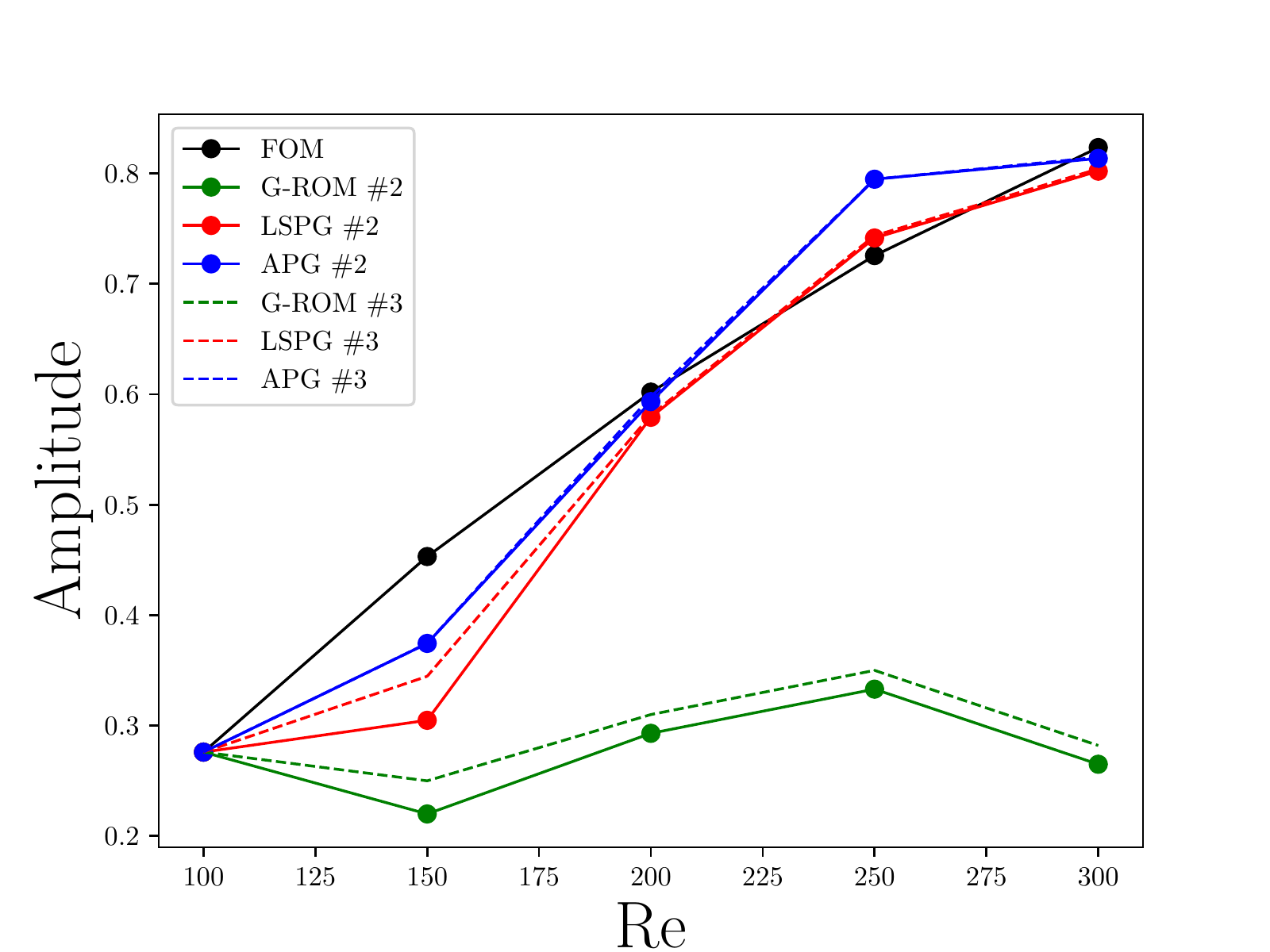}
\caption{Prediction for lift coefficient amplitudes}
\end{subfigure}
\begin{subfigure}[t]{0.49\textwidth}
\includegraphics[trim={0cm 0cm 0cm 0cm},clip,width=1.\linewidth]{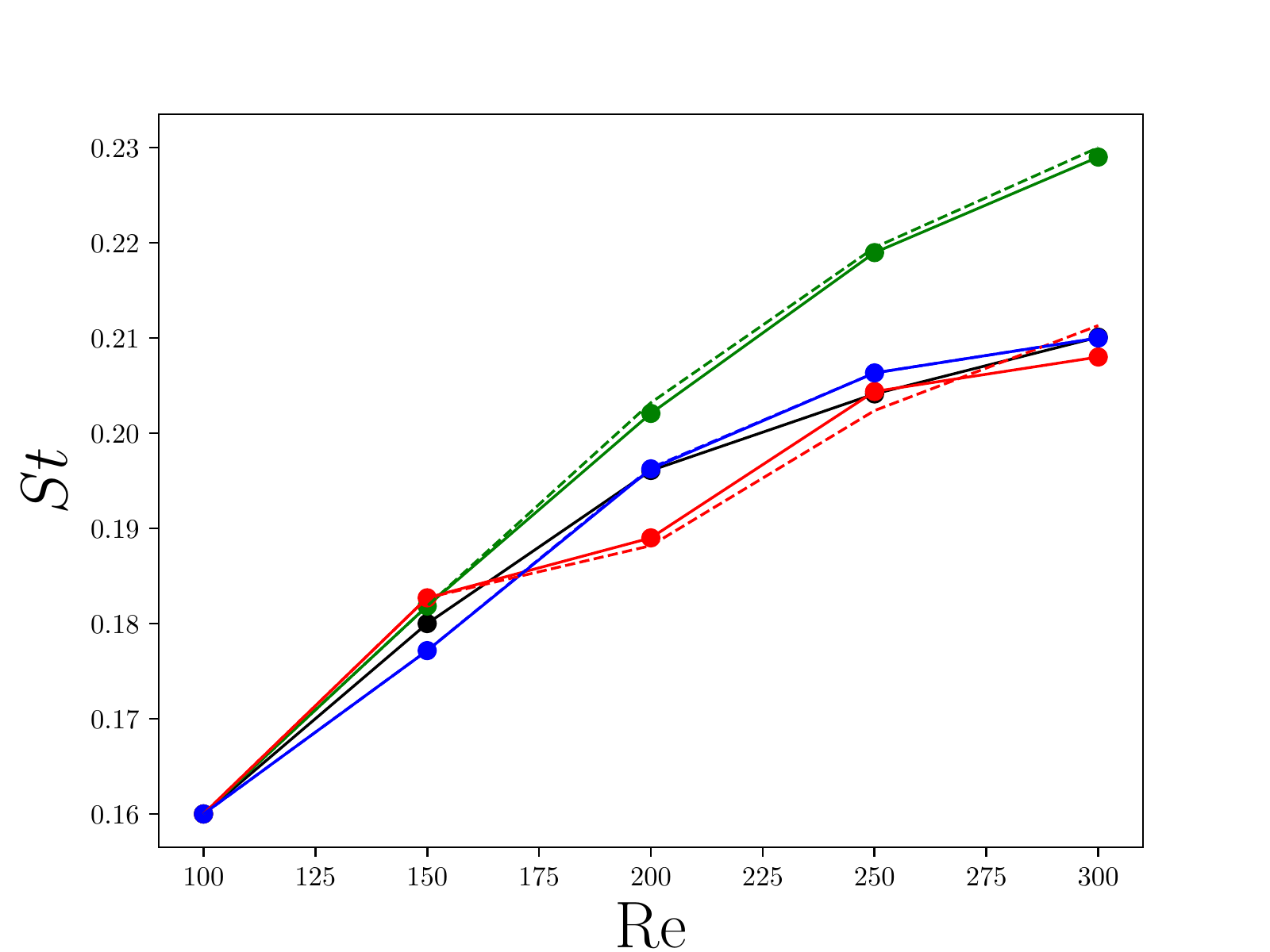}
\caption{Prediction for shedding frequency}
\end{subfigure}

\begin{subfigure}[t]{0.49\textwidth}
\includegraphics[trim={0cm 0cm 0cm 0cm},clip,width=1.\linewidth]{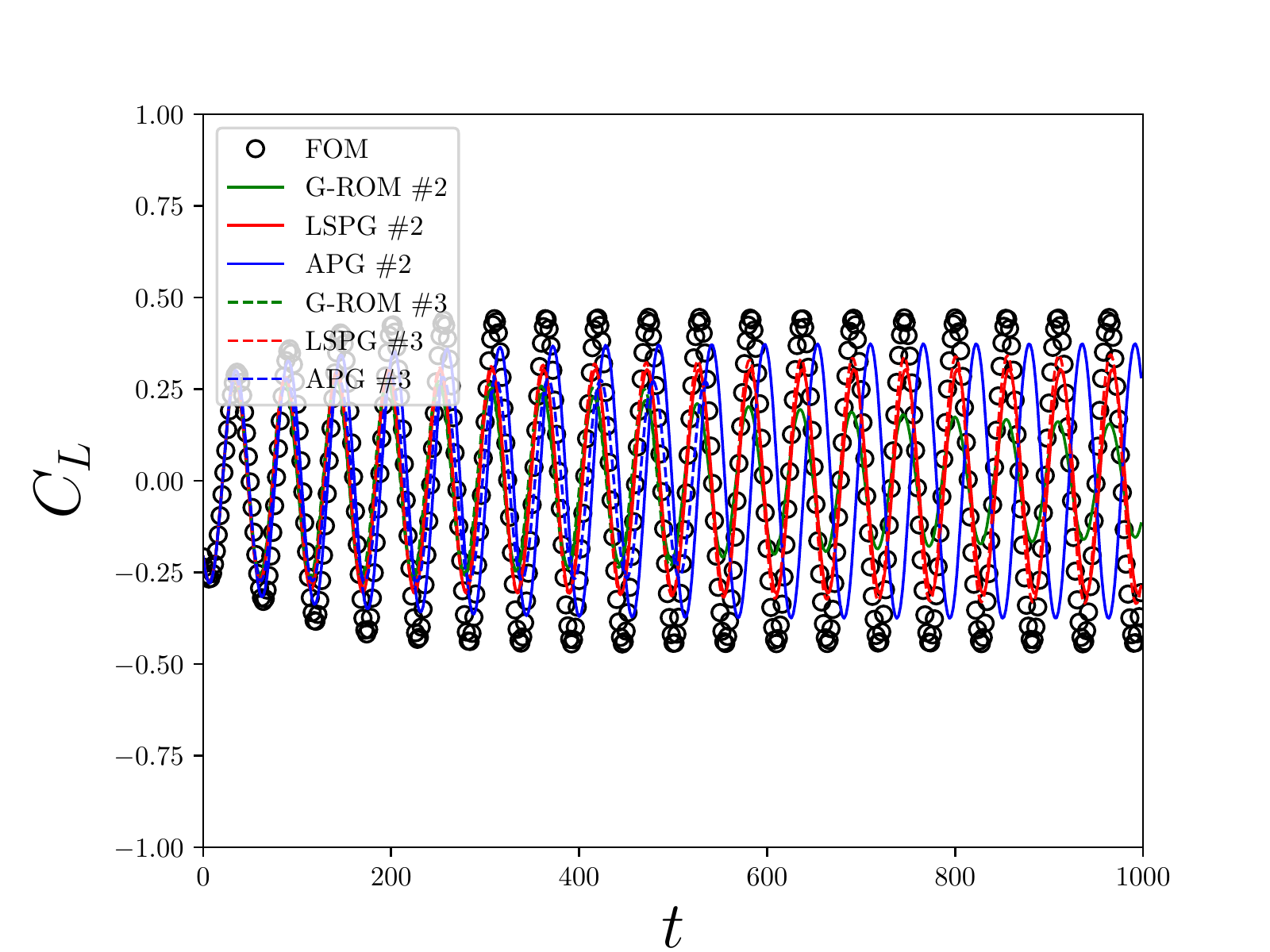}
\caption{Re=$150$}
\end{subfigure}
\begin{subfigure}[t]{0.49\textwidth}
\includegraphics[trim={0cm 0cm 0cm 0cm},clip,width=1.\linewidth]{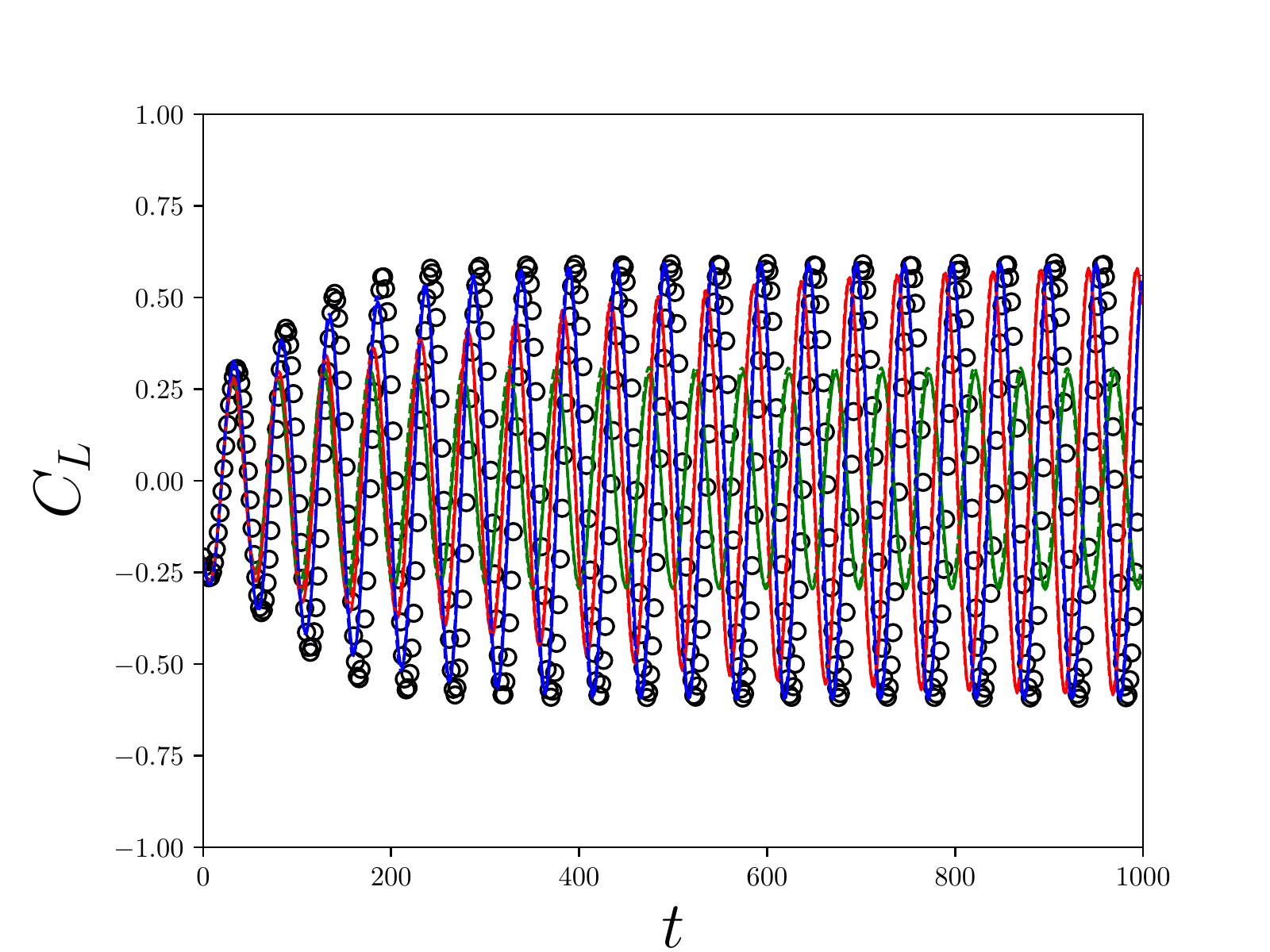}
\caption{Re=$200$}
\end{subfigure}
\begin{subfigure}[t]{0.49\textwidth}
\includegraphics[trim={0cm 0cm 0cm 0cm},clip,width=1.\linewidth]{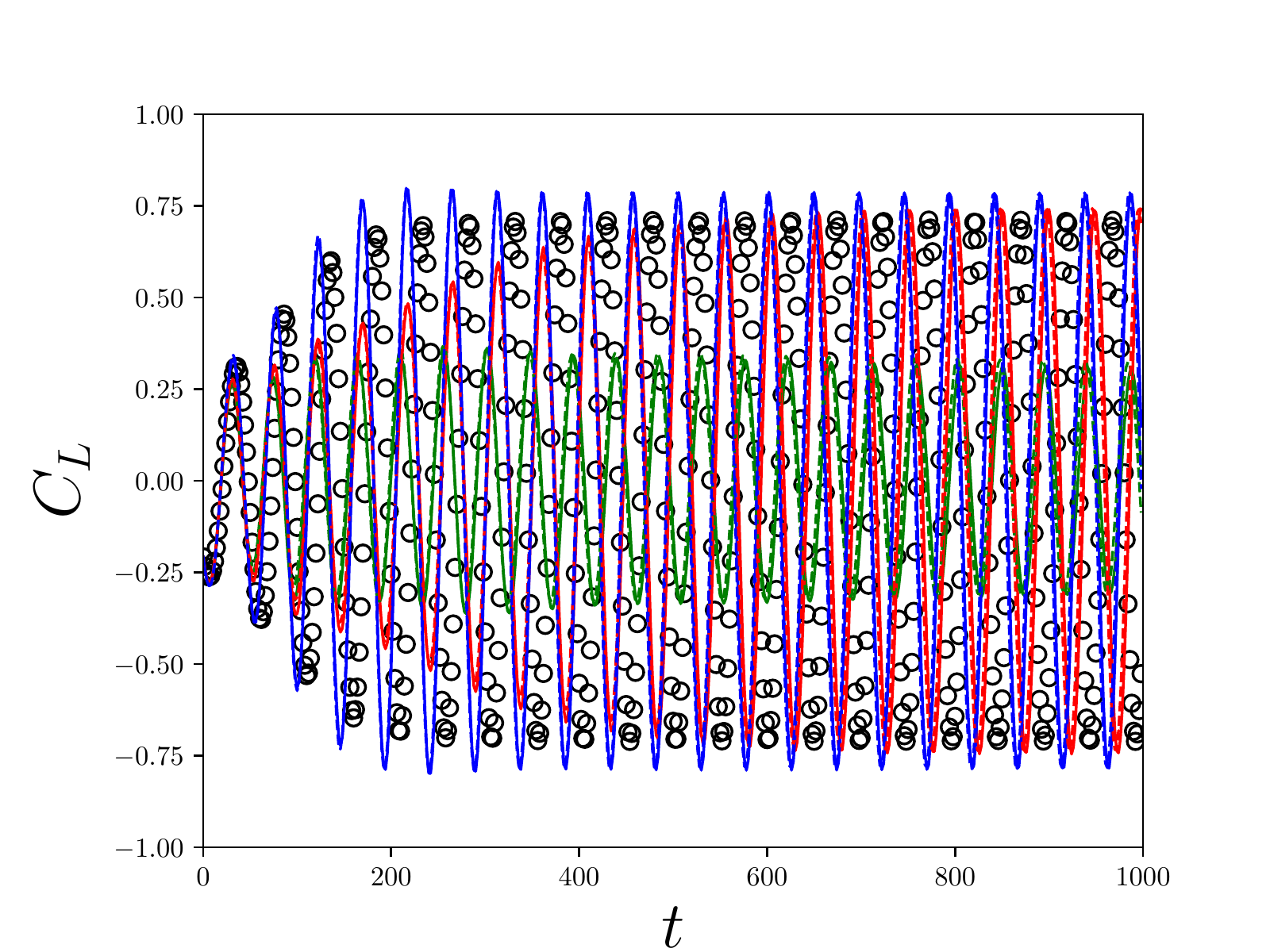}
\caption{Re=$250$}
\end{subfigure}
\begin{subfigure}[t]{0.49\textwidth}
\includegraphics[trim={0cm 0cm 0cm 0cm},clip,width=1.\linewidth]{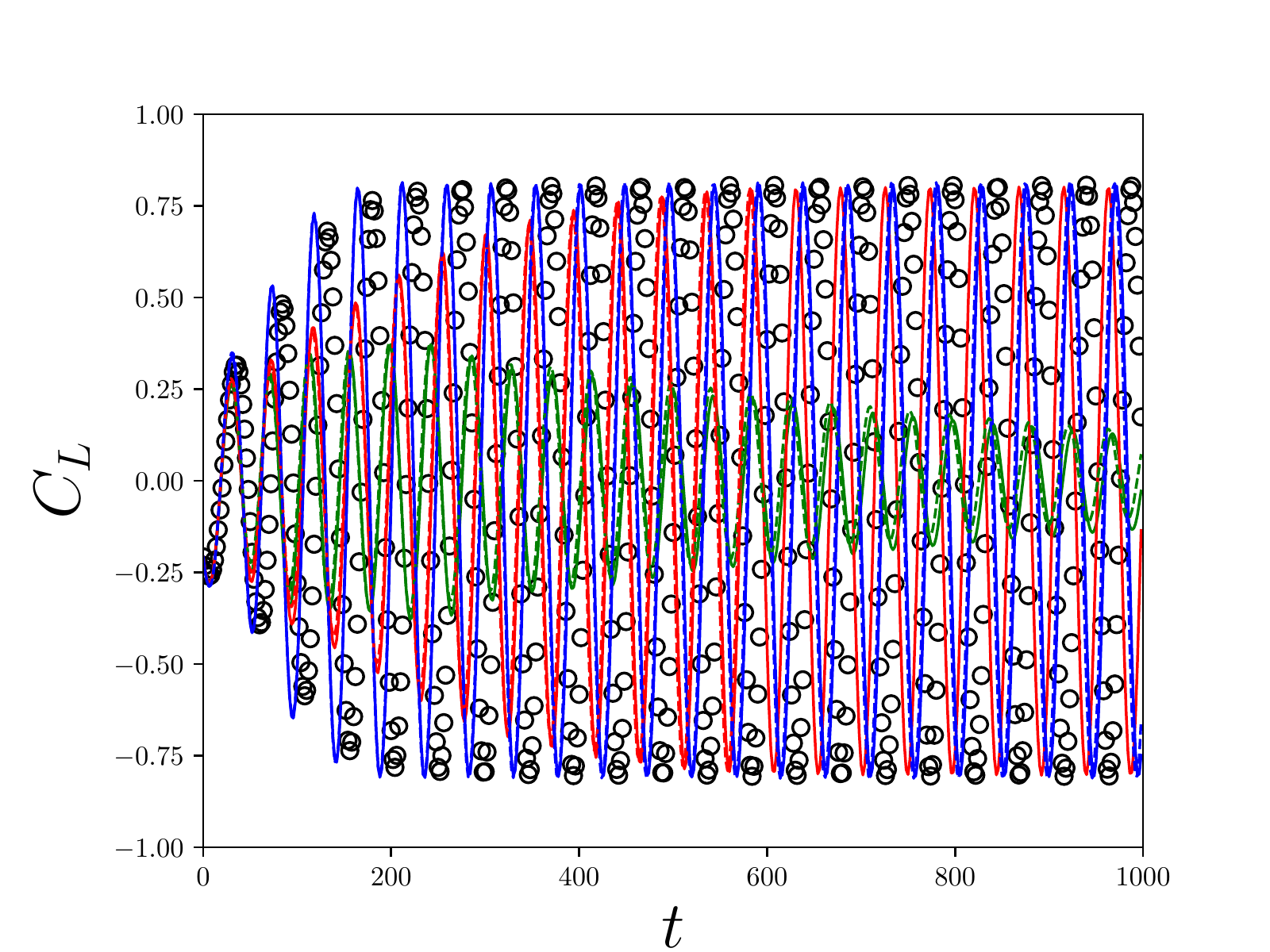}
\caption{Re=$300$}
\end{subfigure}

\end{center}
\caption{Summary of ROM simulations}
\label{fig:ROM_summary}
\end{figure}

\subsubsection{Flow Over Cylinder at Re=$100$ with Hyper-Reduction}\label{sec:cylhyper}
The last example considered is again flow over a cylinder, but this time the reduced-order models are augmented with hyper-reduction. The purpose of this example is to examine the performance of the different reduced-order models when fully equipped with state-of-the-art reduction techniques. For hyper-reduction, an additional snapshot matrix of the right-hand side is generated. This additional snapshot matrix is generated by following steps one through five provided in Section~\ref{sec:cyl_romsteps}. Hyper-reduction for the G ROM and the APG ROM is achieved through the Gappy POD method~\cite{everson_sirovich_gappy}. Hyper-reduction for LSPG is achieved through collocation using the same sampling points.\footnote{It is noted that collocated LSPG out-performed the GNAT method for this example, and thus GNAT is not considered.} When augmented with hyper-reduction, the trial basis dimension ($K$), right-hand side basis dimension ($r$), and number of sample points $(N_s)$ can impact the performance of the ROMs. Table~\ref{tab:rom_basis_details} summarizes the various permutations of $K$, $r$, and $N_s$ considered in this example. The sample points are selected through a QR factorization of the right-hand side snapshot matrix~\cite{qdeim_drmac}. These sample points are then augmented such that they contain every conserved variable and quadrature point at the selected cells. The sample mesh corresponding to Basis numbers 4,5, and 6 in Table~\ref{tab:rom_basis_details} is shown in Figure~\ref{fig:re100_sample}. Details on hyper-reduction and its implementation in our discontinuous Galerkin code are provided in Appendix~\ref{appendix:hyper}.  

\begin{table}[]
\begin{tabular}{c c c c c}
\hline
Basis \# & Trial Basis Dimension ($K$) &  RHS Basis Dimension ($r$)  & Sample Points ($N_S$) & Maximum Stable $\Delta t$\\
\hline
1    & $ 11$ & $103$ & $4230$ & 4.0\\
2    & $ 42$ & $103$ & $4230$ & 2.0\\
3    & $ 86$ & $103$ & $4230$ & 1.0\\
4    & $ 11$ & $268$ & $8460$ & 4.0\\
5    & $ 42$ & $268$ & $8460$ & 2.0\\
6    & $ 86$ & $268$ & $8460$ & 1.0\\
\hline
\end{tabular}
\caption{Summary of the various reduced-order models evaluated on the flow over cylinder problem. The maximum stable $\Delta t$ for each basis is reported for the SSP-RK3 explicit time-marching scheme and was empirically determined.}
\label{tab:rom_basis_details}
\end{table}


Flow at Re=$100$ is considered. All simulations are performed at the maximum stable time-step for a given basis dimension, as summarized in Table~\ref{tab:rom_basis_details}. Figure~\ref{fig:re100_qdeim_pareto} shows the integrated error as a function of relative wall time for the various ROM techniques and basis numbers. All methods  show significant computational speedup, with the G and APG ROMs producing wall-times up to 5000 times faster than the FOM while retaining a reasonable MSE. It is noted that the majority of this speed up is attributed to the increase in time-step size. 
For a given level of accuracy, LSPG is significantly more expensive than the G and APG ROMs. The reason for this expense is three-fold. First, LSPG is inherently implicit. For a given time-step size, an implicit step is more expensive than an explicit step. Second, LSPG requires an intermediate time-step for optimal accuracy. In this example, this intermediate time-step is small enough that the computational gains that could be obtained with an implicit method are negated. The third reason is that, for each Gauss-Newton iteration, LSPG requires the computation of the action of the Jacobian on the trial space basis, $\vcoarsevec$. This expense can become significant for large basis dimensions as it requires the computation of a dense Jacobian. It is noted that it may be possible to achieve computational speed-ups in our implementation of LSPG through the development of a least-squares solver more tailored to LSPG, sparse Jacobian updates, or Jacobian approximation techniques.


\begin{figure}
\begin{center}
\begin{subfigure}[t]{0.48\textwidth}
\includegraphics[trim={4cm 0cm 4cm 0cm},clip,width=1.\linewidth]{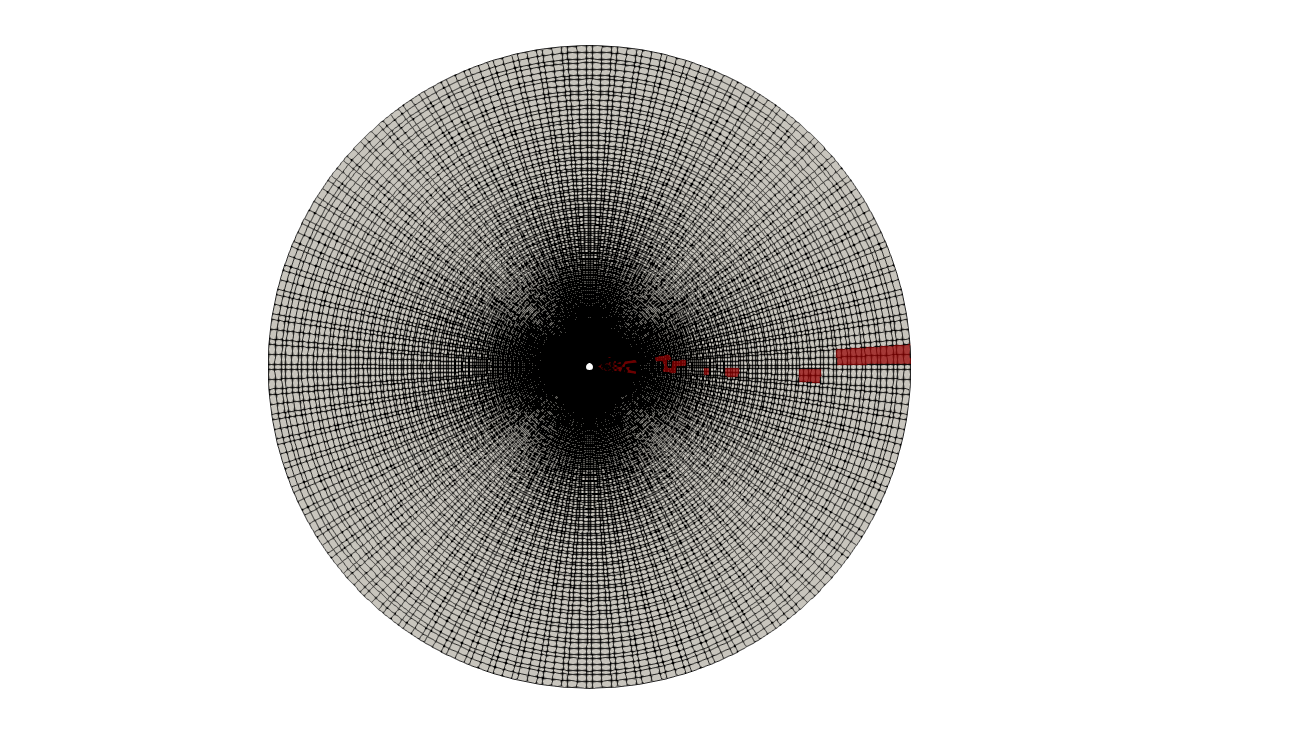}
\caption{Full sample mesh}
\label{subfig:re100_sampling_full}
\end{subfigure}
\begin{subfigure}[t]{0.48\textwidth}
\includegraphics[trim={0cm 0cm 4cm 0cm},clip,width=1.\linewidth]{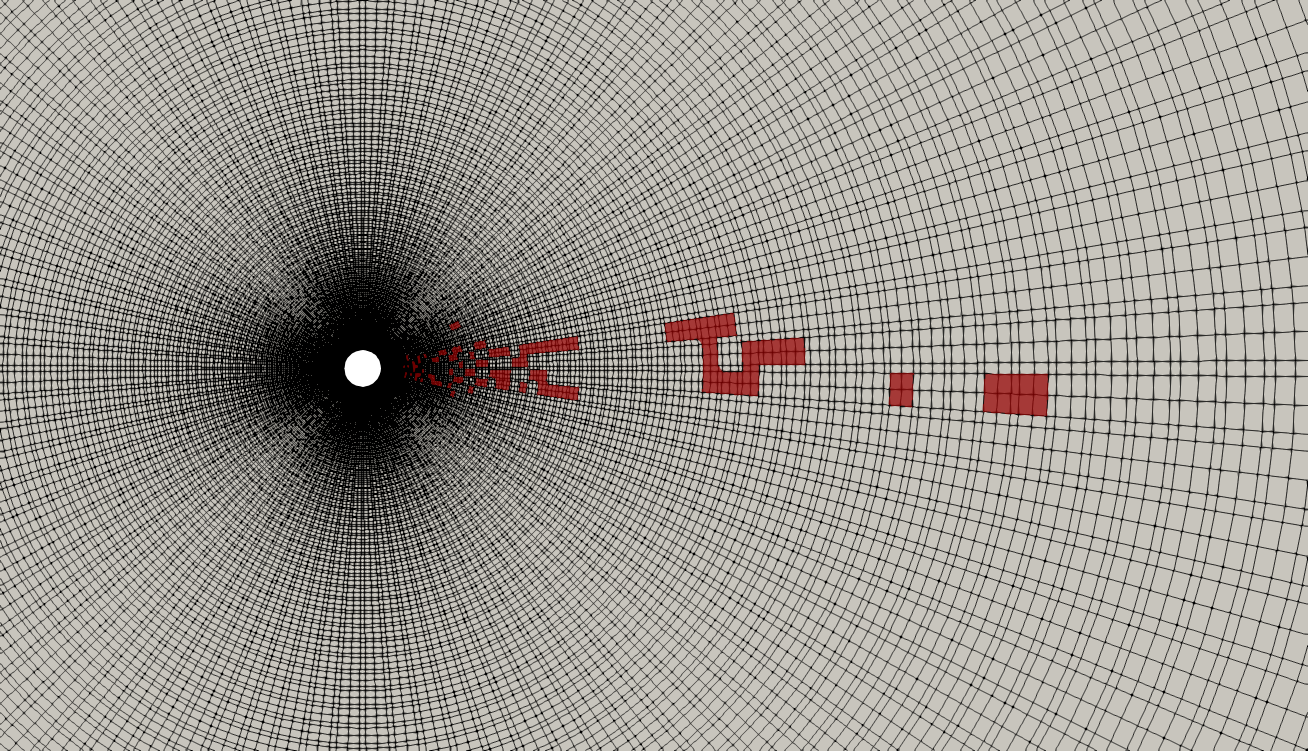}
\caption{Close up of wake}
\label{subfig:re100_sampling_zoom}
\end{subfigure}
\end{center}
\caption{Mesh used for hyper-reduction. Cells colored in red are the sampled cells. Note that all conserved variables and quadrature points are computed within a cell.}
\label{fig:re100_sample}
\end{figure}

\begin{figure}
\begin{center}
\begin{subfigure}[t]{0.49\textwidth}
\includegraphics[trim={0cm 0cm 0cm 0cm},clip,width=1.\linewidth]{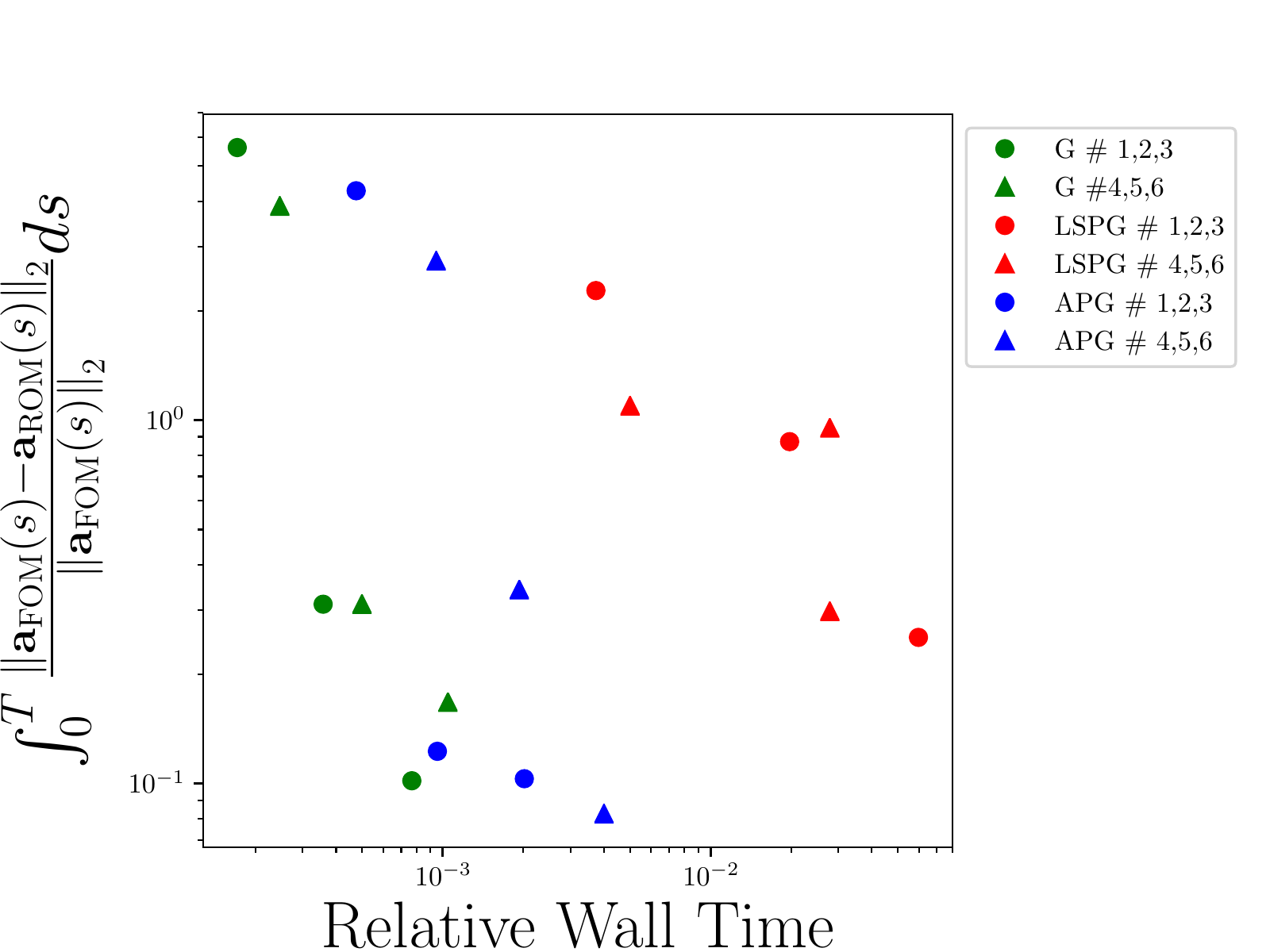}
\caption{Wall time vs normalized error for hyper-reduced ROMs.}
\end{subfigure}
\begin{subfigure}[t]{0.49\textwidth}
\includegraphics[trim={0cm 0cm 0cm 0cm},clip,width=1.\linewidth]{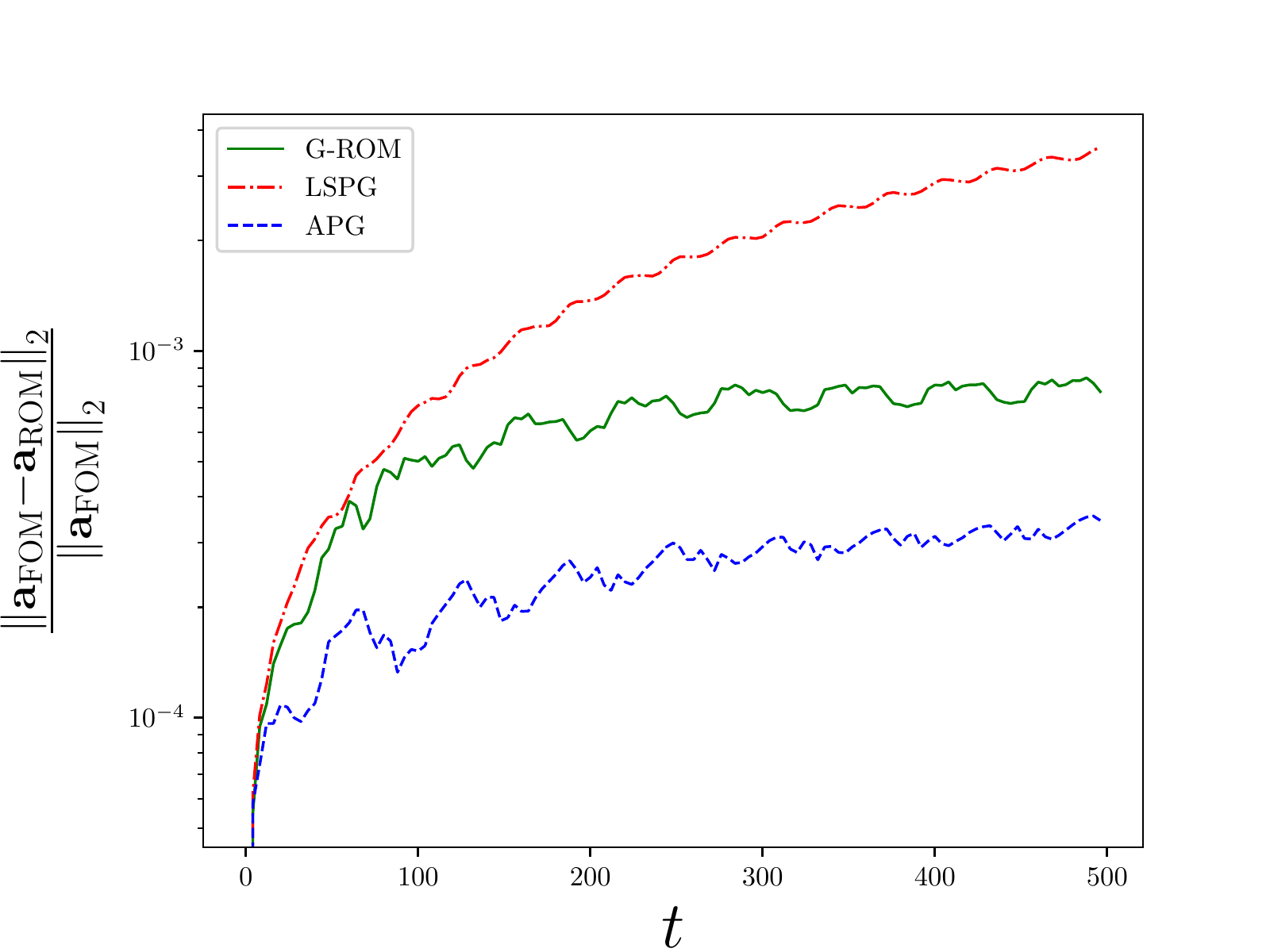}
\caption{Normalized error as a function of time for Basis \#5}
\end{subfigure}
\begin{subfigure}[t]{0.49\textwidth}
\includegraphics[trim={0cm 0cm 0cm 0cm},clip,width=1.\linewidth]{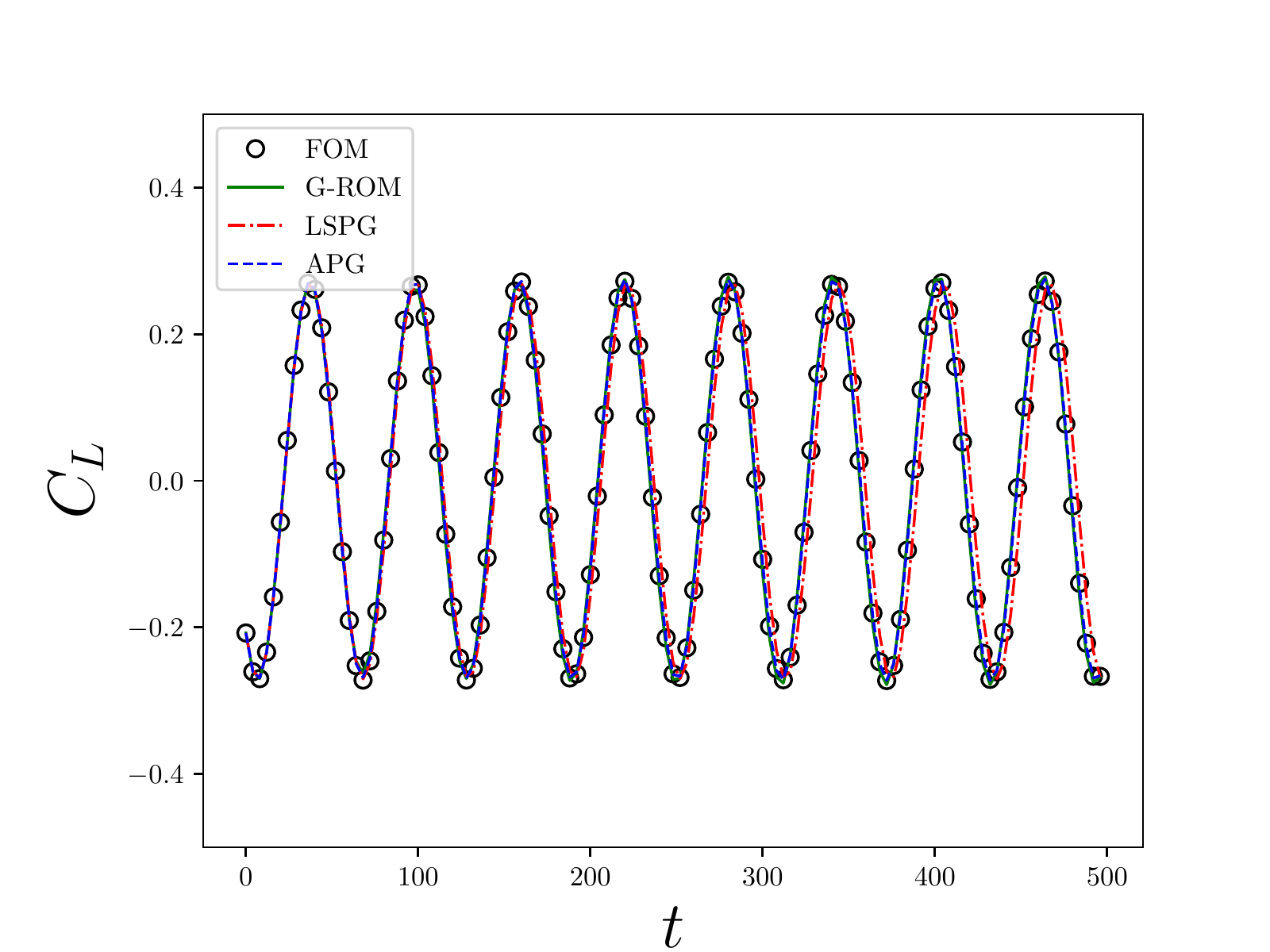}
\caption{Lift coefficient as a function of time for Basis \# 5.}
\end{subfigure}
\begin{subfigure}[t]{0.49\textwidth}
\includegraphics[trim={0cm 0cm 0cm 0cm},clip,width=1.\linewidth]{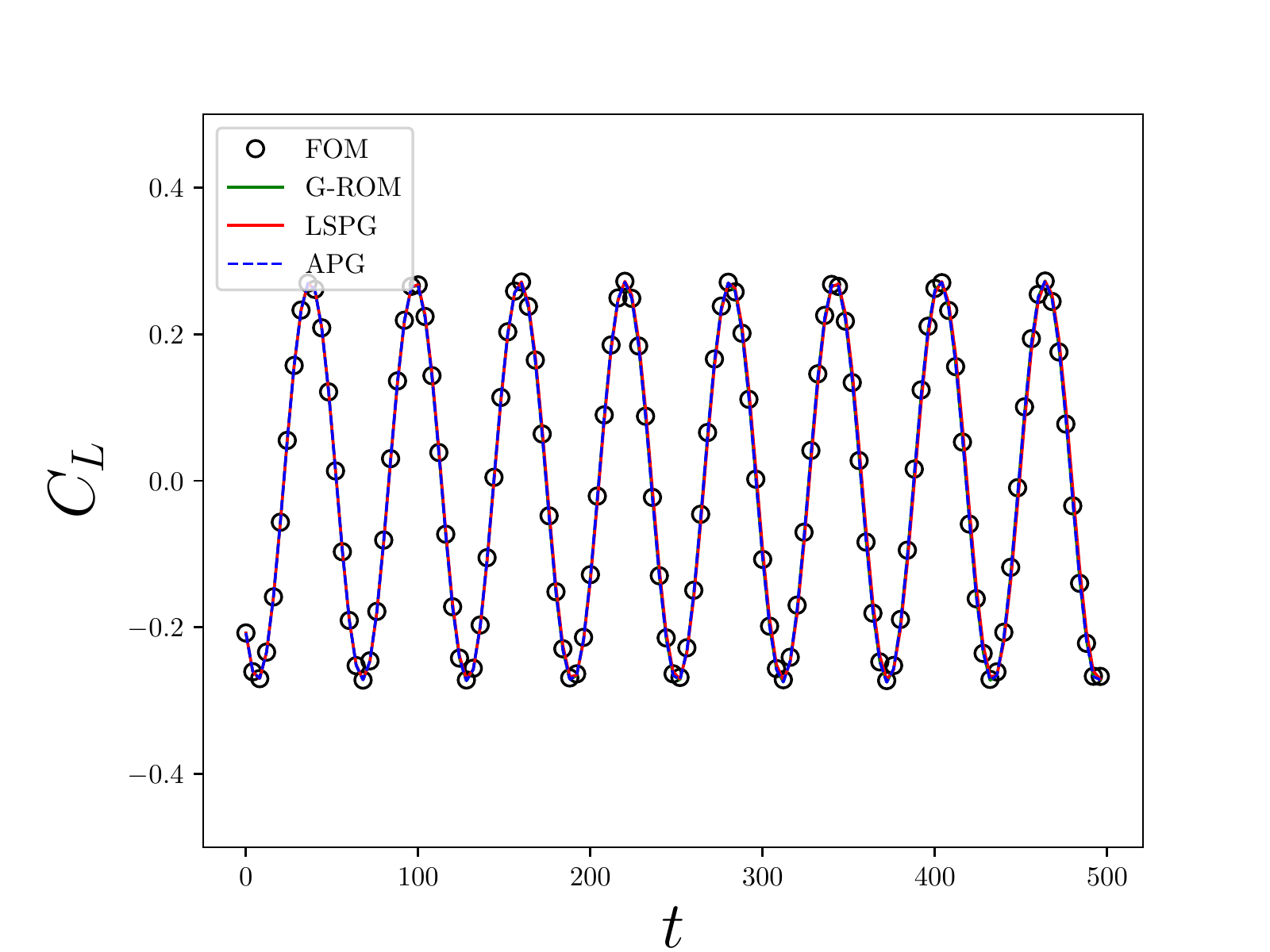}
\caption{Lift coefficient as a function of time for Basis \# 6.}
\end{subfigure}

\end{center}
\caption{Results for hyper-reduced ROMs at Re$=100$.}
\label{fig:re100_qdeim_pareto}
\end{figure}

\section{Conclusion}\label{sec:conclude}
This work introduced the Adjoint Petrov--Galerkin method for non-linear model reduction. Derived from the variational multiscale method and Mori-Zwanzig formalism, the Adjoint Petrov--Galerkin method is a Petrov--Galerkin projection technique with a non-linear time-varying test basis. The method is designed to be applied at the semi-discrete level, i.e., after spatial discretization of a partial differential equation, and is compatible with both implicit and explicit time integration schemes. The method displays commonalities with the adjoint-stabilization method used in the finite element community as well as the Least-Squares Petrov--Galerkin approach used in non-linear model-order reduction. Theoretical error analysis was presented that showed conditions under which the Adjoint Petrov--Galerkin ROM may have lower \textit{a priori} error bounds than the Galerkin ROM. The theoretical cost of the Adjoint Petrov--Galerkin method was considered for both explicit and implicit schemes, where it was shown to be approximately twice that of the Galerkin method. In the case of implicit time integration schemes, the Adjoint Petrov--Galerkin ROM was shown to be capable of being more efficient than Least-Squares Petrov--Galerkin when the non-linear system is solved via Jacobian-Free Newton-Krylov methods. 

Numerical experiments with the Adjoint Petrov--Galerkin, Galerkin, and Least-Squares Petrov--Galerkin method were presented for the Sod shock tube problem and viscous compressible flow over a cylinder parameterized by the Reynolds number. In all examples, the Adjoint Petrov--Galerkin method provided more accurate predictions than the Galerkin method for a fixed basis dimension. Improvements over the Least-Squares Petrov--Galerkin method were observed in most cases. In particular, the Adjoint Petrov--Galerkin method was shown to provide relatively accurate predictions for the cylinder flow at Reynolds numbers outside of the training set used to construct the POD basis. The Galerkin method, with both an equivalent and an enriched trial space, failed to produce accurate results in these cases. Additionally, numerical evidence showed a correlation between the spectral radius of the reduced Jacobian and the optimal value of the stabilization parameter appearing in the Adjoint Petrov--Galerkin method.

When augmented with hyper-reduction, the Adjoint Petrov--Galerkin ROM was shown to be capable of producing accurate predictions within the POD training set with computational speedups up to 5000 times compared to the full-order models. This speed-up is a result of hyper-reduction of the right-hand side, as well as the ability to use explicit time integration schemes at large time-steps. A study of the Pareto front for simulation error versus relative wall time showed that, for the compressible cylinder problem, the Adjoint Petrov--Galerkin ROM is competitive with the Galerkin ROM, and more efficient than the LSPG ROM for the problems considered.

\section{Acknowledgements}\label{sec:acknowledge}
The authors acknowledge support from the US Air Force Office of Scientific Research through the Center of Excellence Grant FA9550-17-1-0195 (Tech. Monitors: Mitat Birkan \& Fariba Fahroo) and the project LES Modeling of Non-local effects using Statistical Coarse-graining (Tech. Monitors: Jean-Luc Cambier \& Fariba Fahroo). E. Parish acknowledges an appointment to the Sandia National Laboratories John von Neumann fellowship. This paper describes objective technical results and analysis. Any subjective
views or opinions that might be expressed in the paper do not necessarily
represent the views of the U.S. Department of Energy or the United States
Government Sandia National Laboratories is a multimission laboratory managed
and operated by National Technology and Engineering Solutions of Sandia, LLC.,
a wholly owned subsidiary of Honeywell International, Inc., for the U.S.
Department of Energy’s National Nuclear Security Administration under contract
DE-NA-0003525.

\begin{appendices}

\section{Hyper-reduction for the Adjoint Petrov--Galerkin Reduced-Order Model}\label{appendix:hyper}

In the numerical solution of non-linear dynamical systems, the evaluation of the non-linear right-hand side term usually accounts for a large portion (if not the majority) of the computational cost. Equation~\ref{eq:GROM_modal} shows that standard projection-based ROMs are incapable of reducing this cost, as the evaluation of $\mathbf{R}(\vcoarsevec \acoarsevec)$ still scales with the number of degrees of freedom $N$. If this issue is not addressed, and there is no reduction in temporal dimensionality, the ROM will typically be \emph{more} expensive than the FOM due to additional matrix-vector products from projection onto the reduced-order space. Techniques for overcoming this bottleneck are typically referred to as hyper-reduction methods.  The thesis of hyper-reduction methods is that, instead of computing the entire right-hand side vector, only a few entries are calculated. The missing entries can either be ignored (as done in collocation methods) or reconstructed (as done in the discrete interpolation and gappy POD methods). This section outlines the Gappy POD method, selection of the sampling points through QR factorization, and the algorithm for the hyper-reduced Adjoint Petrov--Galerkin ROM.

\subsection{Gappy POD}
The gappy POD method~\cite{everson_sirovich_gappy} seeks to find an approximation for $\mathbf{R}(\cdot)$ that evaluates the right-hand side term at a reduced number of spatial points $r \ll N$. This is achieved through the construction of a trial space for the right-hand side and least-squares reconstruction of a sampled signal. The offline steps required in the Gappy POD method are given in Algorithm~\ref{alg:gappy_hyper_offline}. 

\begin{algorithm}
\caption{Algorithm for the offline steps required for hyper-reduction via gappy POD.}
\label{alg:gappy_hyper_offline}
Offline Steps:
\begin{enumerate}
    \item Compute the full-order solution, storing $n_t$ time snapshots in the following matrices,
    \begin{alignat*}{2}
        &\text{Right-hand side snapshots:} \; &&\mathbf{F} = [\mathbf{R}(\ufullvec_1) \quad \mathbf{R}(\ufullvec_2) \quad ... \quad \mathbf{R}(\ufullvec_{n_t})] \in \mathbb{R}^{N \times n_t}
    \end{alignat*}    
    \item Compute the right-hand side POD basis $\RHSbasisvec \in \mathbb{R}^{N \times r}$, from $\mathbf{F}$.
    \item Compute the sampling point matrix $\sampvec = [\mathbf{e}_{p_1} \quad \mathbf{e}_{p2} \quad ... \mathbf{e}_{p_r}] \in \mathbb{R}^{N_{p} \times N}$, where $\mathbf{e}_i$ is the $i$th cannonical unit vector and $N_p$ is the number of sampling points. 
    \item Compute the stencil matrix $\sampvec_s = [\mathbf{e}_{s_1} \quad \mathbf{e}_{s2} \quad ... \mathbf{e}_{s_r}] \in \mathbb{R}^{N_{s} \times N}$, where $\mathbf{e}_i$ is the $i$th cannonical unit vector and $N_s$ is the number of stencil points required to reconstruct the residual at the sample points. Note $N_s \ge N_p$. 

    \item Compute the least-squares reconstruction matrix: $\big[ \mathbf{P}^T \RHSbasisvec \big]^{+}$, where the superscript $+$ denotes the pseudo-inverse. 
 Note that this matrix corresponds to the solution of the least-squares problem for a gappy signal $\mathbf{f} \in \mathbb{R}^N$:
$$ \afullvec_{\mathbf{f}} = \underset{\mathbf{b} \in \mathbb{R}^r}{\text{argmin}} || \sampvec^T \RHSbasisvec \mathbf{b}   - \sampvec^T \mathbf{f}||,$$
which has the solution,
$$\afullvec_{\mathbf{f}} = \bigg[ \mathbf{P}^T \RHSbasisvec \bigg]^{+} \mathbf{f}.$$
\end{enumerate}
\end{algorithm}

Note that, because the gappy POD approximation of the right-hand side only samples the right-hand side term at $N_p$ spatial points, the cost of the ROM no longer scales with the full-order degrees of freedom $N$, but instead with the number of POD basis modes $K$ and the number of sample points $N_p$. Thus, the cost of evaluating the full right-hand side may be drastically reduced at the price of storing another snapshot matrix $\mathbf{F}$ and computing another POD basis $\RHSbasisvec$ in the offline stage of computation. Furthermore, the product $\vcoarsevec^T \RHSbasisvec \big[ \sampvec^T \RHSbasisvec \big]^{+}$ may be precomputed during the offline stage if both $\sampvec$ and $\RHSbasisvec$ remain static throughout the simulation. This results in an relatively small $K\times N_p$ matrix. As such, an increase in offline computational cost may produce a significant decrease in online computational cost.

\subsection{Selection of Sampling Points}
Step 3 in Algorithm~\ref{alg:gappy_hyper_offline} requires the construction of the sampling point matrix, for which several methods exist. In the discrete interpolation method proposed by Chaturantabut and Sorensen~\cite{deim}, the sample points are selected inductively from the basis $\RHSbasisvec$, based on an error measure between the basis vectors and approximations of the basis vectors via interpolation. The method proposed by Drmac and Gugercin~\cite{qdeim_drmac} leverages the rank-revealing QR factorization to compute $\sampvec$. Dynamic updates of the $\RHSbasisvec$ and $\sampvec$ via periodic sampling of the full-order right-hand side term is even possible via the methods developed by Peherstorfer and Willcox~\cite{adeim_peherstorfer}. 

The 2D compressible cylinder simulations presented in this manuscript uses a modified version of the rank-revealing QR factorization proposed in~\cite{qdeim_drmac} to obtain the sampling points. The modifications are added to enhance the stability and accuracy of the hyper-reduced ROM within the discontinuous Galerkin method. Algorithm~\ref{alg:qdeim} outlines the steps used in this manuscript to compute the sampling points.

\begin{algorithm}
\caption{Algorithm for QR-factorization-based selection of sampling matrix.}
\label{alg:qdeim}
Input: Right-hand side POD basis $\RHSbasisvec \in \mathbb{R}^{N \times r}$ \;
\newline
Output: Sampling matrix $\sampvec$ \;
\newline
Steps:
\begin{enumerate}
    \item Transpose the right-hand side POD basis, $\RHSbasisvec' = \RHSbasisvec^T$
    \item Compute the rank-revealing QR factorization of $\RHSbasisvec'$, generating a permutation matrix $\Gamma \in \mathbb{R}^{N \times N}$, unitary $\mathbf{Q} \in \mathbb{R}^{r \times r}$, and orthonormal $\mathbf{R} \in \mathbb{R}^{r \times N}$ such that
    \begin{equation*}
        \RHSbasisvec' \Gamma = \mathbf{QR}
    \end{equation*}
    Details on computing rank-revealing QR decompositions can be found in~\cite{qr_decomp_gu}, though many math libraries include optimized routines for this operation.
    \item From the permutation matrix $\Gamma$, select the first $r$ columns to form the interpolation point matrix $\sampvec$. 
    \item When applied to systems of equations, sampling approaches that select only specific indices of the residual can be inaccurate due to the fact that, at a given cell, the residual of all unknowns at that cell may not be calculated. This issue is further exacerbated in the discontinuous Galerkin method, where each cell has a number of quadrature points. As such, we perform the additional step: 
    \begin{enumerate}
    \item Augment the sampling point matrix, $\sampvec$, with additional columns such that all unknowns are computed at the mesh cells selected by Step 3. In the present context, these additional unknowns correspond to each conserved variable and quadrature point in the selected cells. This step leads to $N_p > r$.
    \end{enumerate}
\end{enumerate}

\end{algorithm}

\subsection{Hyper-Reduction of the Adjoint Petrov--Galerkin ROM}
Lastly, the online steps required for an explicit Euler update to the Adjoint Petrov--Galerkin ROM with Gappy POD hyper-reduction is provided in Algorithm~\ref{alg:alg_ag_hyper}. It is worth noting that Step 3 in Algorithm~\ref{alg:alg_ag_hyper} requires one to reconstruct the right-hand side at the stencil points. Hyper-reduction via a standard collocation method, which provides no means to reconstruct the right-hand side, is thus not compatable with the Adjoint Petrov--Galerkin ROM. 
\begin{algorithm}
\caption{Algorithm for an explicit Euler update for the Adjoint Petrov--Galerkin ROM with gappy POD hyper-reduction.}
\label{alg:alg_ag_hyper}
Input: $\acoarsevec^n$ \;
\newline
Output: $\acoarsevec^{n+1}$\;
\newline
Online steps at time-step $n$:
\begin{enumerate}
\item Compute the state at the stencil points: $\ucoarsevec_s^n = \sampvec_s^T \vcoarsevec \acoarsevec^n,$ with $\ucoarsevec_s^n \in \mathbb{R}^{N_s}$ 
\item Compute the generalized coordinates to the right-hand side evaluation via,
\begin{equation*}\label{eq:RHSapprox}
   \mathbf{a}_{\mathbf{R}}^n = \big[ \sampvec^T \RHSbasisvec \big]^{+} \sampvec^T \mathbf{R}(\ucoarsevec_s^n),
\end{equation*}
with $\mathbf{a}_{\mathbf{R}}^n \in \mathbb{R}^r$.
Note that the product $\sampvec^T \mathbf{R}(\ucoarsevec_s^n)$ requires computing $\mathbf{R}(\ucoarsevec_s^n)$ \emph{only at the sample points}, as given by the unit vectors stored in $\sampvec$.
\item Reconstruct the right-hand side at the stencil points: $\overline{\mathbf{R}_s(\ucoarsevec_s^n)} = \mathbf{P}_s^T \RHSbasisvec \mathbf{a}_{\mathbf{R}}^n$
\item Compute the orthogonal projection of the approximated right-hand side at the stencil points:
$$\Pifine \overline{\mathbf{R}_s(\ucoarsevec^n)} =  \overline{\mathbf{R}_s(\ucoarsevec^n)} -  \vcoarsevec \vcoarsevec^T \overline{ \mathbf{R}_s(\ucoarsevec^n)}$$
\item Compute the generalized coordintes for the action of the Jacobian on $\Pifine \overline{\mathbf{R}_s(\ucoarsevec^n)}$ at the sample points using either finite difference or exact linearization. For finite difference:
    \begin{equation*}
    \mathbf{a}_{\mathbf{J}}^n \approx \frac{1}{\epsilon} \big[ \sampvec^T \RHSbasisvec \big]^{+} \sampvec^T  \Big[ \mathbf{R}_s\big(\ucoarsevec^n_s + \epsilon \Pifine  \overline{\mathbf{R}_s}(\ucoarsevec^n_s) \big) - {\mathbf{R}_s(\ucoarsevec^n_s  )} \Big], 
    \end{equation*}
    with $\mathbf{a}_{\mathbf{J}}^n \in \mathbb{R}^r$. Note $\epsilon$ is a small constant value, usually  $\sim \MC{O}(10^{-5})$.
\item Compute the combined sampled right-hand side:  $\vcoarsevec^T \RHSbasisvec \bigg[ \mathbf{a}_{\mathbf{R}}^n + \tau \mathbf{a}_{\mathbf{J}}^n \bigg]$
\item Update the state: $\acoarsevec^{n+1} = \acoarsevec^n + \Delta t \vcoarsevec^T \RHSbasisvec \bigg[ \mathbf{a}_{\mathbf{R}}^n + \tau \mathbf{a}_{\mathbf{J}}^n \bigg] $
\end{enumerate}

\end{algorithm}

\section{POD Basis Construction}\label{appendix:basisconstruction}

For the 1D Euler case detailed in this manuscript, the procedure for constructing separate POD bases for each conserved variables is as follows:

\begin{enumerate}
\item Run the full-order model for $t \in (0,1)$ at a time-step of $\Delta t = 0.0005$. The state vector is saved at every other time step to create 1000 state snapshots.
\item Collect the snapshots for each state into three state snapshot matrices:
$$
\mathbf{S}_{\rho}  = \begin{bmatrix}
\mathbf{\rho}_1 & \mathbf{\rho}_2 & \hdots & \rho_{1000}
\end{bmatrix},
\mathbf{S}_{\rho u}  = \begin{bmatrix}
\mathbf{\rho u}_1 & \mathbf{\rho u}_2 & \hdots & \mathbf{\rho u}_{1000}
\end{bmatrix},
\mathbf{S}_{\rho E}  = \begin{bmatrix}
\mathbf{ \rho E}_1 & \mathbf{ \rho E}_2 & \hdots & \mathbf{\rho E}_{1000}
\end{bmatrix},
$$
where $\mathbf{\rho}_i, \mathbf{\rho u}_i, \mathbf{\rho E}_i \in \mathbb{R}^{1000}$.
\item Compute the singular-value decomposition (SVD) of each snapshot matrix, e.g. for $\mathbf{S}_{\rho}$,
$$\mathbf{S}_{\rho} \mathrel{\overset{\makebox[0pt]{\mbox{\normalfont\tiny\sffamily SVD}}}{=}} \mathbf{V}_{\rho} {\Sigma}_{\rho} \mathbf{U}^T_{\rho}.$$
The columns of $\mathbf{V}_{\rho}$ and $\mathbf{U}_{\rho}$ are the left and right singular vectors of $\mathbf{S}_{\rho}$, respectively. ${\Sigma}_{\rho}$ is a diagonal matrix of the singular values of $\mathbf{S}_{\rho}$. The columns of $\mathbf{V}_{\rho}$ form a basis for the solution space of $\rho_i$. 

\textit{Remarks}
\begin{enumerate}
\item In this example, a separate basis is computed for each conserved quantity. It is also possible to construct a global basis by stacking $\mathbf{S}_{\rho}$, $\mathbf{S}_{\rho \mathbf{u}}$, and $\mathbf{S}_{\rho \mathbf{E}}$ into one snapshot matrix and computing one ``global" SVD.
\end{enumerate}
\item Decompose each basis into bases for the resolved and unresolved scales by selecting the first $K$ columns and last $1000-K$ columns, respectively, e.g.
$$\mathbf{V}_\rho = \begin{bmatrix} \vcoarsevec_{\rho} & ; & \vfinevec_{\rho} \end{bmatrix},$$
where $\vcoarsevec_{\rho} \in \mathbb{R}^{1000 \times K}$ and $\vfinevec \in \mathbb{R}^{1000 \times 1000 - K}.$  \\
\textit{Remarks}
\begin{enumerate}
\item Each basis vector is orthogonal to the others, hence the coarse and fine-scales are orthogonal.
\item In this example, we have selected 1000 snapshots such that the column space of $\vfullvec$ spans $\Vfull$. In general, this is not the case. As the APG method requires no processing of the fine-scale basis functions, however, this is not an issue.
\end{enumerate}
\item Construct a global coarse-scale basis,
$$\vcoarsevec = \begin{bmatrix}
\vcoarsevec_{\rho} & \mathbf{0} & \mathbf{0} \\
\mathbf{0} & \vcoarsevec_{\rho u} & \mathbf{0} \\
\mathbf{0} & \mathbf{0} & \vcoarsevec_{\mathbf{\rho E}} \\
\end{bmatrix}.$$
\end{enumerate}

\section{Algorithms for the Galerkin and LSPG ROMs}\label{appendix:algorithms}
Section~\ref{sec:cost} presented an analysis on the computational cost of the Adjoint Petrov--Galerkin ROM. This appendix presents similar algorithms and FLOP counts for the Galerkin and LSPG ROMs. The following algorithms and FLOP counts are reported:
\begin{enumerate}
    \item An explicit Euler update to the Galerkin ROM (Algorithm~\ref{alg:alg_g_exp}, Table~\ref{tab:alg_g_exp}).
    \item An implicit Euler update to the Galerkin ROM using Newton's method with Gaussian elimination (Algorithm~\ref{alg:alg_g_imp}, Table~\ref{tab:alg_g_imp}).
    \item An implicit Euler update to the Least-Squares Petrov--Galerkin ROM using the Gauss-Newton method with Gaussian elimination (Algorithm~\ref{alg:alg_LSPG}, Table~\ref{tab:alg_LSPG}).
\end{enumerate}

\begin{algorithm}[h!]
\caption{Algorithm for an explicit Euler update for the Galerkin ROM.}
\label{alg:alg_g_exp}
Input: $\acoarsevec^n$ \;
\newline
Output: $\acoarsevec^{n+1}$\;
\newline
Steps:
\begin{enumerate}
\item Compute the state from the generalized coordinates, $\ucoarsevec^n =\vcoarsevec \acoarsevec^{n}$
\item Compute the right-hand side from the state, $\mathbf{R}(\ucoarsevec^n)$
\item Project the right-hand side, $\vcoarsevec^T\mathbf{R}(\ucoarsevec^n)$
\item Update the state $\acoarsevec^{n+1} = \acoarsevec^n + \Delta t \vcoarsevec^T\mathbf{R}(\ucoarsevec^n)$
\end{enumerate}
\end{algorithm}

\begin{table}[h!]
\begin{tabular}{p{7cm} p{8cm}}
\hline
Step in Algorithm~\ref{alg:alg_g_exp} & Approximate FLOPs \\
\hline
1    & $2NK - N$  \\
2    & $\omega N$  \\
3    & $2NK - K$  \\
4    & $2K $  \\
\hline
Total    & $4 N K + (\omega-1) N + K$ \\
\hline
\end{tabular}
\caption{Approximate floating-point operations for an explicit Euler update to the Galerkin method reported in Algorithm~\ref{alg:alg_g_exp}.}
\label{tab:alg_g_exp}
\end{table}

\begin{algorithm}[h!]
\caption{Algorithm for an implicit Euler update for the Galerkin ROM using Newton's Method with Gaussian Elimination}
\label{alg:alg_g_imp}
Input: $\acoarsevec^n$, residual tolerance $\xi$ \;
\newline
Output: $\acoarsevec^{n+1}$\;
\newline
Steps:
\begin{enumerate}
\item Set initial guess, $\acoarsevec_k$
\item  Loop while $\mathbf{r}^k > \xi$
\begin{enumerate}
    \item Compute the state from the generalized coordinates, $\ucoarsevec_k = \vcoarsevec \acoarsevec_k$
    \item Compute the right-hand side from the full state, $\mathbf{R}(\ucoarsevec_k)$
    \item Project the right-hand side, $\vcoarsevec^T \mathbf{R}(\ucoarsevec_k)$
    \item Compute the Galerkin residual, $\mathbf{r}_G(\acoarsevec_k) = \acoarsevec_k - \acoarsevec^n - \Delta t \vcoarsevec^T \mathbf{R}(\vcoarsevec \acoarsevec_k)$
    \item Compute the residual Jacobian, $\frac{\partial \mathbf{r}(\acoarsevec_k)}{\partial \acoarsevec_k}$
    \item Solve the linear system via Gaussian Elimination: $\frac{\partial \mathbf{r}(\acoarsevec_k)}{\partial \acoarsevec_k} \Delta  \acoarsevec = - \mathbf{r}(\acoarsevec_k)$
    \item Update the state: $\acoarsevec_{k+1} = \acoarsevec_k + \Delta \acoarsevec$
    \item $k = k + 1$
\end{enumerate}
\item Set final state, $\acoarsevec^{n+1} = \acoarsevec_k$
\end{enumerate}
\end{algorithm}

\begin{table}[h!]
\centering
\begin{tabular}{p{7cm} p{8cm}}
\hline
Step in Algorithm~\ref{alg:alg_apg_imp}& Approximate FLOPs \\
\hline
2a    & $2 N K - N $ \\
2b    & $ \omega N $ \\
2c    & $2 N K - K$  \\
2d    & $ 3K $  \\
2e    & $ 4 N K^2 +  (\omega - 1) N K + 2K^2 $ \\
2f    & $ K^3 $ \\
2g    & $ K $ \\
\hline
Total & $ (\omega - 1)N + 3K + (\omega + 3)NK + 2K^2 + 4NK^2 + K^3 $
\end{tabular}
\caption{Approximate floating-point operations for one Newton iteration for the implicit Euler update to the Galerkin method reported in Algorithm~\ref{alg:alg_g_imp}.}
\label{tab:alg_g_imp}
\end{table}

\clearpage

\begin{algorithm}[H]
\caption{Algorithm for an implicit Euler update for the LSPG ROM using a Gauss-Newton method with Gaussian Elimination}
\label{alg:alg_LSPG}
Input: $\acoarsevec^n$, residual tolerance $\xi$ \;
\newline
Output: $\acoarsevec^{n+1}$\;
\newline
Steps:
\begin{enumerate}
\item Set initial guess, $\acoarsevec_k$
\item  Loop while $\mathbf{r}_k > \xi$
\begin{enumerate}
    \item Compute the state from the generalized coordinates, $\ucoarsevec_k =\vcoarsevec \acoarsevec_{k}$
    \item Compute the right-hand side from the full state, $\mathbf{R}(\ucoarsevec_k)$
    \item Compute the residual, $\mathbf{r} (\ucoarsevec_k) = \ucoarsevec_k - \ucoarsevec^n - \Delta t \mathbf{R}(\ucoarsevec_k)$
    \item Compute the test basis, $\mathbf{W}_k = \frac{\partial \mathbf{r}(\ucoarsevec_k)}{\partial \ucoarsevec_k} \vcoarsevec = \frac{\partial \mathbf{r}(\ucoarsevec_k)}{\partial \acoarsevec_k} $
    \item Compute the product, $\Wcoarsevec_k^T \Wcoarsevec_k$
    \item Project the residual onto the test space, $\Wcoarsevec^T \mathbf{r}(\acoarsevec_k)$
    \item Solve $\Wcoarsevec^T \Wcoarsevec \Delta \acoarsevec = - \Wcoarsevec^T \mathbf{r}(\acoarsevec_k) $ for $\Delta \acoarsevec$ via Gaussian elimination
    \item Update solution, $\acoarsevec_{k+1} = \acoarsevec_k + \Delta \acoarsevec$
    \item k = k + 1
\end{enumerate}
\item Set final state, $\acoarsevec^{n+1} = \acoarsevec_k$ 
\end{enumerate}
\end{algorithm}

\begin{table}[H]
\begin{tabular}{p{7cm} p{8cm}}
\hline
Step in Algorithm~\ref{alg:alg_LSPG}& Approximate FLOPs \\
\hline
2a    & $ 2NK - N $ \\
2b    & $ \omega N $  \\
2c    & $ 3N $  \\
2d    & $ (\omega + 2)NK + 2NK^2  $  \\
2e    & $ 2NK^2 - K^2 $ \\
2f    & $ 2NK - K $ \\
2g    & $ K^3 $ \\
2h    & $ K $ \\
\hline
Newton Iteration Total & $ (\omega + 2)N + (\omega + 6) NK - K^2 + 4NK^2 + K^3 $ \\

\end{tabular}
\caption{Approximate floating-point operations for one Newton iteration for the implicit Euler update to the LSPG method reported in Algorithm~\ref{alg:alg_LSPG}.}
\label{tab:alg_LSPG}
\end{table}


\end{appendices}

\clearpage

\bibliographystyle{aiaa}
\bibliography{refs}{}

\begin{thebibliography}{10}
\newcommand{\enquote}[1]{``#1''}

\bibitem{balanced_truncation_moore}
{Moore}, B., \enquote{Principal component analysis in linear systems:
  Controllability, observability, and model reduction,} {\em IEEE Transactions
  on Automatic Control\/}, Vol.~26, No.~1, February 1981, pp.~17--32.

\bibitem{balanced_truncation_roberts}
Mullis, C.~T. and Roberts, R.~A., \enquote{Synthesis of minimum roundoff noise
  fixed point digital filters,} {\em IEEE Transactions on Circuits and
  Systems\/}, Vol.~23, 1976, pp.~551--562.

\bibitem{krylov_rom}
Pillage, L.~T., Huang, X., and Rohrer, R.~A., \enquote{Asymptotic Waveform
  Evaluation for Timing Analysis,} {\em Proceedings of the 26th ACM/IEEE Design
  Automation Conference\/}, DAC '89, ACM, New York, NY, USA, 1989, pp.
  634--637.

\bibitem{Hesthaven2016}
Hesthaven, J.~S., Rozza, G., and Stamm, B., {\em Certified Reduced Basis
  Methods for Parametrized Partial Differential Equations\/}, Springer
  International Publishing, Cham, 2016.

\bibitem{chatterjee_pod_intro}
Chatterjee, A., \enquote{An introduction to the proper orthogonal
  decomposition,} {\em Current Science\/}, Vol.~78, No.~7, 2000, pp.~808--817.

\bibitem{kerschen_mech_pod}
Kerschen, G., Golinval, J.-C., Vakakis, A.~F., and Bergman, L.~A., \enquote{The
  method of proper orthogonal decomposition for dynamical characterization and
  order reduction in mechanical systems: An overview,} {\em Nonlinear
  Dynamics\/}, Vol.~41, 2005, pp.~147--169.

\bibitem{padhi_neural_net_pod}
Padhi, R. and Balakrishnan, S., \enquote{Proper orthogonal decomposition based
  optimal neurocontrol synthesis of a chemical reactor process using
  approximate dynamic programming,} {\em Neural Networks\/}, Vol.~16, 2003,
  pp.~719--728.

\bibitem{cao_meteorology_pod}
Cao, Y., Zhu, J., Navon, I., and Zhengdong, L., \enquote{A reduced-order
  approach to four-dimensional variational data assimilation using proper
  orthogonal decomposition,} {\em Int. J. Numer. Meth. Fluids\/}, Vol.~53,
  2007, pp.~1571--1583.

\bibitem{rowley_pod_energyproj}
Rowley, C.~W., Colonius, T., and Murray, R.~M., \enquote{{Model reduction for
  compressible flows using POD and Galerkin projection},} {\em Physica D:
  Nonlinear Phenomena\/}, Vol.~189, No. 1-2, 2004, pp.~115--129.

\bibitem{huang_combustion_roms}
Huang, C., Duraisamy, K., and Merkle, C.~L., \enquote{{Challenges in Reduced
  Order Modeling of Reacing Flow},} {\em 2018 Joint Propulsion Conference, AIAA
  Propulsion and Energy Forum\/}, Cincinnati, Ohio, 2018.

\bibitem{Kalashnikova_sand2014}
Kalashnikova, I., Arunajatesan, S., Barone, M.~F., van Bloemen~Waanders, B.~G.,
  and Fike, J.~A., \enquote{{Reduced Order Modeling for Prediction and Control
  of Large-Scale Systems},} Report SAND2014-4693, Sandia, May 2014.

\bibitem{sirovich_symmetry_trans}
Sirovich, L., \enquote{Turbulence and the dynamics of coherent structures. II:
  Symmetries and transformations,} {\em Quarterly of Applied Mathematics\/},
  Vol.~45, No.~3, 1987, pp.~573--582.

\bibitem{carlberg_hadaptation}
Carlberg, K., \enquote{Adaptive h-refinement for reduced-order models,} {\em
  International Journal for Numerical Methods in Engineering\/}, Vol.~102,
  No.~5, 2015, pp.~1192--1210.

\bibitem{adeim_peherstorfer}
Peherstorfer, B. and Willcox, K., \enquote{Online adaptive model reduction for
  nonlinear systems via low-rank updates,} {\em J. Sci. Comput.\/}, Vol.~37,
  2015, pp.~A2123--A2150.

\bibitem{l1}
Abgrall, R. and Crisovan, R., \enquote{Model reduction using L1-norm
  minimization as an application to nonlinear hyperbolic problems,} {\em
  International Journal for Numerical Methods in Fluids\/}, Vol.~87, No.~12,
  2018, pp.~628--651.

\bibitem{basis_rotation}
Balajewicz, M., Tezaur, I., and Dowell, E., \enquote{Minimal subspace rotation
  on the Stiefel manifold for stabilization and enhancement of projection-based
  reduced order models for the compressible Navier–Stokes equations,} {\em
  Journal of Computational Physics\/}, Vol.~321, 2016, pp.~224 -- 241.

\bibitem{bui_resmin_steady}
Bui-Thanh, T., Willcox, K., and Ghattas, O., \enquote{Model Reduction for
  Large-Scale Systems with High-Dimensional Parametric Input Space,} {\em SIAM
  Journal on Scientific Computing\/}, Vol.~30, No.~6, 2008, pp.~3270--3288.

\bibitem{bui_unsteady}
Bui-Thanh, T., Willcox, K., and Ghattas, O., \enquote{Parametric Reduced-Order
  Models for Probabilistic Analysis of Unsteady Aerodynamic Applications,} {\em
  AIAA Journal\/}, Vol.~46, No.~10, 2008, pp.~2520--2529.

\bibitem{rovas_thesis}
Rovas, D.~V., {\em {Reduced-basis output bound methods for parametrized partial
  differential equations}\/}, Ph.D. thesis, Massachusetts Institute of
  Technology, 2003.

\bibitem{carlberg_thesis}
Carlberg, K., {\em Model Reduction of Nonlinear Mechanical Systems via Pptimal
  Projection and Tensor Approximation\/}, Ph.D. thesis, Stanford University,
  2011.

\bibitem{bui_thesis}
Bui-Thanh, T., {\em {Model-constrained optimization methods for reduction of
  parameterized large-scale systems}\/}, Ph.D. thesis, Massachusetts Institute
  of Technology, 2007.

\bibitem{carlberg_lspg}
Carlberg, K., Bou-Mosleh, C., and Farhat, C., \enquote{Efficient non-linear
  model reduction via a least-squares Petrov-Galerkin projection and
  compressive tensor approximations,} {\em Int. J. Numer. Methods Eng.\/},
  Vol.~86, 2011, pp.~155--181.

\bibitem{carlberg_lspg_v_galerkin}
Carlberg, K., Barone, M., and Antil, H., \enquote{{Galerkin v. least-squares
  Petrov-Galerkin projection in nonlinear model reduction},} {\em Journal of
  Computational Physics\/}, Vol.~330, 2017, pp.~693--734.

\bibitem{carlberg_gnat}
Carlberg, K., Farhat, C., Cortial, J., and Amsallem, D., \enquote{The GNAT
  method for nonlinear model reduction: Effective implementation and
  application to computational fluid dynamics and turbulent flows,} {\em
  Journal of Computational Physics\/}, Vol.~242, 2013, pp.~623--647.

\bibitem{huang_scitech19}
Huang, C., Duraisamy, K., and Merkle, C., {\em Investigations and Improvement
  of Robustness of Reduced-Order Models of Reacting Flow\/}.

\bibitem{carlberg_conservative_rom}
Carlberg, K., Choi, Y., and Sargsyan, S., \enquote{Conservative model reduction
  for finite-volume models,} {\em Journal of Computational Physics\/},
  Vol.~371, 2018, pp.~280--314.

\bibitem{Wang_ROM_thesis}
Wang, Z., {\em {Reduced-Order Modeling of Complex Engineering and Geophysical
  Flows: Analysis and Computations}\/}, Ph.D. thesis, Virginia Polytechnic
  Institute and State University, 2012.

\bibitem{aubry_mixlength_pod}
Aubry, N., Holmes, P., Lumley, J.~L., and Stone, E., \enquote{The dynamics of
  coherent structures in the wall region of a turbulent boundary layer,} {\em
  J. Fluid Mech.\/}, Vol.~192, 1988, pp.~115--173.

\bibitem{Ullmann_smag}
Ullmann, S. and Lang, J., \enquote{{A POD-Galerkin reduced model with updated
  coefficients for Smagorinsky LES},} {\em J.C.F. Pereira, A. Sequeira (Eds.),
  V European Conference on Computational Fluid Dynamics, ECCOMAS CFD\/},
  Lisbon, Portugal, 2010.

\bibitem{wang_smag}
Wang, Z., Akhtar, I., Borggaard, J., and Iliescu, T., \enquote{{Two-level
  discretizations of nonlinear closure models for proper orthogonal
  decomposition},} {\em Journal of Computational Physics\/}, Vol.~230, No.~1,
  2011, pp.~126--146.

\bibitem{smag_ROM}
Noack, B., Papas, P., and Monkewitz, P., \enquote{{Low-dimensional Galerkin
  model of a laminar shear-layer},} Report 2002-01, \'Ecole Polytechnique
  F\'ed\'erale de Lausanne, 2002.

\bibitem{san_iliescu_geostrophic}
San, O. and Iliescu, T., \enquote{{A stabilized proper orthogonal decomposition
  reduced-order model for large scale quasigeostrophic ocean circulation},}
  {\em Adv Comput Math\/}, Vol.~41, No.~1, 2015, pp.~1289--1319.

\bibitem{Bergmann_pod_vms}
Bergmann, M., Bruneaua, C., and Iollo, A., \enquote{{Enablers for robust POD
  models},} {\em Journal of Computational Physics\/}, Vol.~228, No.~1, 2009,
  pp.~516--538.

\bibitem{Stabile2019}
Stabile, G., Ballarin, F., Zuccarino, G., and Rozza, G., \enquote{A reduced
  order variational multiscale approach for turbulent flows,} {\em Advances in
  Computational Mathematics\/}, Jun 2019.

\bibitem{iliescu_pod_eddyviscosity}
San, O. and Iliescu, T., \enquote{{Proper Orthogonal Decomposition Closure
  Models for Fluid Flows: Burgers Equation},} {\em Comput. Methods Appl. Mech.
  Engrg\/}, Vol.~5, No.~3, 2014, pp.~217--237.

\bibitem{iliescu_vms_pod_ns}
Iliescu, T. and Wang, Z., \enquote{{Variational Multiscale Proper Orthogonal
  Decomposition: Navier-Stokes Equations},} {\em Numerical Methods for Partial
  Differential Equations\/}, Vol.~30, No.~2, 2014, pp.~641--663.

\bibitem{iliescu_ciazzo_residual_rom}
Caiazzo, A., Iliescu, T., John, V., and Schyschlowa, S., \enquote{{A numerical
  investigation of velocity-pressure reduced models for incompressible flows},}
  {\em Journal of Computational Physics\/}, Vol.~259, No.~1, 2014,
  pp.~598--616.

\bibitem{MoriTransport}
Mori, H., \enquote{{Transport, collective motion, and Brownian motion},} {\em
  Prog. Theoret. Phys\/}, Vol.~33, No.~3, 1965, pp.~423--455.

\bibitem{ZwanzigLangevin}
Zwanzig, R., \enquote{{Nonlinear generalized Langevin equations},} {\em Journal
  of Statistical Physics\/}, 1973, pp.~215--220.

\bibitem{ChorinOptimalPrediction}
Chorin, A.~J., Hald, O., and Kupferman, R., \enquote{{Optimal prediction and
  the Mori-Zwanzig representation of irreversible processes},} {\em Proc. Natl
  Acad. Sci.\/}, Vol.~97, No. (doi:10.1073/pnas.97.7.2968), 2000,
  pp.~2968--2973.

\bibitem{ChorinOptimalPredictionMemory}
Chorin, A.~J., Hald, O.~H., and Kupferman, R., \enquote{Optimal prediction with
  memory,} {\em Physica D: Nonlinear Phenomena\/}, Vol.~166, No.~3, 2002,
  pp.~239 -- 257.

\bibitem{Chorin_book}
Chorin, A. and Hald, O., {\em {Stochastic Tools in Mathematics and Science}\/},
  Springer-Verlag, 2005.

\bibitem{ProblemReduction}
Chorin, A. and Stinis, P., \enquote{{Problem reduction, renormalization, and
  memory},} {\em Commun. Appl. Math. Comput. Sci.\/}, 2006, pp.~239--257.

\bibitem{stinisEuler}
Hald, O.~H. and Stinis, P., \enquote{{Optimal prediction and the rate of decay
  for solutions of the Euler equations in two and three dimensions},} {\em
  Proceedings of the National Academy of Sciences\/}, Vol.~104, No.~16, 2007,
  pp.~6527--6532.

\bibitem{stinisHighOrderEuler}
Stinis, P., \enquote{{Higher order Mori-Zwanzig models for the Euler
  equations},} {\em Multiscale Model. Simul.\/}, Vol.~6, No.~3, 2007,
  pp.~741--760.

\bibitem{Stinis-rMZ}
Stinis, P., \enquote{{Renormalized reduced models for singular PDEs},} {\em
  Commun. Appl. Math. Comput. Sci.\/}, Vol.~8, No.~1, 2013, pp.~39--66.

\bibitem{stinis_finitememory}
Stinis, P., \enquote{{Mori-Zwanzig reduced models for uncertainty
  quantification I: Parametric uncertainty},} {\em arXiv:1211.4285\/}, 2012.

\bibitem{PriceMZ}
Price, J. and Stinis, P., \enquote{Renormalized Reduced Order Models with
  Memory for Long Time Prediction,} {\em Multiscale Modeling \& Simulation\/},
  Vol.~17, No.~1, 2019, pp.~68--91.

\bibitem{PriceMZ2}
Price, J. and Stinis, P., \enquote{{Renormalization and blow-up for the 3D
  Euler equations},} {\em arXiv:1805.08766v1\/}, 2018.

\bibitem{parishAIAA2016}
Parish, E. and Duraisamy, K., \enquote{{Reduced Order Modeling of Turbulent
  Flows Using Statistical Coarse-graining},} {\em 46th AIAA Fluid Dynamics
  Conference, AIAA AVIATION Forum\/}, Washington, D.C., June 2016.

\bibitem{parishMZ1}
Parish, E.~J. and Duraisamy, K., \enquote{{Non-Markovian Closure Models for
  Large Eddy Simulation using the Mori-Zwanzig Formalism},} {\em Physical
  Review Fluids\/}, Vol.~2, No.~1, 2017.

\bibitem{parish_dtau}
Parish, E.~J. and Duraisamy, K., \enquote{{A dynamic subgrid scale model for
  Large Eddy Simulations based on the Mori-Zwanzig formalism},} {\em Journal of
  Computational Physics\/}, Vol.~349, 2017, pp.~154--175.

\bibitem{GouasmiMZ1}
Gouasmi, A., Parish, E.~J., and Duraisamy, K., \enquote{{A priori estimation of
  memory effects in reduced-order models of nonlinear systems using the
  Mori-Zwanzig formalism},} {\em Proceedings of The Royal Society A\/},
  Vol.~473, No. 20170385, 2017.

\bibitem{parishVMS}
Parish, E.~J. and Duraisamy, K., \enquote{{A Unified Framework for Multiscale
  Modeling Using Mori-Zwanzig and the Variational Multiscale Method},} {\em
  arXiv Preprint\/}, 2018.

\bibitem{everson_sirovich_gappy}
Everson, R. and Sirovich, L., \enquote{{Karhunen-Lo`eve procedure for gappy
  data},} {\em Journal of the Optical Society of America A\/}, Vol.~12, 1995,
  pp.~1657--1644.

\bibitem{eim}
Barrault, M., Maday, Y., Nguyen, N.~C., and Patera, A.~T., \enquote{An
  ‘empirical interpolation’ method: application to efficient reduced-basis
  discretization of partial differential equations,} {\em C. R. Acad. Sci.
  Paris\/}, Vol.~339, No.~9, 2004, pp.~667--672.

\bibitem{deim}
Chaturantabut, S. and Sorensen, D.~C., \enquote{{Nonlinear Model Reduction via
  Discrete Empirical Interpolation},} {\em SIAM J. Sci. Comput.\/}, Vol.~32,
  No.~5, 2010, pp.~2737--2764.

\bibitem{brooks_supg}
Brooks, A.~N. and Hughes, T.~J., \enquote{Streamline upwind/Petrov-Galerkin
  formulations for convection dominated flows with particular emphasis on the
  incompressible Navier-Stokes equations,} {\em Comput. Method Appl. M.\/},
  Vol.~32, 1982, pp.~199--259.

\bibitem{hughes_petrovgalerkin}
Hughes, T.~J. and Tezduyar, T.~E., \enquote{Finite element methods for
  first-order hyperbolic systems with particular emphasis on the compressible
  Euler equations,} {\em Comput Method Appl. M.\/}, Vol.~45, 1984,
  pp.~217--284.

\bibitem{hughes_GLS}
Hughes, T.~J., Franca, L.~P., and Hulbert, G.~M., \enquote{A new finite element
  formulation for computational fluid dynamics: VIII. The
  galerkin/least-squares method for advective-diffusive equations,} {\em
  Computer Methods in Applied Mechanics and Engineering\/}, Vol.~73, No.~2,
  1989, pp.~173 -- 189.

\bibitem{hughes0}
Hughes, T.~J., \enquote{{Multiscale phenomena: Green's functions, the
  Dirichlet-to-Neumann formulation, subgridscale models, bubbles and the
  origins of stabilized methods},} {\em Comput. Methods Appl. Mech. Eng.\/},
  Vol.~127, 1995, pp.~387--401.

\bibitem{ZwanzigBook}
Zwanzig, R., {\em {Nonequilibrium Statistical Mechanics}\/}, Oxford University
  Press, 2001.

\bibitem{Koopman}
Koopman, B., \enquote{{Hamiltonian systems and transformations in Hilbert
  space},} {\em Proceedings of the National Academy of Sciences of the US A\/},
  Vol.~17, No.~5, 1931, pp.~315--318.

\bibitem{BarberThesis}
Barber, J.~L., {\em {Application of Optimal Prediction to Molecular
  Dynamics}\/}, Ph.D. thesis, University of California, Berkeley, CA, 2004.

\bibitem{gmres}
Saad, Y. and Schultz, M., \enquote{GMRES: A Generalized Minimal Residual
  Algorithm for Solving Nonsymmetric Linear Systems,} {\em SIAM Journal on
  Scientific and Statistical Computing\/}, Vol.~7, No.~3, 1986, pp.~856--869.

\bibitem{nlls_JacobianFree}
Xu, W., Zheng, N., and Hayami, K., \enquote{Jacobian-Free Implicit
  Inner-Iteration Preconditioner for Nonlinear Least Squares Problems,} {\em
  Journal of Scientific Computing\/}, Vol.~68, No.~3, Sep 2016, pp.~1055--1081.

\bibitem{sod}
Sod, G.~A., \enquote{{Survey of several finite difference methods for systems
  of nonlinear hyperbolic conservation laws},} {\em Journal of Computational
  Physics\/}, Vol.~27, 1978, pp.~1--31.

\bibitem{roescheme}
Roe, P., \enquote{{Approximate Riemann solvers, parameter vectors and
  difference schemes},} {\em Journal of Computational Physics\/}, Vol.~43,
  1981, pp.~357--372.

\bibitem{SSP_RK3}
Gottlieb, S., Shu, C.-W., and Tadmor, E., \enquote{{Strong Stability-Preserving
  High-Order Time Discretization Methods},} {\em SIAM Review\/}, Vol.~43,
  No.~1, 2001, pp.~89--112.

\bibitem{BR1}
Bassi, F. and Rebay, S., \enquote{{A High Order Accurate Discontinuous Finite
  Element Method for the Numerical Solution of the Compressible Navier-Stokes
  Equations},} {\em Journal of Computational Physics\/}, Vol.~131, 1997,
  pp.~267--279.

\bibitem{scipy_leastsquares_dogbox}
Voglis, C. and Lagaris, I.~E., \enquote{{A Rectangular Trust Region Dogleg
  Approach for Unconstrained and Bound Constrained Nonlinear Optimization},}
  {\em WSEAS International Conference on Applied Mathematics.\/}, 2004.

\bibitem{qdeim_drmac}
Drmac, Z. and Gugercin, S., \enquote{A new selection operator for the discrete
  empirical Interpolation method---Improved a priori error bound and
  extensions,} {\em J. Sci. Comput.\/}, Vol.~38, 2016, pp.~A631--A648.

\bibitem{qr_decomp_gu}
Gu, M. and Eisenstat, S.~C., \enquote{Efficient algorithms for computing a
  strong rank-revealing QR factorization,} {\em J. Sci. Comput.\/}, Vol.~17,
  1996, pp.~848--869.

\end{thebibliography}

\end{document}